\documentclass[twocolumn,10pt]{asme2ej}

\usepackage{epsfig} 
\usepackage{amsmath,amssymb,url}
\usepackage{graphicx,subfigure,float}
\usepackage{epstopdf}
\usepackage{color}
\usepackage{enumitem}

\newcommand{\norm}[1]{\ensuremath{\left\| #1 \right\|}}

\newcommand{\refeqn}[1]{(\ref{eqn:#1})}

\newcommand{\trs}[1]{\mathrm{tr}\ensuremath{[#1]}}

\newcommand{\SO}{\ensuremath{\mathsf{SO(3)}}}

\newcommand{\SE}{\ensuremath{\mathsf{SE(3)}}}
\renewcommand{\Re}{\ensuremath{\mathbb{R}}}
\newcommand{\Sph}{\ensuremath{\mathsf{S}}}

\newcommand{\g}{\ensuremath{\mathfrak{g}}}
\usepackage{dsfont}
\usepackage[hidelinks]{hyperref}
\newcommand{\xb}{\mathbf{x}}
\newcommand{\diag}{\mathop{\mathrm{diag}}}
\newcommand{\sat}{\mathop{\mathrm{sat}}}
\newcommand{\satr}{\mathop{\mathrm{sat}}}

\newtheorem{prop}{Proposition}

\newcommand{\Xb}{\mathbf{X}}
\newcommand{\Pb}{\mathbf{P}}
\newcommand{\Mb}{\mathbf{M}}
\newcommand{\Nb}{\mathbf{N}}

\newcommand{\Bb}{\mathbf{B}}

\newcommand{\Gb}{\mathbf{G}}
\newcommand{\zb}{\mathbf{z}}

\title{Stabilization of a Rigid Body Payload with Multiple Cooperative Quadrotors}

\author{Farhad A. Goodarzi
    \affiliation{
	Post Doctoral Fellow\\
	George Washington University\\
	Washington, DC 20052\\
    Email: fgoodarzi@gwu.edu
    }	
}

\author{Taeyoung Lee
    \affiliation{Associate Professor\\
	George Washington University\\
	Washington, DC 20052\\
        Email: tylee@gwu.edu
    }
}

\begin{document}

\maketitle   

\begin{abstract}
{\it This paper presents the full dynamics and control of arbitrary number of quadrotor unmanned aerial vehicles (UAV) transporting a rigid body. The rigid body is connected to the quadrotors via flexible cables where each flexible cable is modeled as a system of arbitrary number of serially-connected links. It is shown that a coordinate-free form of equations of motion can be derived for the complete model without any simplicity assumptions that commonly appear in other literature, according to Lagrangian mechanics on a manifold. A geometric nonlinear controller is presented to transport the rigid body to a fixed desired position while aligning all of the links along the vertical direction. A rigorous mathematical stability proof is given and the desirable features of the proposed controller are illustrated by numerical examples and experimental results.
}
\end{abstract}

\begin{nomenclature}
\entry{$i=1,\cdots,m$} {m number of quadrotors}
\entry{$n_{i}$} {Number of links in the $i$-th cable}
\entry{$\vec{e}_{1},\vec{e}_{2},\vec{e}_{3}\in\Re^3$} {Inertial frame}
\entry{$\vec{b}_{1},\vec{b}_{2},\vec{b}_{3}\in\Re^3$} {Body-fixed frame of the payload}
\entry{$\vec{b}_{1_i},\vec{b}_{2_i},\vec{b}_{3_i}\in\Re^3$} {Body-fixed frame of the $i$-th quadrotor}
\entry{$m_{i}\in\Re$} {Mass of the $i$-th quadrotor}
\entry{$m_{0}\in\Re$} {Mass of the payload}
\entry{$J_i\in\Re^{3\times 3}$} {Inertia matrix of the $i$-th quadrotor}
\entry{$R_i\in\SO$} {Attitude of the $i$-th quadrotor}
\entry{$\Omega_i\in\Re^3$} {Angular velocity of the $i$-th quadrotor}
\entry{$x_i\in\Re^3$} {Position of the $i$-th quadrotor}
\entry{$v_i\in\Re^3$} {Velocity of the $i$-th quadrotor}
\entry{$g\in\Re$} {Gravitational acceleration}
\entry{$\SO$} {Special Orthogonal  group}
\entry{$\SE$} {Special Euclidean group}
\end{nomenclature}

\section{Introduction}

Quadrotor UAVs are being considered for various missions such as Mars surface exploration, search and rescue, and particularly payload transportation. There are various applications for aerial load transportation such as usage in construction, military operations, emergency response, or delivering packages. Load transportation with UAVs can be performed using a cable or by grasping the payload~\cite{kumargrasp2013,kim2013grasping}. There are several limitations for grasping a payload with UAVs such as in situations where the landing area is inaccessible, or, it transporting a heavy/bulky object by multiple quadrotors.

\begin{figure}[h]
\centerline{
\includegraphics[width=0.9\columnwidth]{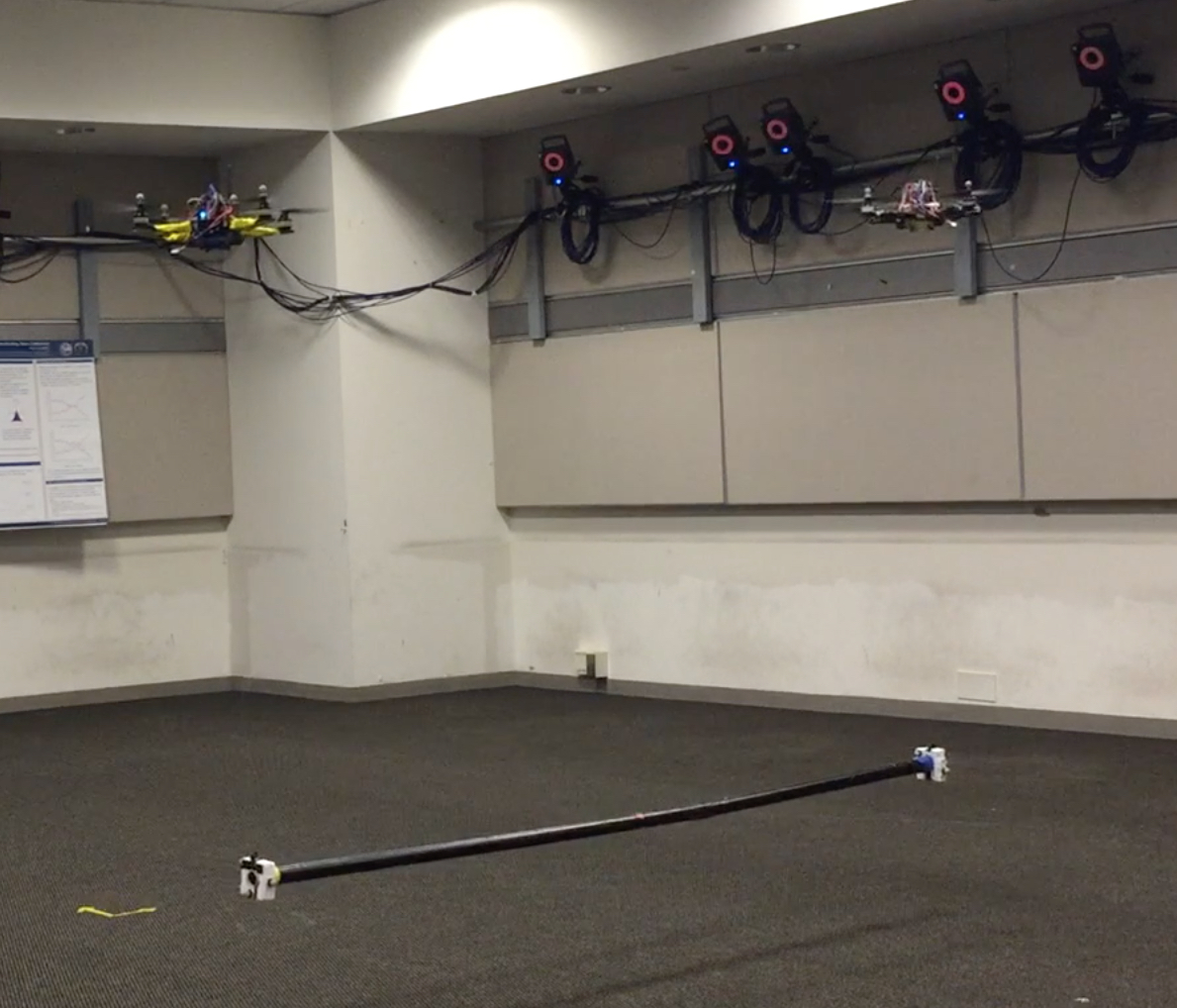}
}
\caption{Two quadrotors stabilizing a payload cooperatively.}\label{fig:fig00}
\end{figure}

Load transportation with the cable-suspended load has been studied traditionally for a helicopter~\cite{CicKanJAHS95,BerPICRA09} or for small unmanned aerial vehicles such as quadrotor UAVs~\cite{PalCruIRAM12,MicFinAR11,MazKonJIRS10}. 

In most of the prior works, the dynamics of aerial transportation has been simplified due to the inherent dynamic complexities. For example, it is assumed that the dynamics of the payload is considered completely decoupled from quadrotors, and the effects of the payload and the cable are regarded as arbitrary external forces and moments exerted to the quadrotors~\cite{ ZamStaJDSMC08, SchMurIICRA12, PalFieIICRA12}, thereby making it challenging to suppress the swinging motion of the payload actively, particularly for agile aerial transportations.

Recently, the coupled dynamics of the payload or cable has been explicitly incorporated into control system design~\cite{LeeSrePICDC13}. In particular, a complete model of a quadrotor transporting a payload modeled as a point mass, connected via a flexible cable is presented, where the cable is modeled as serially connected links to represent the deformation of the cable~\cite{gooddaewontaeyoungacc14,IJCAS2015}. In these studies the payload simplified and considered as a point mass without the attitude and the moment of inertia. In another study, multiple quadrotors transporting  a rigid body payload has been studied~\cite{LeeMultipleRigid14}, but it is assumed that the cables connecting the rigid body payload and quadrotors are always taut. These assumptions and simplifications in the dynamics of the system reduce the stability of the controlled system, particularly in rapid and aggressive load transportation where the motion of the cable and payload is excited nontrivially.

The other critical issue in designing controllers for quadrotors is that they are mostly based on local coordinates. Some aggressive maneuvers are demonstrated at~\cite{MelMicIJRR12} based on Euler angles. However they involve complicated expressions for trigonometric functions, and they exhibit singularities in representing quadrotor attitudes, thereby restricting their ability to achieve complex rotational maneuvers significantly. A quaternion-based feedback controller for attitude stabilization was shown in~\cite{TayMcGITCSTI06}. By considering the Coriolis and gyroscopic torques explicitly, this controller guarantees exponential stability. Quaternions do not have singularities but, as the three-sphere double-covers the special orthogonal group, one attitude may be represented by two antipodal points on the three-sphere. This ambiguity should be carefully resolved in quaternion-based attitude control systems, otherwise they may exhibit unwinding, where a rigid body unnecessarily rotates through a large angle even if the initial attitude error is small~\cite{BhaBerSCL00}. To avoid these, an additional mechanism to lift attitude onto the unit-quaternion space is introduced~\cite{MaySanITAC11}.

Recently, the dynamics of a quadrotor UAV is globally expressed on the special Euclidean group, $\SE$, and nonlinear control systems are developed to track outputs of several flight modes~\cite{farhadASME15}. There are also several studies using the estimations for dynamical objects developed on the special Euclidean group~\cite{Izadi2013DSCC,CDC151}. Several aggressive maneuvers of a quadrotor UAV are demonstrated based on a hybrid control architecture, and a nonlinear robust control system is also considered in~\cite{LeeLeoPACC12,Farhad2013}. As they are directly developed on the special Euclidean/Orthogonal group, complexities, singularities, and ambiguities associated with minimal attitude representations or quaternions are completely avoided~\cite{ChaSanICSM11,CDC152}. The proposed control system is particularly useful for rapid and safe payload transportation in complex terrain, where the position of the payload should be controlled concurrently while suppressing the deformation of the cables.

Comparing with the prior work of the authors in~\cite{farhadacc15,LeeLeoAJC13,farhadthesisphd} and other existing studies, this paper is the first study considering a complete model which includes a rigid body payload with attitude, arbitrary number of quadrotors, and flexible cables. A rigorous mathematical stability analysis is presented, and numerical and experimental validations in presence of uncertainties and disturbances are provided. More explicitly, we present the complete dynamic model of an arbitrary number of quadrotors transporting a rigid body where each quadrotor is connected to the rigid body via a flexible cable. Each flexible cable is modeled as an arbitrary number of serially connected links, and it is valid for various masses and lengths. A coordinate free form of equations of motion is derived according to Lagrange mechanics on a nonlinear manifold for the full dynamic model. These sets of equations of motion are presented in a complete and organized manner without any restrictive assumption or simplification.

Another contribution of this study is designing a control system to stabilize the rigid body at desired position. Geometric nonlinear controllers presented and generalized for the presented model. More explicitly, we show that the rigid body payload is asymptotically transported into a desired location, while aligning all of the links along the vertical direction corresponding to a hanging equilibrium. This paper presents a rigorous Lyapunov stability analysis for the proposed controller to establish stability properties without any timescale separation assumptions or singular perturbation, and a nonlinear integral control term is designed to guarantee robustness against unstructured uncertainties in both rotational and translational dynamics.

In short, new contributions and the unique features of the dynamics model and control system proposed in this paper compared with other studies are as follows: (i) it is developed for the full dynamic model of arbitrary number of multiple quadrotor UAVs on $\SE$ transporting a rigid body connected via flexible cables, including the coupling effects between the translational dynamics and the rotational dynamics on a nonlinear manifold, (ii) the control systems are developed directly on the nonlinear configuration manifold in a coordinate-free fashion.  Thus, singularities of local parameterization are completely avoided to generate agile maneuvers in a uniform way, (iii) a rigorous Lyapunov analysis is presented to establish stability properties without any timescale separation assumption, and (iv) an integral control term is proposed to guarantee asymptotical convergence of tracking error variables in the presence of uncertainties, (v) the proposed algorithm is validated with experiments for payload transportation with multiple cooperative quadrotor UAVs. A rigorous and complete mathematical analysis for multiple quadrotor UAVs transporting a payload on $\SE$ with experimental validations for payload transportation maneuvers is unprecedented. 

This paper is organized as follows. A dynamic model is presented and the problem is formulated at Section II. Control systems are constructed at Sections III and IV, which are followed by numerical examples in Section V. Finally, experimental results are presented in Section VI.

\section{Problem Formulation}

Consider a rigid body with the mass $m_{0}\in\Re$ and the moment of inertia $J_{0}\in\Re^{3\times 3}$, being transported with $n$ arbitrary number of quadrotors as shown in Figure~\ref{fig:fig1}. The location of the mass center of the rigid body is denoted by $x_{0}\in\Re^{3}$, and its attitude is given by $R_{0}\in\SO$, where the special orthogonal group is given by $\SO=\{R\in\Re^{3\times 3} \mid R^{T}R=I,\det(R)=1\}$. We choose an inertial frame $\{\vec{e}_{1},\vec{e}_{2},\vec{e}_{3}\}$ and body fixed frame $\{\vec{b}_{1},\vec{b}_{2},\vec{b}_{3}\}$ attached to the payload. We also consider a body fixed frame attached to the $i$-th quadrotor $\{\vec{b}_{1_i},\vec{b}_{2_i},\vec{b}_{3_i}\}$. In the inertial frame, the third axes $\vec{e}_{3}$ points downward along the gravity and the other axes are chosen to form an orthonormal frame. 

\begin{figure}[h]
\centerline{
	\setlength{\unitlength}{0.09\columnwidth}\scriptsize
\begin{picture}(5,8.8)(0,0)
\put(0,0){\includegraphics[width=.6\columnwidth]{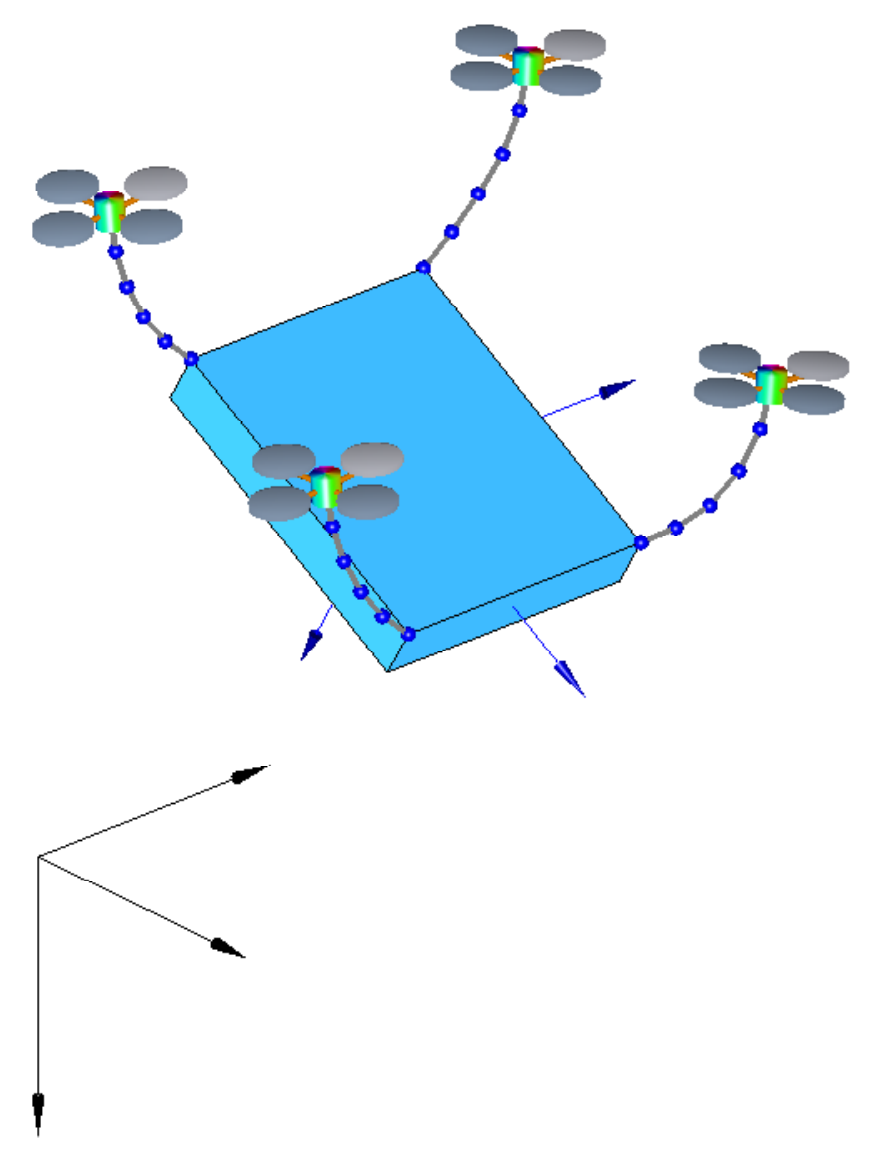}}
\put(3.6,5.3){\shortstack[c]{$m_{0}$}}
\put(3.6,4.9){\shortstack[c]{$J_{0}$}}
\put(-0.4,7.3){\shortstack[c]{$m_{1}$}}
\put(0.3,6.2){\shortstack[c]{$m_{1j}$}}
\put(-0.4,6.9){\shortstack[c]{$J_{1}$}}
\put(4.9,8.5){\shortstack[c]{$m_{2}$}}
\put(3.9,7.1){\shortstack[c]{$m_{2j}$}}
\put(4.9,8.1){\shortstack[c]{$J_{2}$}}
\put(6.7,6.0){\shortstack[c]{$m_{3}$}}
\put(5.9,5.0){\shortstack[c]{$m_{3j}$}}
\put(6.7,5.6){\shortstack[c]{$J_{3}$}}
\put(2.2,2.8){\shortstack[c]{$e_{1}$}}
\put(1.9,1.3){\shortstack[c]{$e_{2}$}}
\put(0.1,-0.1){\shortstack[c]{$e_{3}$}}
\put(4.9,5.9){\shortstack[c]{$b_{1}$}}
\put(4.6,3.2){\shortstack[c]{$b_{2}$}}
\put(2.0,3.4){\shortstack[c]{$b_{3}$}}
\end{picture}
}
\caption{Quadrotor UAVs with a rigid body payload. Cables are modeled as a serial connection of an arbitrary number of links (only 4 quadrotors with 5 links in each cable are illustrated).}\label{fig:fig1}
\end{figure}
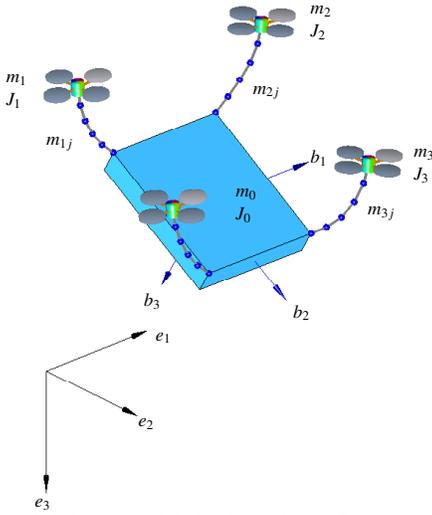

The mass and the moment of inertia of the $i$-th quadrotor are denoted by $m_{i}\in\Re$ and $J_{i}\in\Re^{3\times 3}$ respectively. The cable connecting each quadrotor to the rigid body is modeled as an arbitrary numbers of links for each quadrotor with varying masses and lengths. The direction of the $j$-th link of the $i$-th quadrotor, measured outward from the quadrotor toward the payload is defined by the unit vector $q_{ij}\in\Sph^2$, where $\Sph^2=\{q\in\Re^{3}\mid\|q\|=1\}$, where the mass and length of that link is denoted with $m_{ij}$ and $l_{ij}$ respectively. The number of links in the cable connected to the $i$-th quadrotor is defined as $n_{i}$. The configuration manifold for this system is given by $\SO\times \Re^{3}\times (\SO^{n})\times (\Sph^2)^{\sum_{i=1}^{n}n_{i}}$. 

The $i$-th quadrotor can generate a thrust force of $-f_{i}R_{i}e_{3}\in\Re^{3}$ with respect to the inertial frame, where $f_{i}\in\Re$ is the total thrust magnitude of the $i$-th quadrotor. It also generates a moment $M_{i}\in\Re^{3}$ with respect to its body-fixed frame. Also we define $\Delta_{x_i}$ and $\Delta_{R_i}\in\Re^3$ as fixed disturbances applied to the $i$-th quadrotor's translational and rotational dynamics respectively. It is also assumed that an upper bound of the infinite norm of the uncertainty is known
\begin{align}\label{eqn:disturbancecond}
\|\Delta_{x}\|_{\infty}\leq \delta,
\end{align}
for a positive constant $\delta$. Throughout this paper, the two norm of a matrix $A$ is denoted by $\|A\|$. The standard dot product is denoted by $x\cdot y=x^{T}y$ for any $x,y\in\Re^{3}$.
\subsection{Lagrangian}
The kinematics equations for the links, payload, and quadrotors are given by
\begin{gather}
\dot{q}_{ij}=\omega_{ij}\times q_{ij}=\hat{\omega}_{ij}q_{ij},\\
\dot{R}_{0}=R_{0}\hat{\Omega}_{0},\\
\dot{R}_{i}=R_{i}\hat{\Omega}_{i},
\end{gather}
where $\omega_{ij}\in\Re^{3}$ is the angular velocity of the $j$-th link in the $i$-th cable satisfying $q_{ij}\cdot\omega_{ij}=0$. Also, $\Omega_{0}\in\Re^{3}$ is the angular velocity of the payload and $\Omega_{i}\in\Re^{3}$ is the angular velocity of the $i$-th quadrotor, expressed with respect to the corresponding body fixed frame. The \textit{hat map} $\hat\cdot:\Re^3\rightarrow\SO$ is defined by the condition that $\hat x y=x\times y$ for all $x,y\in\Re^3$. More explicitly, for a vector $a=[a_1,a_2,a_3]^{T}\in\Re^3$, the matrix $\hat a$ is given by
\begin{align}
    \hat a = \begin{bmatrix} 0 & -a_3 & a_2\\
                                a_3 & 0 & -a_1\\
                                -a_2 & a_1 & 0 \end{bmatrix}\label{eqn:hat}.
\end{align}
This identifies the Lie algebra $\SO$ with $\Re^3$ using the vector cross product in $\Re^3$. The inverse of the hat map is denoted by the \textit{vee} map, $\vee:\SO\rightarrow\Re^3$. The position of the $i$-th quadrotor is given by
\begin{align}\label{eqn:xi}
x_{i}=x_{0}+R_{0}\rho_{i}-\sum_{a=1}^{n_{i}}{l_{ia}q_{ia}},
\end{align}
where $\rho_{i}\in\Re^{3}$ is the vector from the center of mass of the rigid body to the point that $i$-th cable is connected to the rigid body. Similarly the position of the $j$-th link in the cable connecting the $i$-th quadrotor to the rigid body is given by
\begin{align}\label{eqn:xij}
x_{ij}=x_{0}+R_{0}\rho_{i}-\sum_{a=j+1}^{n_{i}}{l_{ia}q_{ia}}.
\end{align}

We derive equations of motion according to Lagrangian mechanics. Total kinetic energy of the system is given by
\begin{align}
T=&\frac{1}{2}m_{0}\|\dot{x}_{0}\|^{2}+\sum_{i=1}^{n}\sum_{j=1}^{n_{i}}{\frac{1}{2}m_{ij}\|\dot{x}_{ij}\|^{2}}+\frac{1}{2}\sum_{i=1}^{n}{m_{i}\|\dot{x}_{i}\|^{2}}\nonumber\\
&+\frac{1}{2}\sum_{i=1}^{n}{\Omega_{i}\cdot J_{i}\Omega_{i}}+\frac{1}{2}\Omega_{0}\cdot J_{0}\Omega_{0}.
\end{align}
The gravitational potential energy is given by
\begin{align}
V=-m_{0}ge_{3}\cdot x_{0}-\sum_{i=1}^{n}{m_{i}ge_{3}}\cdot x_{i}-\sum_{i=1}^{n}\sum_{j=1}^{n_{i}}{m_{ij}ge_{3}}\cdot x_{ij},
\end{align}
where it is assumed that the unit-vector $e_{3}$ points downward along the gravitational acceleration as shown at Figure \ref{fig:fig1}. The corresponding Lagrangian of the system is $L=T-V$.

\subsection{Euler-Lagrange equations}
Coordinate-free form of Lagrangian mechanics on the two-sphere $\Sph^2$ and the special orthogonal group $\SO$ for various multi-body systems has been studied in~\cite{Lee08,LeeLeoIJNME08}. The key idea is representing the infinitesimal variation of $R_i\in\SO$ in terms of the exponential map
\begin{align}
\delta R_{i} = \frac{d}{d\epsilon}\bigg|_{\epsilon = 0} R_{i}\exp(\epsilon \hat\eta_{i}) = R_{i}\hat\eta_{i},\label{eqn:delR}
\end{align}
for $\eta_{i}\in\Re^3$. The corresponding variation of the angular velocity is given by $\delta\Omega_{i}=\dot\eta_{i}+\Omega_{i}\times\eta_{i}$. Similarly, the infinitesimal variation of $q_{ij}\in\Sph^2$ is given by
\begin{align}
\delta q_{ij} = \xi_{ij}\times q_{ij},\label{eqn:delqi}
\end{align}
for $\xi_{ij}\in\Re^3$ satisfying $\xi_{ij}\cdot q_{ij}=0$. Using these, we obtain the following Euler-Lagrange equations.

\begin{prop}\label{prop:FDM}
The equations of motion for the proposed payload transportation system are as follows
\begin{gather}\label{eqn:EOM}
M_{T}\ddot{x}_{0}-\sum_{i=1}^{n}\sum_{j=1}^{n_{i}}{M_{0ij}l_{ij}\ddot{q}_{ij}}-\sum_{i=1}^{n}{M_{iT}R_{0}\hat{\rho}_{i}\dot{\Omega}_{0}} \nonumber\\
=M_{T}ge_{3}+\sum_{i=1}^{n}{(-f_{i}R_{i}e_{3}+\Delta_{x_i})}-\sum_{i=1}^{n}{M_{iT}R_{0}\hat{\Omega}_{0}^2 \rho_{i}},\label{eqn:EOMM1}\\
\bar{J}_{0}\dot{\Omega}_{0}+\sum_{i=1}^{n}{M_{iT}\hat{\rho}_{i}R_{0}^{T}\ddot{x}_{0}}-\sum_{i=1}^{n}\sum_{j=1}^{n_{i}}{M_{0ij}l_{ij}\hat{\rho}_{i}R_{0}^{T}\ddot{q}_{ij}} \nonumber\\
=\sum_{i=1}^{n}{\hat{\rho}_{i}R_{0}^{T}(-f_{i}R_{i}e_{3}+M_{iT}ge_{3}+\Delta_{x_i})}-\hat{\Omega}_{0}\bar{J}_{0}\Omega_{0},\label{eqn:EOMM2}\\
\sum_{k=1}^{n_{i}}{M_{0ij}l_{ik}\hat{q}_{ij}^{2}\ddot{q}_{ik}}-M_{0ij}\hat{q}_{ij}^{2}\ddot{x}_{0}+M_{0ij}\hat{q}_{ij}^{2}R_{0}\hat{\rho}_{i}\dot{\Omega}_{0} \nonumber\\
=M_{0ij}\hat{q}_{ij}^{2}R_{0}\hat{\Omega}_{0}^{2}\rho_{i}-\hat{q}_{ij}^{2}(M_{0ij}ge_{3}-f_{i}R_{i}e_{3}+\Delta_{x_i}),\label{eqn:EOMM3}\\
J_{i}\Omega_{i}+\Omega_{i}\times J_{i}\Omega_{i}=M_{i}+\Delta_{R_{i}}\label{eqn:EOMM4}.
\end{gather}
Here the total mass $M_{T}$ of the system and the mass of the $i$-th quadrotor and its flexible cable $M_{iT}$ are defined as
\begin{gather}
M_{T}=m_{0}+\sum_{i=1}^{n}M_{iT},\; M_{iT}=\sum_{j=1}^{n_{i}}{m_{ij}}+m_{i},\label{eqn:def1}
\end{gather}
and the constants related to the mass of links are given as
\begin{align}
M_{0ij}&=m_{i}+\sum_{a=1}^{j-1}{m_{ia}}\label{eqn:def3},
\end{align}
The equations of motion can be rearranged in a matrix form as follow
\begin{align}
\Nb\dot{\Xb}=\Pb
\end{align}
where the state vector $X\in\Re^{D_{X}}$ with $D_{X}=6+3\sum_{i=1}^{n}n_{i}$ is given by
\begin{align}
\Xb=[\dot{x}_{0},\; {\Omega}_{0},\; \dot{q}_{1j},\; \dot{q}_{2j},\; \cdots,\; \dot{q}_{nj}]^{T},
\end{align}
and matrix $\Nb\in\Re^{D_{X}\times D_{X}}$ is defined as
\begin{align}\label{eqn:EOM11}
\Nb=\begin{bmatrix}
M_{T}I_{3}&\Nb_{x_{0}\Omega_{0}}&\Nb_{x_{0}1}&\Nb_{x_{0}2}&\cdots&\Nb_{x_{0}n}\\
\Nb_{\Omega_{0} x_{0}}&\bar{J}_{0}&\Nb_{\Omega_{0}1}&\Nb_{\Omega_{0}2}&\cdots&\Nb_{\Omega_{0}n}\\
\Nb_{1 x_{0}}&\Nb_{1\Omega_{0}}&\Nb_{qq1}&0&\cdots&0\\
\Nb_{2 x_{0}}&\Nb_{2\Omega_{0}}&0&\Nb_{qq2}&\cdots&0\\
\vdots&\vdots&\vdots&\vdots&\vdots&\vdots\\
\Nb_{n x_{0}}&\Nb_{n\Omega_{0}}&0&0&\cdots&\Nb_{qqn}
\end{bmatrix},
\end{align}
where the sub-matrices are defined as
\begin{gather}
\Nb_{x_{0}\Omega_{0}}=-\sum_{i=1}^{n}{M_{iT}R_{0}\hat{\rho}_{i}};\; \Nb_{\Omega_{0} x_{0}}=\Mb_{x_{0}\Omega_{0}}^{T},\nonumber\\
\Nb_{x_{0}i}=-[M_{0i1}l_{i1}{I}_{3},\; M_{0i2}l_{i2}{I}_{3},\; \cdots,\;M_{0in_{i}}l_{in_{i}}{I}_{3}],\nonumber\\
\Nb_{\Omega_{0}i}=-[M_{0i1}l_{i1}\hat{\rho}_{i}R_{0}^{T},\; M_{0i2}l_{i2}\hat{\rho}_{i}R_{0}^{T},\; \cdots,\; M_{0in_{i}}l_{in_{i}}\hat{\rho}_{i}R_{0}^{T}],\nonumber\\
\Nb_{ix_{0}}=-[M_{0i1}\hat{q}_{i1}^{2},\; M_{0i2}\hat{q}_{i2}^{2},\; \cdots,\;M_{0in_{i}}\hat{q}_{in_{i}}^{2}]^{T},\nonumber\\
\Nb_{i\Omega_{0}}=[M_{0i1}\hat{q}_{i1}^{2}R_{0}\hat{\rho}_{i},\; M_{0i2}\hat{q}_{i2}^{2}R_{0}\hat{\rho}_{i},\; \cdots,\; M_{0in_{i}}\hat{q}_{in_{i}}^{2}R_{0}\hat{\rho}_{i}]^{T},
\end{gather}
and the sub-matrix $\Nb_{qqi}\in\Re^{3n_i\times 3n_i}$ is given by
\begin{align}\Nb_{qqi}=
\begin{bmatrix}
-M_{011}l_{i1}I_{3}&M_{012}l_{i2}\hat{q}_{i2}^2&\cdots&M_{01n_{i}}l_{in_{i}}\hat{q}_{in_{i}}^2\\
M_{021}l_{i1}\hat{q}_{i1}^2&-M_{022}l_{i2}I_{3}&\cdots&M_{02n_{i}}l_{in_{i}}\hat{q}_{in_{i}}^2\\
\vdots&\vdots&&\vdots\\
M_{0n_{i}1}l_{i1}\hat{q}_{i1}^2&M_{0n_{i}2}l_{i2}\hat{q}_{i2}^2&\cdots&-M_{0n_{i}n_{i}}l_{in_{i}}I_{3}
\end{bmatrix}.
\end{align}
The $\Pb\in\Re^{D_{X}}$ matrix is
\begin{align}
\Pb=[P_{x_{0}},\; P_{\Omega_{0}},\; P_{1j},\; P_{2j},\; \cdots,\; P_{nj}]^{T},
\end{align}
and sub-matrices of $\Pb$ matrix are also defined as 
\begin{align*}
P_{x_{0}}&=M_{T}ge_{3}+\sum_{i=1}^{n}{(-f_{i}R_{i}e_{3}+\Delta_{x_i})}-\sum_{i=1}^{n}{M_{iT}R_{0}\hat{\Omega}_{0}^2\rho_{i}},\\
P_{\Omega_{0}}&=-\hat{\Omega}_{0}\bar{J}_{0}\Omega_{0}+\sum_{i=1}^{n}{\hat{\rho}_{i}R_{0}^T(M_{iT}ge_{3}-f_{i}R_{i}e_{3}+\Delta_{x_i})},\nonumber\\
P_{ij}=&-\hat{q}_{ij}^2(-f_{i}R_{i}e_{3}+M_{0ij}ge_{3}+\Delta_{x_i})+M_{0ij}\hat{q}_{ij}^2 R_{0}\hat{\Omega}_{0}^2 \rho_{i}\\
&+M_{0ij}\|\dot{q}_{ij}\|^{2}q_{ij}.\nonumber
\end{align*}
\end{prop}
\begin{proof}
See Appendix~\ref{sec:PfFDM}.
\end{proof}
These equations are derived directly on a nonlinear manifold without any simplification. The dynamics of the payload, flexible cables, and quadrotors are considered explicitly, and they avoid singularities and complexities associated to local coordinates. 
 
\section{CONTROL SYSTEM DESIGN FOR SIMPLIFIED DYNAMIC MODEL}

\subsection{Control Problem Formulation}

Let $x_{0_{d}}\in\Re^{3}$ be the desired position of the payload. The desired attitude of the payload is considered as $R_{0_d}=I_{3\times 3}$, and the desired direction of links is aligned along the vertical direction. The corresponding location of the $i$-th quadrotor at this desired configuration is given by
\begin{align}
x_{i_d}=x_{0_d}+\rho_{i}-\sum_{a=1}^{n_{i}}{l_{ia}e_{3}}.
\end{align}
We wish to design control forces $f_{i}$ and control moments $M_{i}$ of quadrotors such that this desired configuration becomes asymptotically stable.

\subsection{Simplified Dynamic Model}

Control forces for each quadrotor is given by $-f_{i}R_{i}e_{3}$ for the given equations of motion \refeqn{EOMM1}, \refeqn{EOMM2}, \refeqn{EOMM3}, \refeqn{EOMM4}. As such, the quadrotor dynamics is under-actuated. The total thrust magnitude of each quadrotor can be arbitrary chosen, but the direction of the thrust vector is always along the third body fixed axis, represented by $R_i e_3$. But, the rotational attitude dynamics of the quadrotors are fully actuated, and they are not affected by the translational dynamics of the quadrotors or the dynamics of links. 

Based on these observations, in this section, we simplify the model by replacing the $-f_{i}R_{i}e_{3}$ term by a fictitious control input $u_{i}\in\Re^{3}$, and design an expression for $u_i$ to asymptotically stabilize the desired equilibrium. In other words, we assume that the attitude of the quadrotor can be instantaneously changed. Also $\Delta_{x_i}$ are ignored in the simplified dynamic model. The effects of the attitude dynamics are incorporated at the next section.

\subsection{Linear Control System}

The control system for the simplified dynamic model is developed based on the linearized equations of motion. At the desired equilibrium, the position and the attitude of the payload are given by $x_{0_d}$ and $R_{0_d}=I_{3}$, respectively. Also, we have $q_{{ij}_d}=e_{3}$ and $R_{i_d}=I_{3}$. In this equilibrium configuration, the control input for the $i$-th quadrotor is
\begin{align}
u_{i_d}=-f_{i_d}R_{i_d}e_{3},
\end{align}
where the total thrust is $f_{i_d}=(M_{iT}+\frac{m_{0}}{n})g$.

The variation of $x_{0}$ is given by
\begin{align}\label{eqn:xlin}
\delta x_{0}=x_{0}-x_{0_{d}},
\end{align}
and the variation of the attitude of the payload is defined as
\begin{align*}
\delta R_0 = R_{0_d} \hat\eta_0 = \hat\eta_0,
\end{align*}
for $\eta_0\in\Re^3$. The variation of $q_{ij}$ can be written as
\begin{align}\label{eqn:qlin}
\delta q_{ij}=\xi_{ij}\times e_{3},
\end{align}
where $\xi_{ij}\in\Re^{3}$ with $\xi_{ij}\cdot e_{3}=0$. The variation of $\omega_{ij}$ is given by $\delta \omega_{ij}\in\Re^{3}$ with $\delta\omega_{ij}\cdot e_{3}=0$. Therefore, the third element of each of $\xi_{ij}$ and $\delta\omega_{ij}$ for any equilibrium configuration is zero, and they are omitted in the following linearized equations. The state vector of the linearized equation is composed of $C^{T}\xi_{ij}\in\Re^{2}$, where $C=[e_{1},\;e_{2}]\in\Re^{3\times 2}$. The variation of the control input $\delta u_{i}\in\Re^{3\times 1}$, is given as $\delta u_i = u_i-u_{i_d}$.

\begin{prop}\label{prop:stability1}
The linearized equations of the simplified dynamic model are given by 
\begin{align}\label{eqn:EOMLin}
\Mb\ddot \xb  + \Gb\xb = \Bb \delta u+\g(\xb,\dot{\xb}),
\end{align}
where $\g(\xb,\dot{\xb})$ corresponds to the higher order term and the state vector $\xb\in\Re^{D_{\xb}}$ with $D_{\xb}=6+2\sum_{i=1}^{n}n_{i}$ is given by
\begin{align*}
\xb=\begin{bmatrix}
\delta{x}_{0},\eta_{0},C^{T}{\xi}_{1j},C^T{\xi}_{2j},\cdots,C^{T}{\xi}_{nj}\\
\end{bmatrix},
\end{align*}
and $\delta u=[\delta u_{1}^T,\; \delta u_{2}^T,\;\cdots,\;\delta u_{n}^T]^{T}\in\Re^{3n\times 1}$. The matrix $\Mb\in\Re^{D_{\xb}\times D_{\xb}}$ are defined as
\begin{align*}\Mb=
\begin{bmatrix}
M_{T}I_{3}&\Mb_{x_{0}\Omega_{0}}&\Mb_{x_{0}1}&\Mb_{x_{0}2}&\cdots&\Mb_{x_{0}n}\\
\Mb_{\Omega_{0} x_{0}}&\bar{J}_{0}&\Mb_{\Omega_{0}1}&\Mb_{\Omega_{0}2}&\cdots&\Mb_{\Omega_{0}n}\\
\Mb_{1 x_{0}}&\Mb_{1\Omega_{0}}&\Mb_{qq1}&0&\cdots&0\\
\Mb_{2 x_{0}}&\Mb_{2\Omega_{0}}&0&\Mb_{qq2}&\cdots&0\\
\vdots&\vdots&\vdots&\vdots&\vdots&\vdots\\
\Mb_{n x_{0}}&\Mb_{n\Omega_{0}}&0&0&\cdots&\Mb_{qqn}
\end{bmatrix},
\end{align*}
where the sub-matrices are defined as
\begin{gather}
\Mb_{x_{0}\Omega_{0}}=-\sum_{i=1}^{n}{M_{iT}\hat{\rho}_{i}};\quad \Mb_{\Omega_{0} x_{0}}=\Mb_{x_{0}\Omega_{0}}^{T},\nonumber\\
\Mb_{x_{0}i}=[M_{0i1}l_{i1}\hat{e}_{3}C,\; M_{0i2}l_{i2}\hat{e}_{3}C,\; \cdots,\;M_{0in_{i}}l_{in_{i}}\hat{e}_{3}C],\nonumber\\
\Mb_{\Omega_{0}i}=[M_{0i1}l_{i1}\hat{\rho}_{i}C,\; M_{0i2}l_{i2}\hat{\rho}_{i}C,\; \cdots,\; M_{0in_{i}}l_{in_{i}}\hat{\rho}_{i}C],\nonumber\\
\Mb_{ix_{0}}=-[M_{0i1}C^{T}\hat{e}_{3},\; M_{0i2}C^{T}\hat{e}_{3},\; \cdots,\; M_{0in_{i}}C^{T}\hat{e}_{3}],\\
\Mb_{i\Omega_{0}}=[M_{0i1}C^{T}\hat{e}_{3}\hat{\rho}_{i},\; M_{0i2}C^{T}\hat{e}_{3}\hat{\rho}_{i},\;\cdots,\; M_{0in_{i}}C^{T}\hat{e}_{3}\hat{\rho}_{i}],
\end{gather}
and the sub-matrix $\Mb_{qqi}\in\Re^{2n_i\times 2n_i}$ is given by
\begin{align}\Mb_{qqi}=
\begin{bmatrix}
M_{i11}l_{i1}I_{2}&M_{i12}l_{i2}I_{2}&\cdots&M_{i1n_{i}}l_{in_{i}}I_{2}\\
M_{i21}l_{i1}I_{2}&M_{i22}l_{i2}I_{2}&\cdots&M_{i2n_{i}}l_{in_{i}}I_{2}\\
\vdots&\vdots&&\vdots\\
M_{in_{i}1}l_{i1}I_{2}&M_{in_{i}2}l_{i2}I_{2}&\cdots&M_{in_{i}n_{i}}l_{in_{i}}I_{2}
\end{bmatrix}.
\end{align}
The matrix $\Gb\in\Re^{D_{\xb}\times D_{\xb}}$ is defined as
\begin{align*}
\Gb=\begin{bmatrix}
0&0&0&0&0&0\\
0&\Gb_{\Omega_{0}\Omega_{0}}&0&0&0&0\\
0&0&\Gb_{1}&0&0&0\\
0&0&0&\Gb_{2}&0&0\\
\vdots&\vdots&\vdots&\vdots&\vdots&\vdots\\
0&0&0&0&0&\Gb_{n}
\end{bmatrix},
\end{align*}
where $\Gb_{\Omega_{0}\Omega_{0}}=\sum_{i=1}^{n}\frac{m_{0}}{n}g\hat{\rho}_{i}\hat{e}_{3}$ and the sub-matrices $\Gb_{i}\in\Re^{2n_{i}\times 2n_{i}}$ are
\begin{align*}
\Gb_{i}= \diag[(-M_{iT}-\frac{m_{0}}{n}+M_{0ij})ge_{3}I_{2}].
\end{align*}
The matrix $\Bb\in\Re^{D_{\xb}\times 3n}$ is given by 
\begin{align*}
\Bb=\begin{bmatrix}
I_{3}&I_{3}&\cdots&I_{3}\\
\hat{\rho}_{1}&\hat{\rho}_{2}&\cdots&\hat{\rho}_{n}\\
\Bb_{\Bb}&0&0&0\\
0&\Bb_{\Bb}&0&0\\
\vdots&\vdots&\vdots&\vdots\\
0&0&0&\Bb_{\Bb}
\end{bmatrix},
\end{align*}
where $\Bb_{\Bb}=-[C^{T}\hat{e}_{3},\; C^{T}\hat{e}_{3},\; \cdots,\; C^{T}\hat{e}_{3}]^{T}$.
\end{prop}
\begin{proof}
See Appendix~\ref{sec:P1stability}.
\end{proof}
We present the following PD-type control system for the linearized dynamics
\begin{align}
\delta u_{i}
=&-K_{x_{i}}\xb-K_{\dot{x}_{i}}\dot{\xb},
\end{align}
for controller gains $K_{x_{i}},K_{\dot{x}_{i}}\in\Re^{3\times D_{\xb}}$. Provided that \refeqn{EOMLin} is controllable, we can choose the combined controller gains $K_{\xb}=[K_{x_{1}}^T,\,\ldots\;K_{x_{n}}^T]^{T}$, $K_{\dot{\xb}}=[K_{\dot{x}_{1}}^T,\,\ldots K_{\dot{x}_{n}}^T]^{T}\in\Re^{3n\times D_{\xb}}$ such that the equilibrium is asymptotically stable for the linearized equations~\cite{Kha96}. Then, the equilibrium becomes asymptotically stable for the nonlinear Euler-Lagrange equation. The controlled linearized system can be written as
\begin{align}\label{eqn:zdot1}
\dot{z}_{1}=\mathds{A} z_{1}+\mathds{B}(\Bb X+\g(\xb,\dot{\xb})),
\end{align}
where $z_{1}=[\xb,\dot{\xb}]^{T}\in\Re^{2D_{\xb}}$ and
\begin{align}
\mathds{A}=\begin{bmatrix}
0&I\\
-\Mb^{-1}(\Gb+\Bb K_{\xb})&-\Mb^{-1}\Bb K_{\dot{\xb}}
\end{bmatrix},
\mathds{B}=\begin{bmatrix}
0\\
\Mb^{-1}
\end{bmatrix}.
\end{align}
We can also choose $K_{\xb}$ and $K_{\dot{\xb}}$ such that $\mathds{A}\in\Re^{2D_{x}\times 2D_{x}}$ is Hurwitz. Then for any positive definite matrix $Q\in\Re^{2D_{\xb}\times 2D_{\xb}}$, there exist a positive definite and symmetric matrix $P\in\Re^{2D_{\xb}\times 2D_{\xb}}$ such that $\mathds{A}^{T}P+P\mathds{A}=-Q$ according to~\cite[Thm 3.6]{Kha96}.


\section{CONTROL SYSTEM DESIGN FOR THE FULL DYNAMIC MODEL}
The control system designed at the previous section is based on a simplifying assumption that each quadrotor can generates a thrust along any direction. In the full dynamic model, the direction of the thrust for each quadrotor is parallel to its third body-fixed axis always. In this section, the attitude of each quadrotor is controlled such that the third body-fixed axis becomes parallel to the direction of the ideal control force designed in the previous section. Also in the full dynamics model, we considers the $\Delta_{x_i}$ in the control design and introduce a new integral term to eliminate the disturbances and uncertainties. The central idea is that the attitude $R_{i}$ of the quadrotor is controlled such that its total thrust direction $-R_{i}e_{3}$, corresponding to the third body-fixed axis, asymptotically follows the direction of the fictitious control input $u_{i}$. By choosing the total thrust magnitude properly, we can guarantee asymptotical stability for the full dynamic model. 

Let $A_{i}\in\Re^{3}$ be the ideal total thrust of the $i$-th quadrotor that asymptotically stabilize the desired equilibrium. Therefor, we have
\begin{align}\label{eqn:Ai}
A_{i}=u_{i_{d}}+\delta u_{i}=-K_{x_{i}}\xb-K_{\dot{x}_{i}}\dot{\xb}-K_{z}\satr_{\sigma}(e_{\xb})+u_{i_{d}},
\end{align} 
where $f_{i_{d}}\in\Re$ and $u_{i_{d}}\in\Re^{3}$ are the total thrust and control input of each quadrotor at its equilibrium respectively. 
where the following integral term $e_{\xb}\in\Re^{D_{x}}$ is added to eliminate the effect of disturbance $\Delta_{x_i}$ in the full dynamic model 
\begin{align}\label{eqn:exterm}
e_{\xb}=\int^{t}_{0}{(P\mathds{B})^{T}z_{1}(\tau)\;d\tau},
\end{align}
where $K_z =[k_{z}I_3,k_{z}I_3,k_{z_1}I_{3\times 2},\ldots k_{z_n}I_{3\times 2}]\in\Re^{3\times D_{x}}$ is an integral gain. For a positive constant $\sigma\in\Re$, a saturation function $\sat_\sigma:\Re\rightarrow [-\sigma,\sigma]$ is introduced as
\begin{align*}
\sat_{\sigma}(y) = \begin{cases}
\sigma & \mbox{if } y >\sigma\\
y & \mbox{if } -\sigma \leq y \leq\sigma\\
-\sigma & \mbox{if } y <-\sigma\\
\end{cases}.
\end{align*}
If the input is a vector $y\in\Re^n$, then the above saturation function is applied element by element to define a saturation function $\sat_\sigma(y):\Re^n\rightarrow [-\sigma,\sigma]^n$ for a vector. From the desired direction of the third body-fixed axis of the $i$-th quadrotor, namely $b_{3_{i}}\in\Sph^2$, is given by
\begin{align}
b_{3_{i}}=-\frac{A_{i}}{\|A_{i}\|}.
\end{align}
This provides a two-dimensional constraint on the three dimensional desired attitude of each quadrotor, such that there remains one degree of freedom. To resolve it, the desired direction of the first body-fixed axis $b_{1_{i}}(t)\in\Sph^2$ is introduced as a smooth function of time. Due to the fact that the first body-fixed axis is normal to the third body-fixed axis, it is impossible to follow an arbitrary command $b_{1_{i}}(t)$ exactly. Instead, its projection onto the plane normal to $b_{3_{i}}$ is followed, and the desired direction of the second body-fixed axis is chosen to constitute an orthonormal frame~\cite{LeeLeoAJC13}. More explicitly, the desired attitude of the $i$-th quadrotor is given by
\begin{align}
R_{i_{c}}=\begin{bmatrix}
-\frac{(\hat{b}_{3_{i}})^{2}b_{1_{i}}}{\|(\hat{b}_{3_{i}})^{2}b_{1_{i}}\|} & \frac{\hat{b}_{3_{i}}b_{1_{i}}}{\|\hat{b}_{3_{i}}b_{1_{i}}\|} & b_{3_{i}}\end{bmatrix},
\end{align}
which is guaranteed to be an element of $\SO$. The desired angular velocity is obtained from the attitude kinematics equation, $\Omega_{i_{c}}=(R_{i_{c}}^{T}\dot{R}_{i_{c}})^\vee\in\Re^{3}$.

Define the tracking error vectors for the attitude and the angular velocity of the $i$-th quadrotor as
\begin{align}
e_{R_{i}}=\frac{1}{2}(R_{i_{c}}^{T}R_{i}-R_{i}^{T}R_{i_{c}})^{\vee},\; e_{\Omega_{i}}=\Omega_{i}-R_{i}^{T}R_{i_{c}}\Omega_{i_{c}},
\end{align}
and a configuration error function on $\SO$ as follows
\begin{align}
\Psi_{i}= \frac{1}{2}\trs{I- R_{{i}_c}^T R_{i}}.
\end{align}
The thrust magnitude is chosen as the length of $u_{i}$, projected on to $-R_{i}e_{3}$, and the control moment is chosen as a tracking controller on $\SO$:
\begin{align}
f_{i}=&-A_{i}\cdot R_{i}e_{3},\label{eqn:fi}\\
M_{i}=&-k_R e_{R_{i}} -k_\Omega e_{\Omega_{i}}-k_{I}e_{I_i}\nonumber\\
&+(R_{i}^TR_{c_i}\Omega_{c_{i}})^\wedge J_{i} R_{i}^T R_{c_i} \Omega_{c_i} + J_{i} R_{i}^T R_{c_i}\dot\Omega_{c_i},\label{eqn:Mi}
\end{align}
where $k_R,k_\Omega$, and $k_{I}$ are positive constants and the following integral term is introduced to eliminate the effect of fixed disturbance $\Delta_{R_i}$
\begin{align}\label{eqn:integralterm}
e_{I_i}=\int^{t}_{0}{e_{\Omega_i}(\tau)+c_{2}e_{R_i}(\tau)d\tau},
\end{align}
where $c_{2}$ is a positive constant. Stability of the corresponding controlled systems for the full dynamic model can be studied by showing the the error due to the discrepancy between the desired direction $b_{3_{i}}$ and the actual direction $R_{i}e_{3}$. 
\begin{prop}\label{prop:stability}
Consider control inputs $f_i$, $M_i$ defined in \refeqn{fi} and \refeqn{Mi}. There exist controller parameters and gains such that, (i) the zero equilibrium of tracking error is stable in the sense of Lyapunov; (ii) the tracking errors $e_{R_i}$, $e_{\Omega_i}$, $\xb$, $\dot{\xb}$ asymptotically converge to zero as $t\rightarrow\infty$; (iii) the integral terms $e_{I_i}$ and $e_{\xb}$ are uniformly bounded.
\end{prop}
\begin{proof}
See Appendix \ref{sec:Pstabilityddd}.
\end{proof}
By utilizing geometric control systems for quadrotor, we show that the hanging equilibrium of the links can be asymptotically stabilized while translating the payload to a desired position and attitude. The control systems proposed explicitly consider the coupling effects between the cable/load dynamics and the quadrotor dynamics. We presented a rigorous Lyapunov stability analysis to establish stability properties without any timescale separation assumptions or singular perturbation, and a new nonlinear integral control term is designed to guarantee robustness against unstructured uncertainties in both rotational and translational dynamics.

\section{NUMERICAL EXAMPLE}

We demonstrate the desirable properties of the proposed control system with numerical examples. Two cases are presented. At the first case, a payload is transported to a desired position from the ground. The second case considers stabilization of a payload with large initial attitude errors.

\subsection{Stabilization of the Rigid Body}
Consider four quadrotors $(n=4)$ connected via flexible cables  to a rigid body payload. Initial conditions are chosen as
\begin{gather*}
x_{0}(0)=[1.0,\; 4.8,\; 0.0]^{T}\,\mathrm{m},\; v_{0}(0)=0_{3\times 1},\\
q_{ij}(0)=e_{3},\; \omega_{ij}(0)=0_{3\times 1},\; R_{i}(0)=I_{3\times 3},\; \Omega_{i}(0)=0_{3\times 1}\\
R_{0}(0)=I_{3\times3},\; \Omega_{0}=0_{3\times 1}.
\end{gather*}
The desired position of the payload is chosen as
\begin{align}
x_{0_{d}}(t)=[0.44,\; 0.78,\; -0.5]^{T}\,\mathrm{m}.
\end{align}
\begin{figure}[H]
\centerline{
	\subfigure[Payload position ($x_0$:blue, $x_{0_{d}}$:red)]{
		\includegraphics[width=0.5\columnwidth]{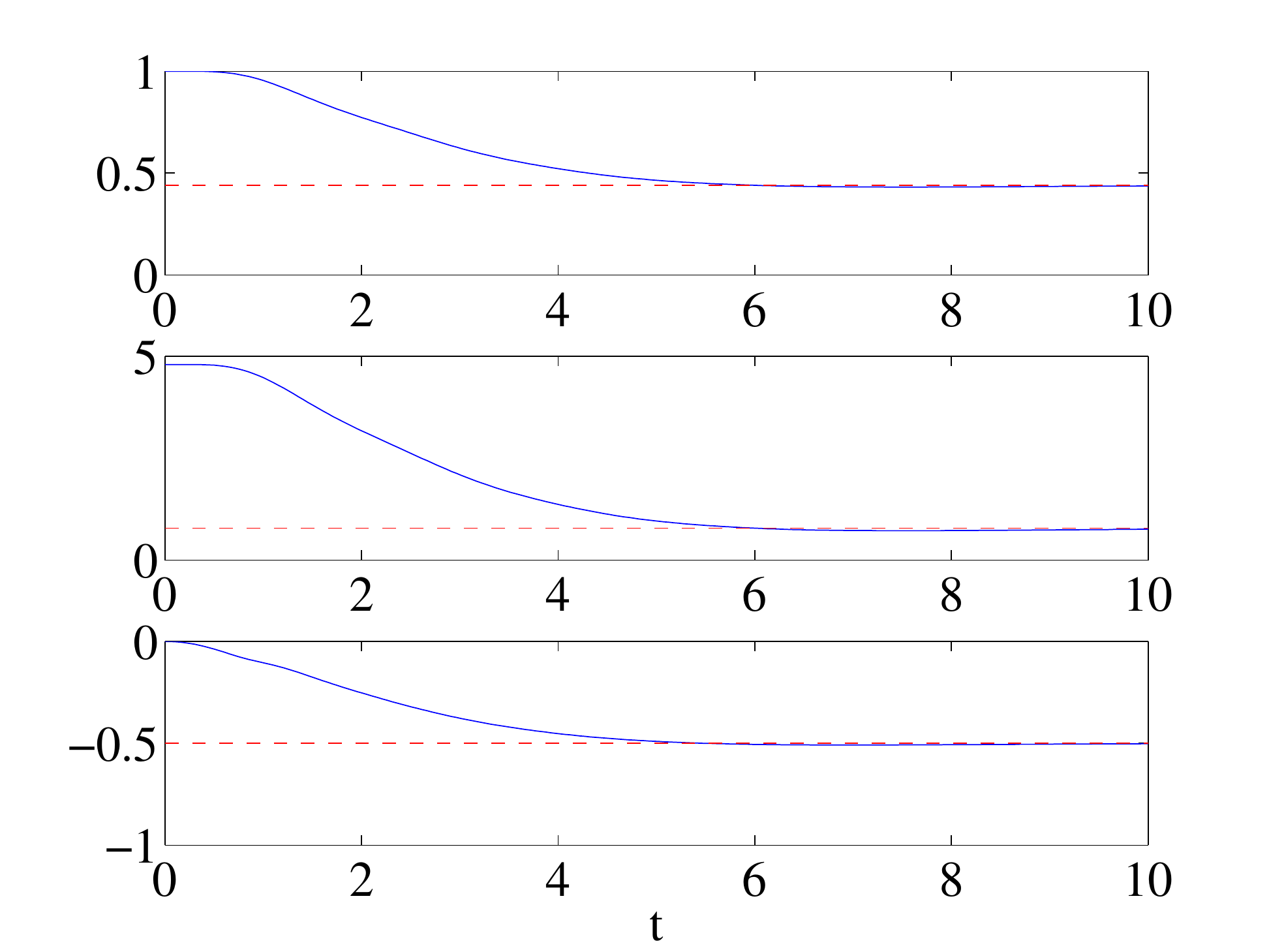}}
	\subfigure[Payload velocity ($v_0$:blue, $v_{0_{d}}$:red)]{
		\includegraphics[width=0.5\columnwidth]{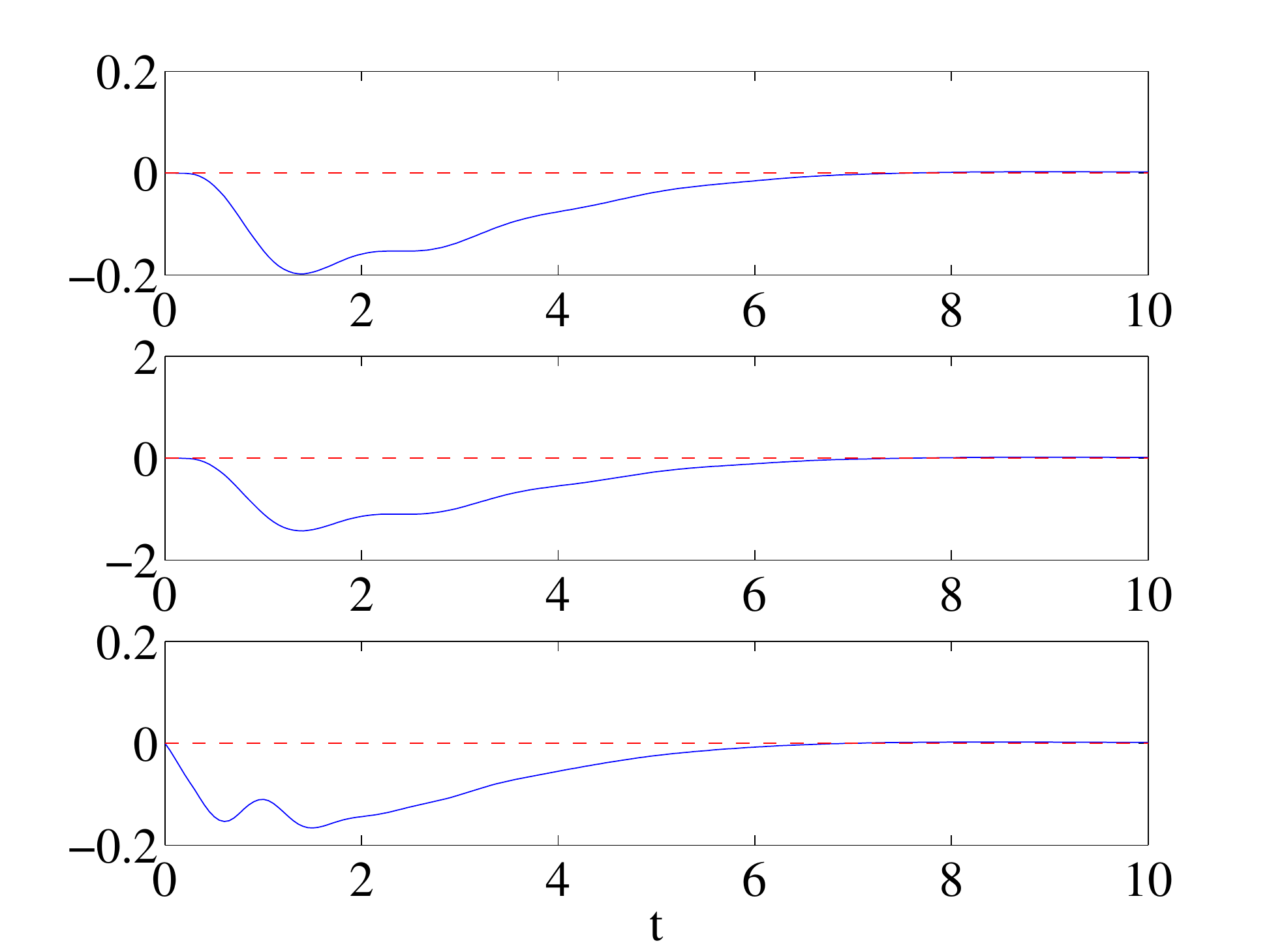}}
}
\centerline{
	\subfigure[Payload angular velocity $\Omega_{0}$]{
		\includegraphics[width=0.5\columnwidth]{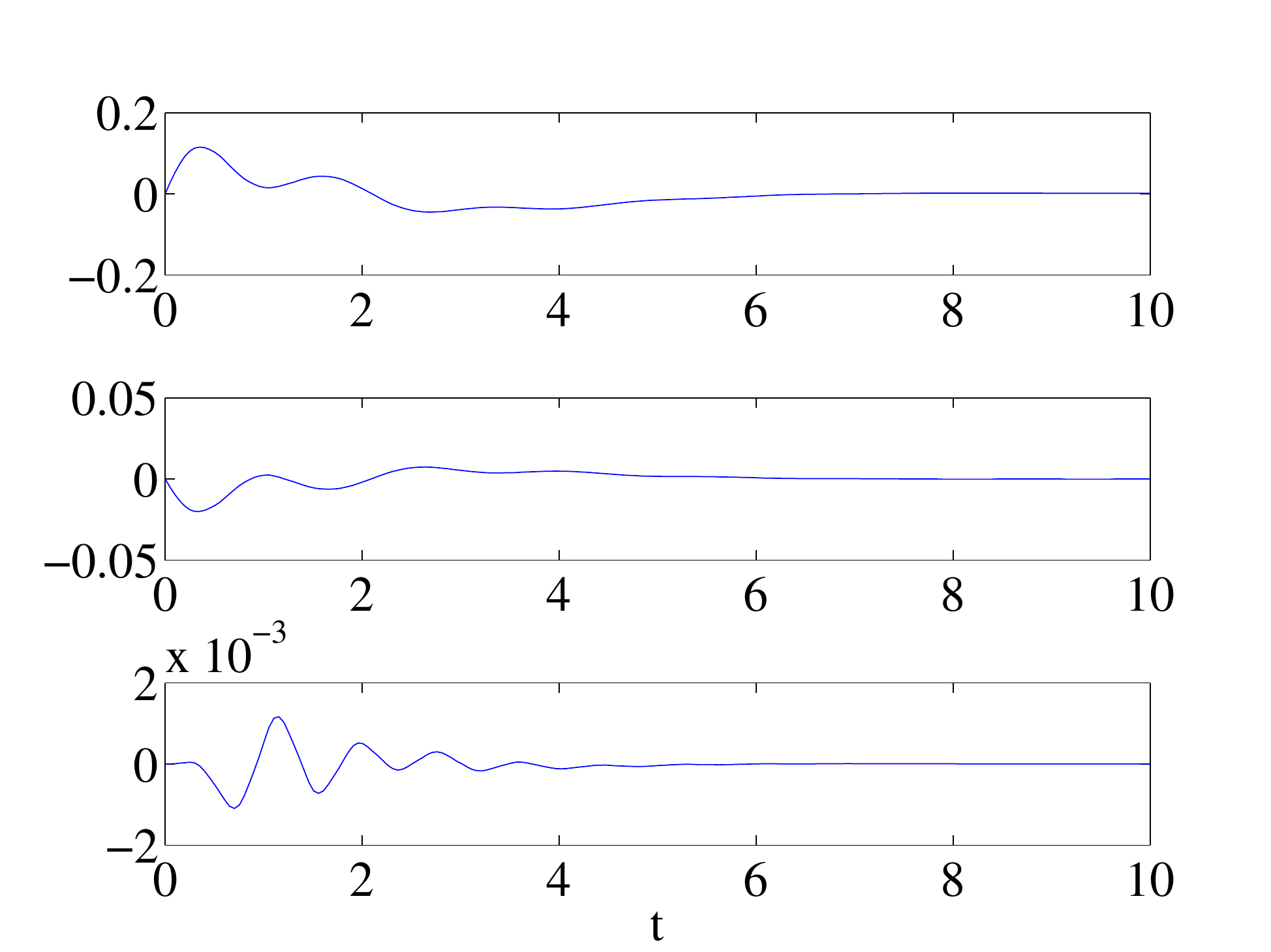}}
	\subfigure[Quadrotors angular velocity errors $e_{\Omega_{i}}$]{
		\includegraphics[width=0.5\columnwidth]{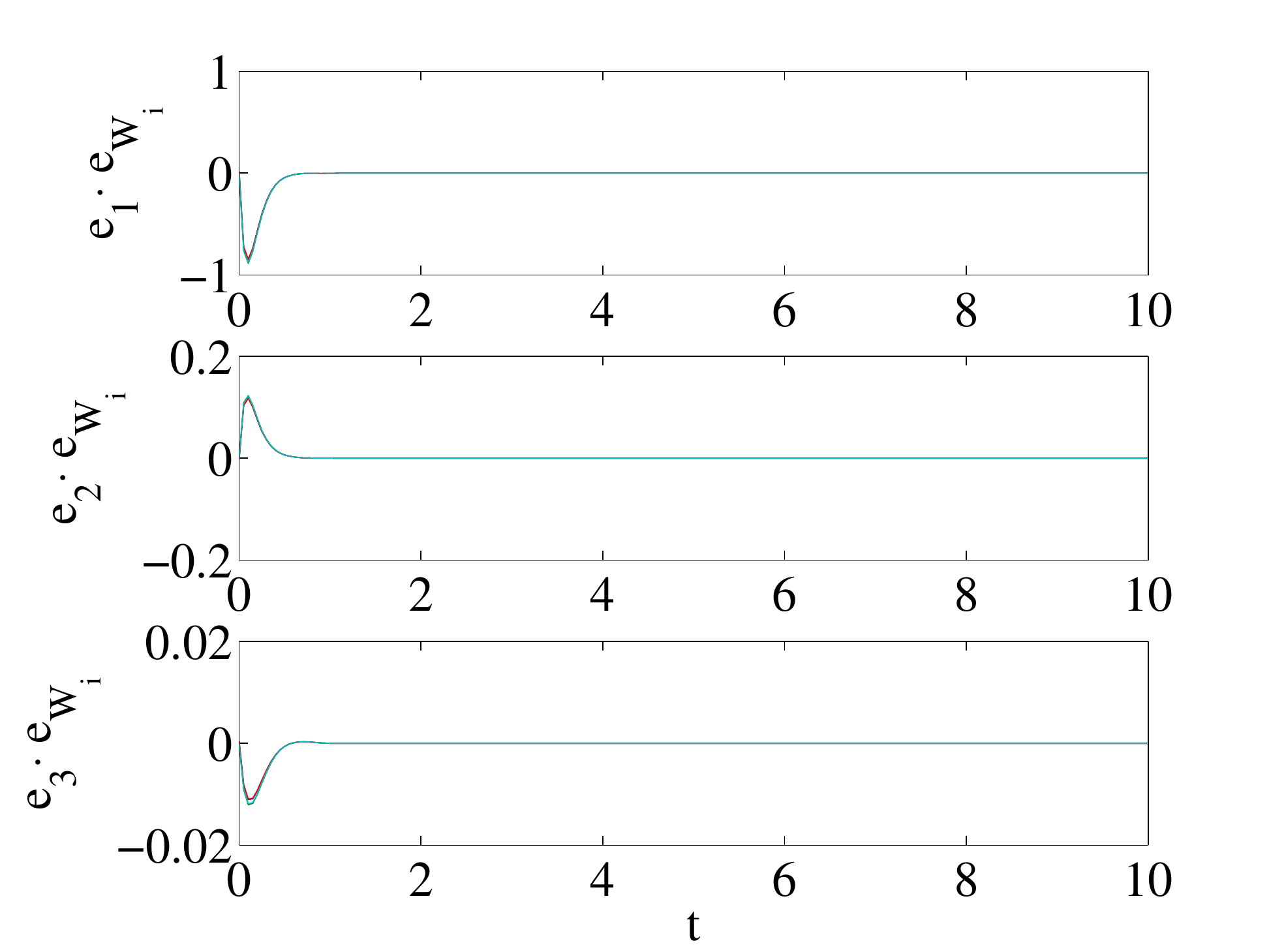}}
}
\centerline{
	\subfigure[Payload attitude error $\Psi_{0}$]{
		\includegraphics[width=0.5\columnwidth]{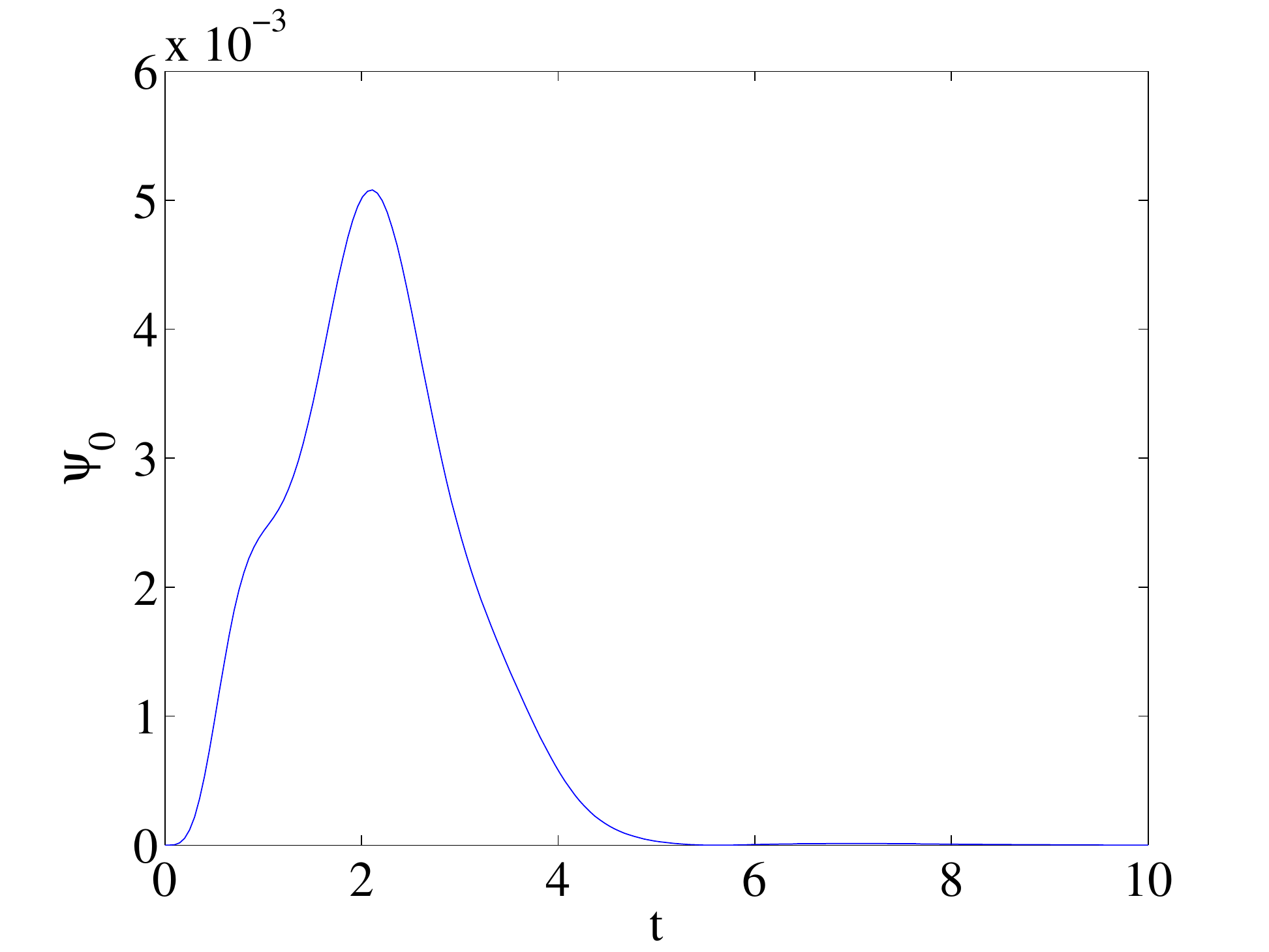}}
	\subfigure[Quadrotors attitude errors $\Psi_{i}$]{
		\includegraphics[width=0.5\columnwidth]{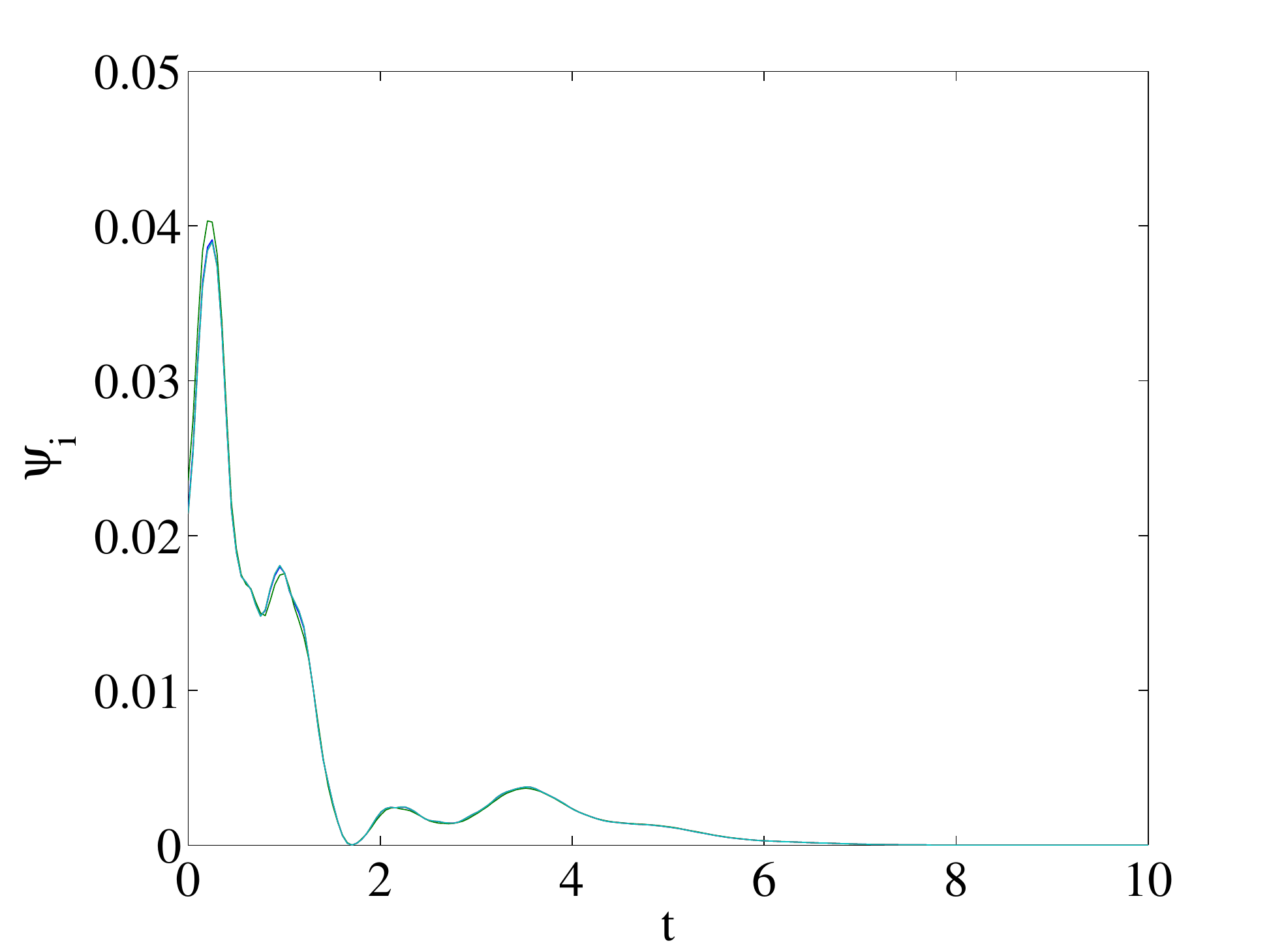}}
}
\centerline{
	\subfigure[Quadrotors total thrust inputs $f_{i}$]{
		\includegraphics[width=0.5\columnwidth]{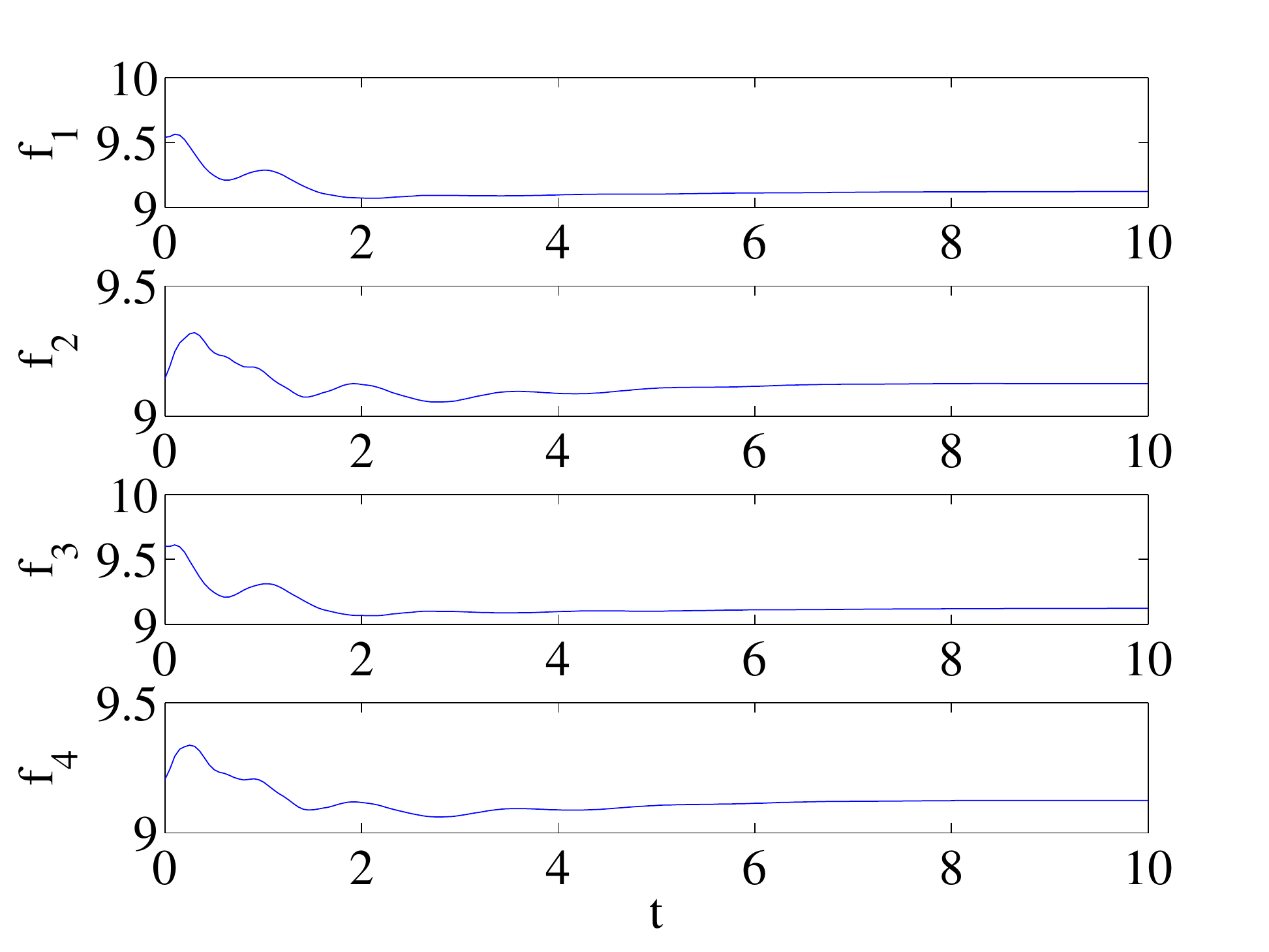}}
	\subfigure[Direction error $e_{q}$, and angular velocity error $e_{\omega}$ for the links]{
		\includegraphics[width=0.5\columnwidth]{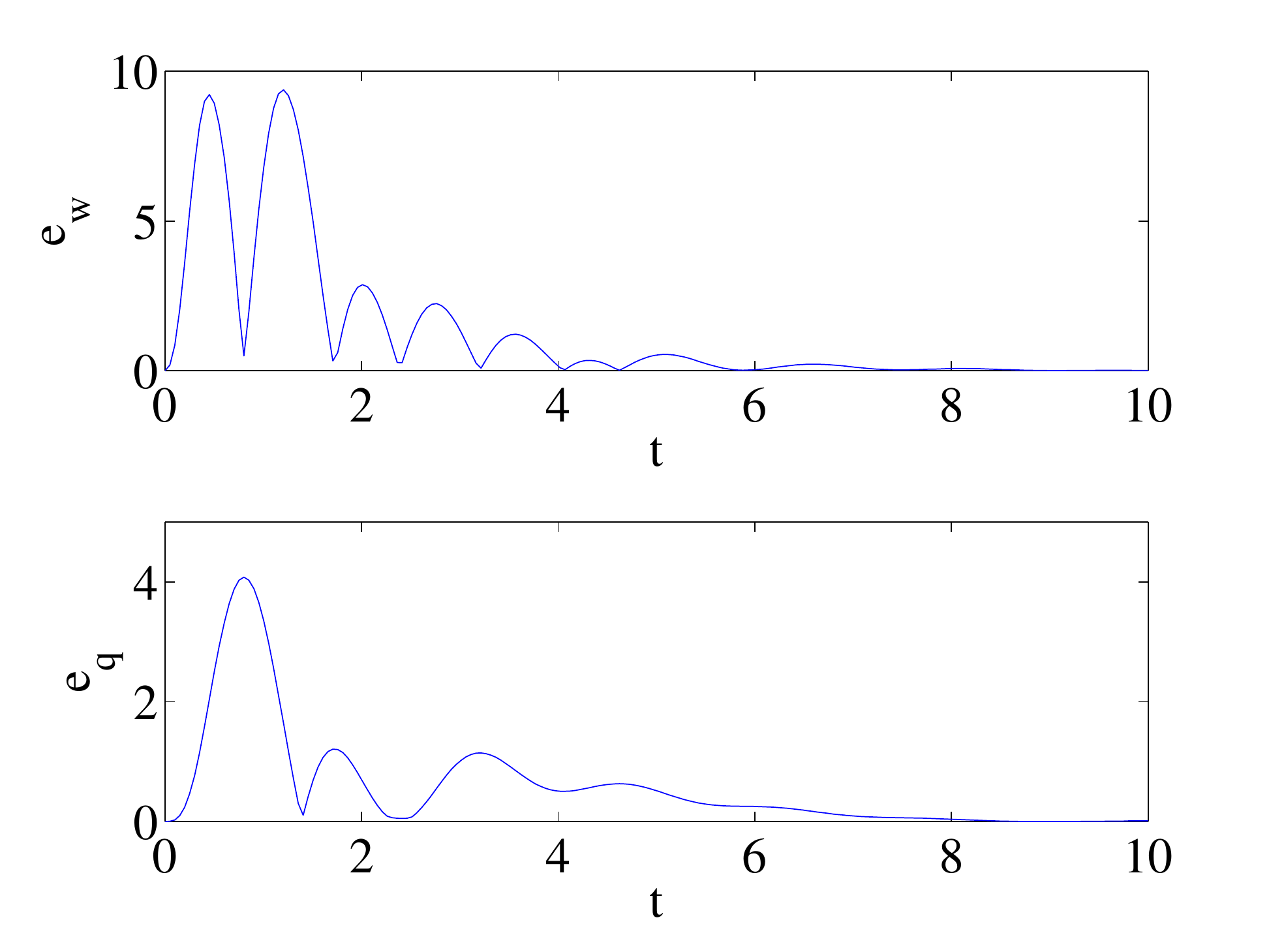}}
}
\caption{Stabilization of a rigid-body connected to multiple quadrotors}\label{fig:simresults1}
\end{figure}
The mass properties of quadrotors are chosen as
\begin{gather}
m_{i}=0.755\,\mathrm{kg},\nonumber\\ 
J_{i}=\diag[0.557,\; 0.557,\; 1.05]\times 10^{-2} \mathrm{kgm^2}.
\end{gather}
The payload is a box with mass $m_{0}=0.5\,\mathrm{kg}$, and its length, width, and height are $0.6$, $0.8$, and $0.2\,\mathrm{m}$, respectively. Each cable connecting the rigid body to the $i$-th quadrotor is considered to be $n_{i}=5$ rigid links. All the links have the same mass of $m_{ij}=0.01\,\mathrm{kg}$ and length of $l_{ij}=0.15\,\mathrm{m}$. Each cable is attached to the following points of the payload
\begin{gather*}
\rho_{1}=[0.3,\; -0.4,\; -0.1]^T\,\mathrm{m},\; \rho_{2}=[0.3,\; 0.4,\; -0.1]^T\,\mathrm{m},\\
\rho_{3}=[-0.3,\; -0.4,\; -0.1]^T\,\mathrm{m},\; \rho_{4}=[-0.3,\; 0.4,\; -0.1]^T\,\mathrm{m}.
\end{gather*}
Numerical simulation results are presented at Figure \ref{fig:simresults1}, which shows the position and velocity of the payload, and its tracking errors. We have also presented the link direction and link angular velocity errors defined as
\begin{gather}
e_{q}=\sum_{i=1}^{m}\sum_{j=1}^{n_{i}}\|q_{ij}-e_3\|,\\
e_{\omega}=\sum_{i=1}^{m}\sum_{j=1}^{n_{i}}\|\omega_{ij}\|.
\end{gather}

\begin{figure}[h]
\centerline{
	\subfigure[3D perspective]{
		\includegraphics[width=0.9\columnwidth]{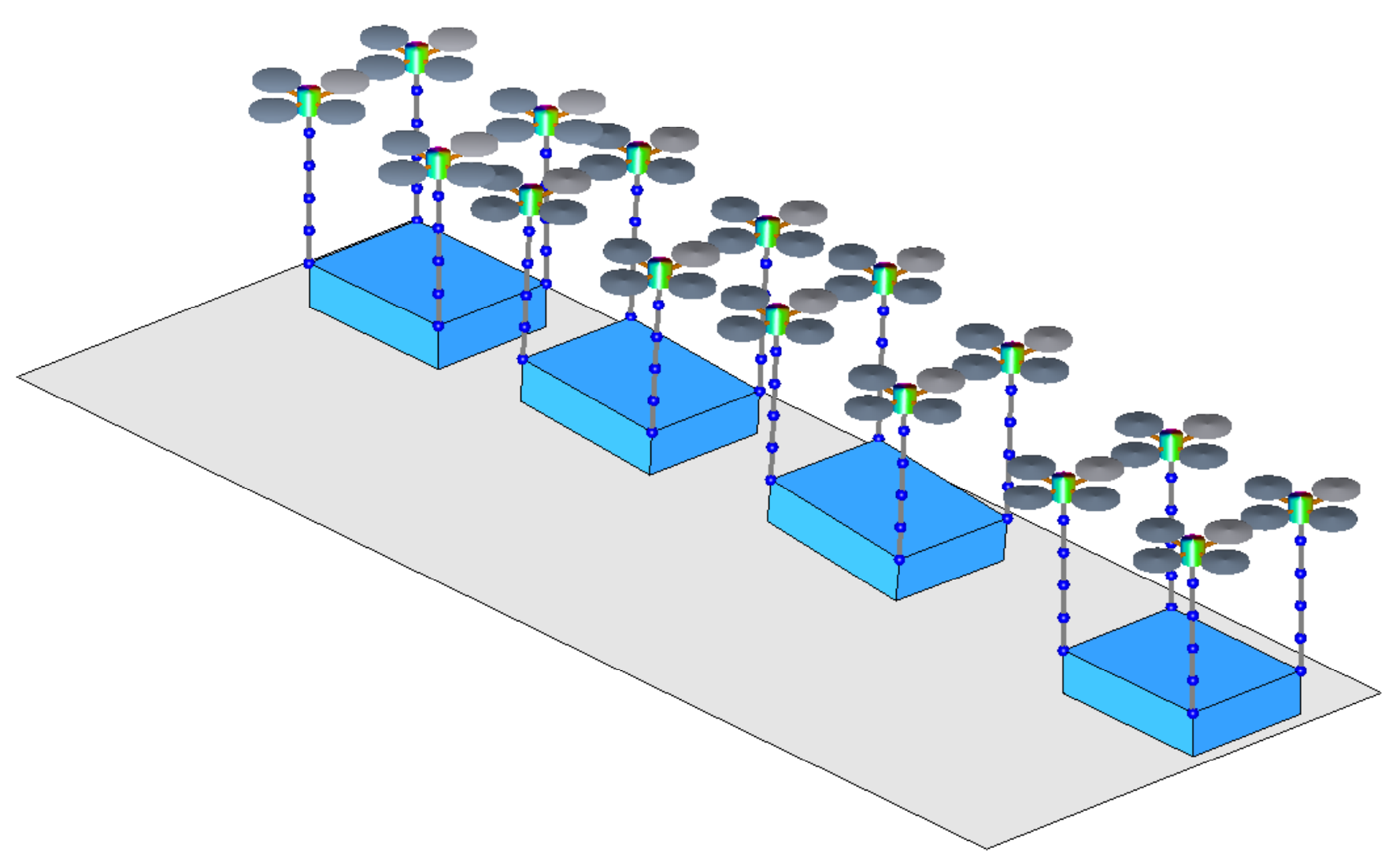}}
}
\centerline{
	\subfigure[Side view]{
		\includegraphics[width=0.9\columnwidth]{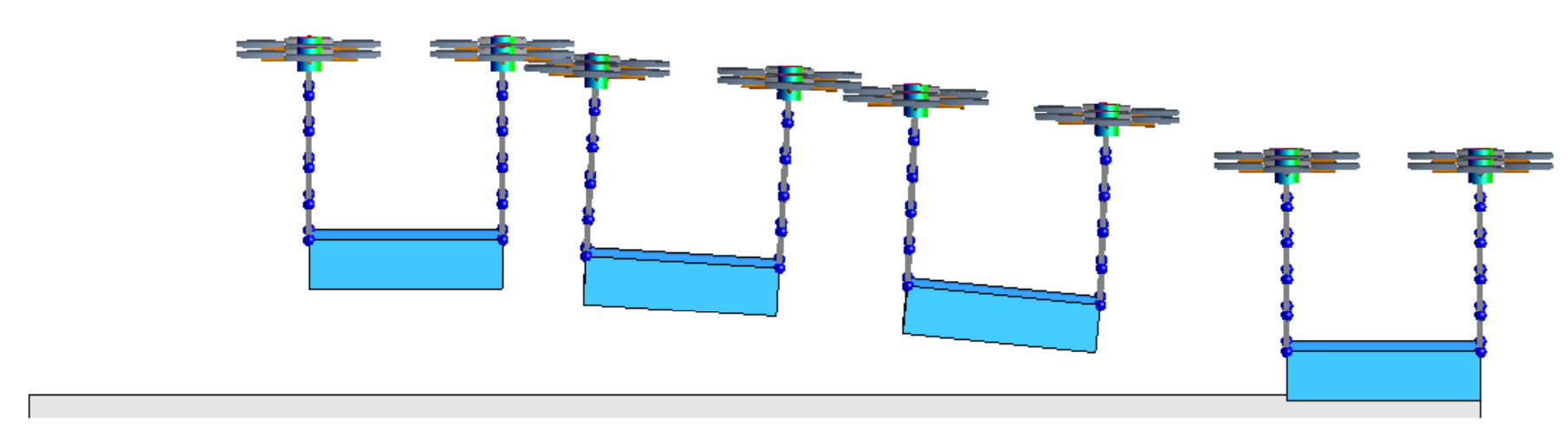}}
}
\centerline{
	\subfigure[Top view]{
		\includegraphics[width=0.9\columnwidth]{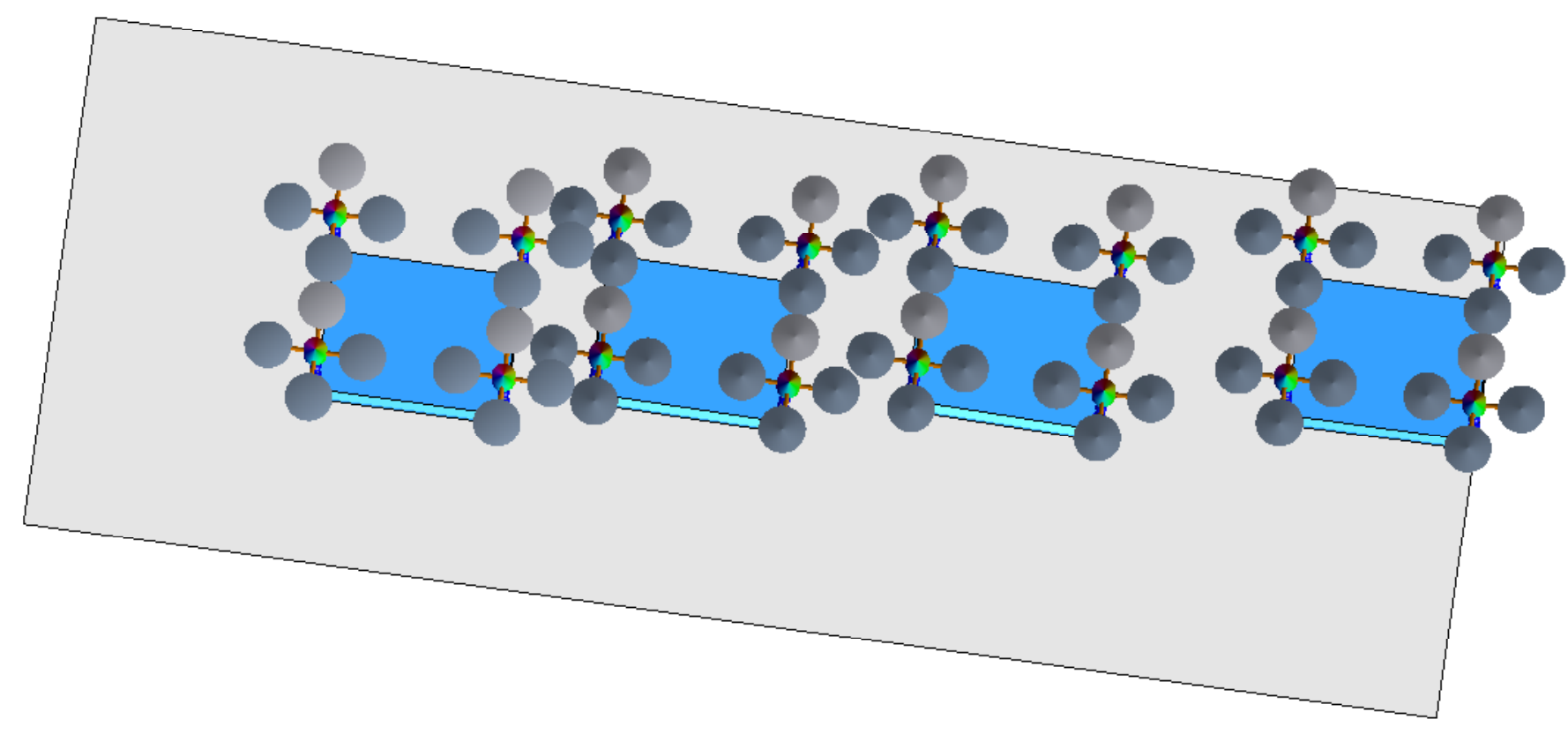}}
}
\caption{Snapshots of controlled maneuver}\label{fig:simresults1snap}
\end{figure}

\begin{figure}[h]
\centerline{\hspace{-0.5cm}
	\subfigure[Payload position $x_0$:blue, $x_{0_{d}}$:red]{
		\includegraphics[width=0.5\columnwidth]{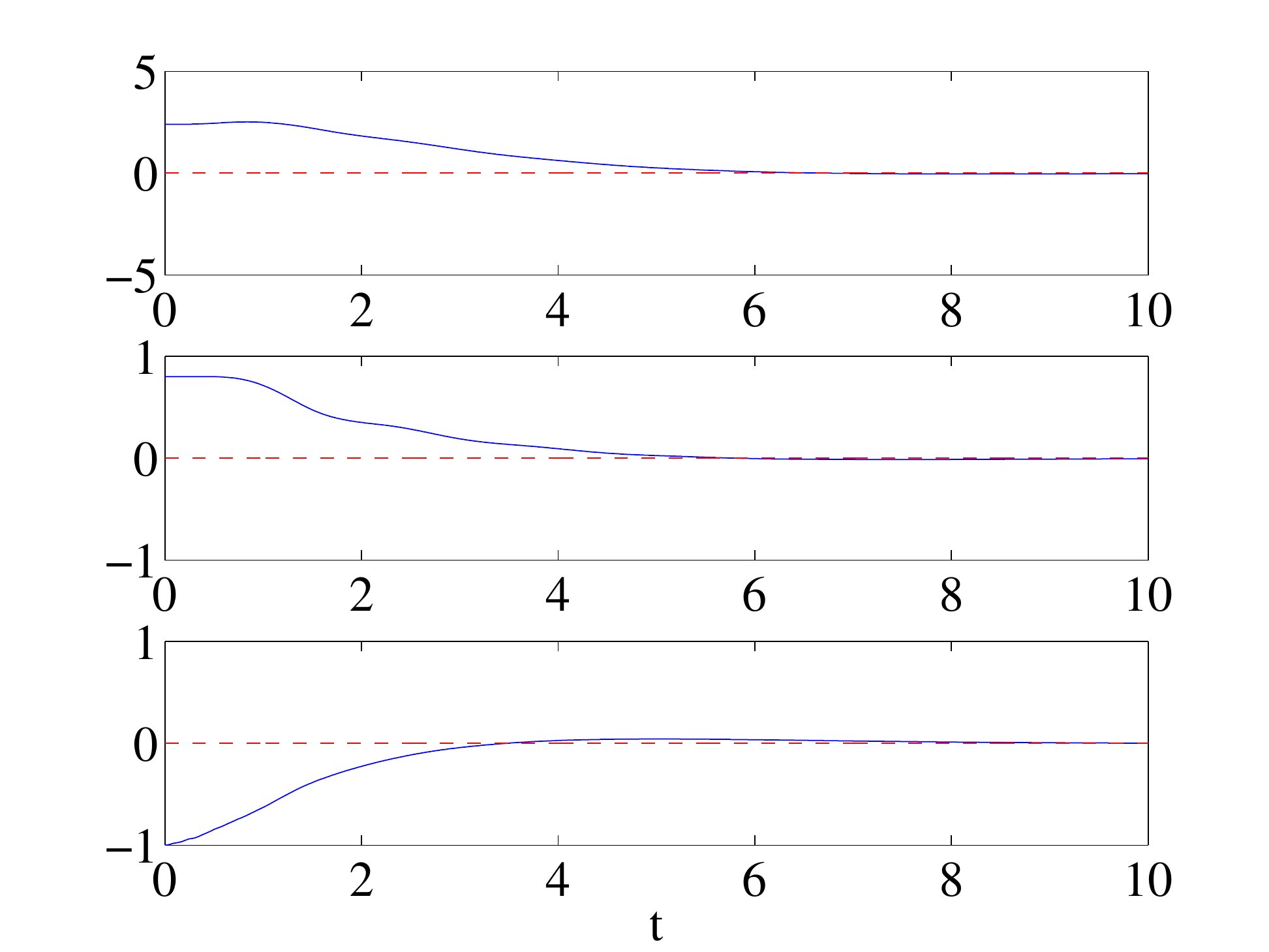}\label{fig:case2x0}}
	\subfigure[Payload velocity $v_0$:blue, $v_{0_{d}}$:red]{
		\includegraphics[width=0.5\columnwidth]{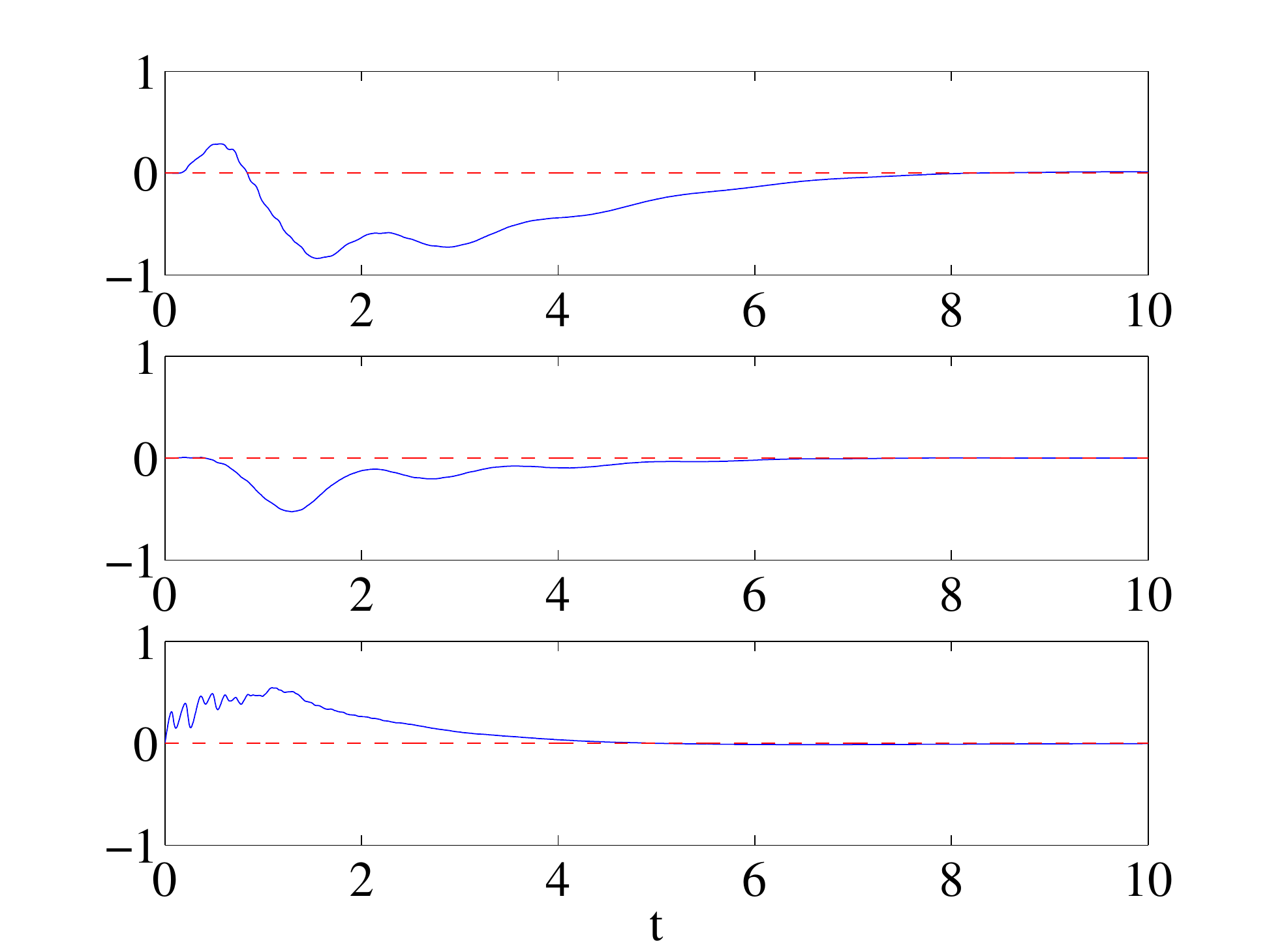}\label{fig:case2v0}}
}
\centerline{\hspace{-0.5cm}
	\subfigure[Payload angular velocity $\Omega_{0}$]{
		\includegraphics[width=0.5\columnwidth]{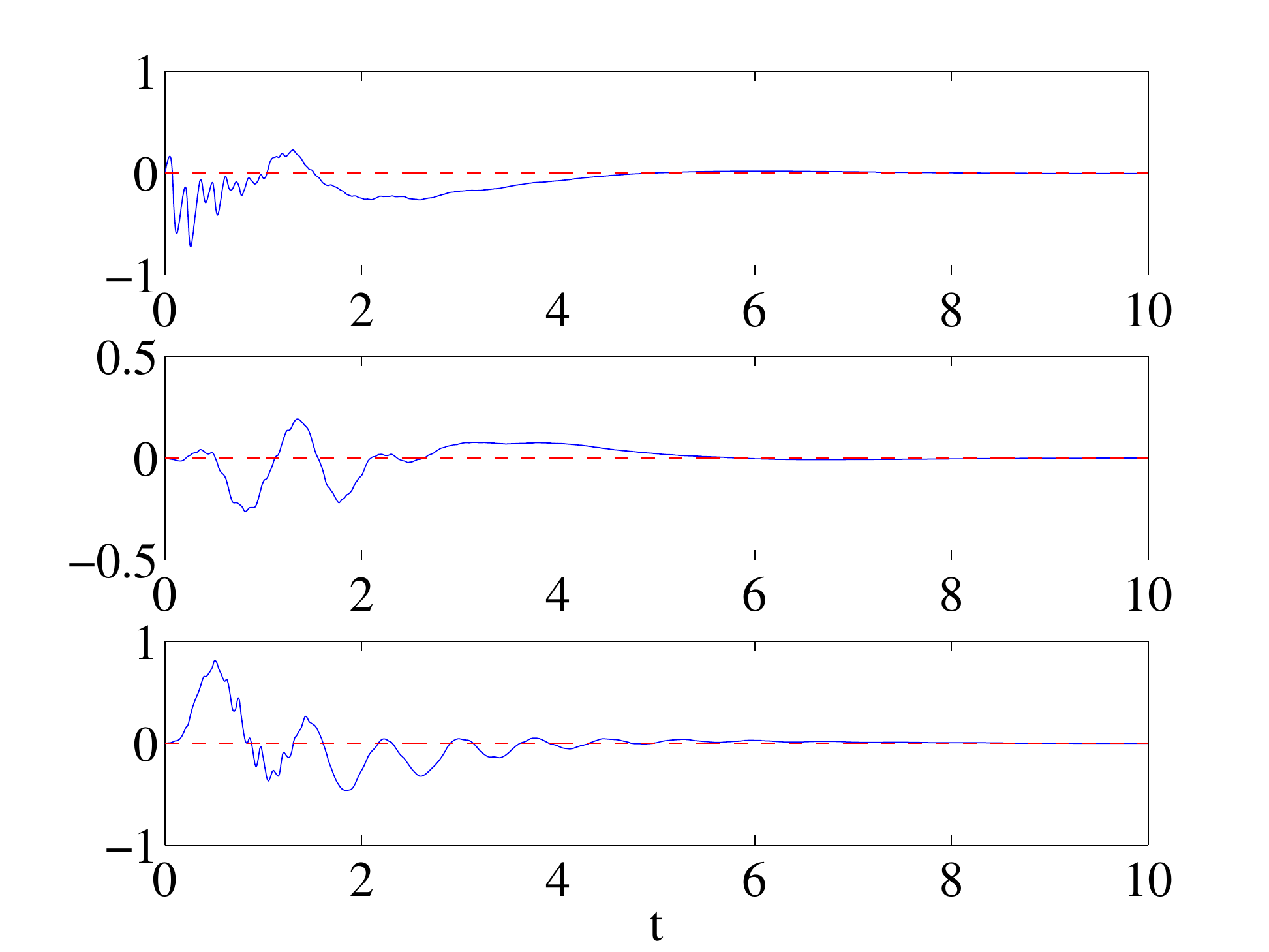}\label{fig:case2W0}}
	\subfigure[Quadrotors angular velocity errors $e_{\Omega_{i}}$]{
		\includegraphics[width=0.5\columnwidth]{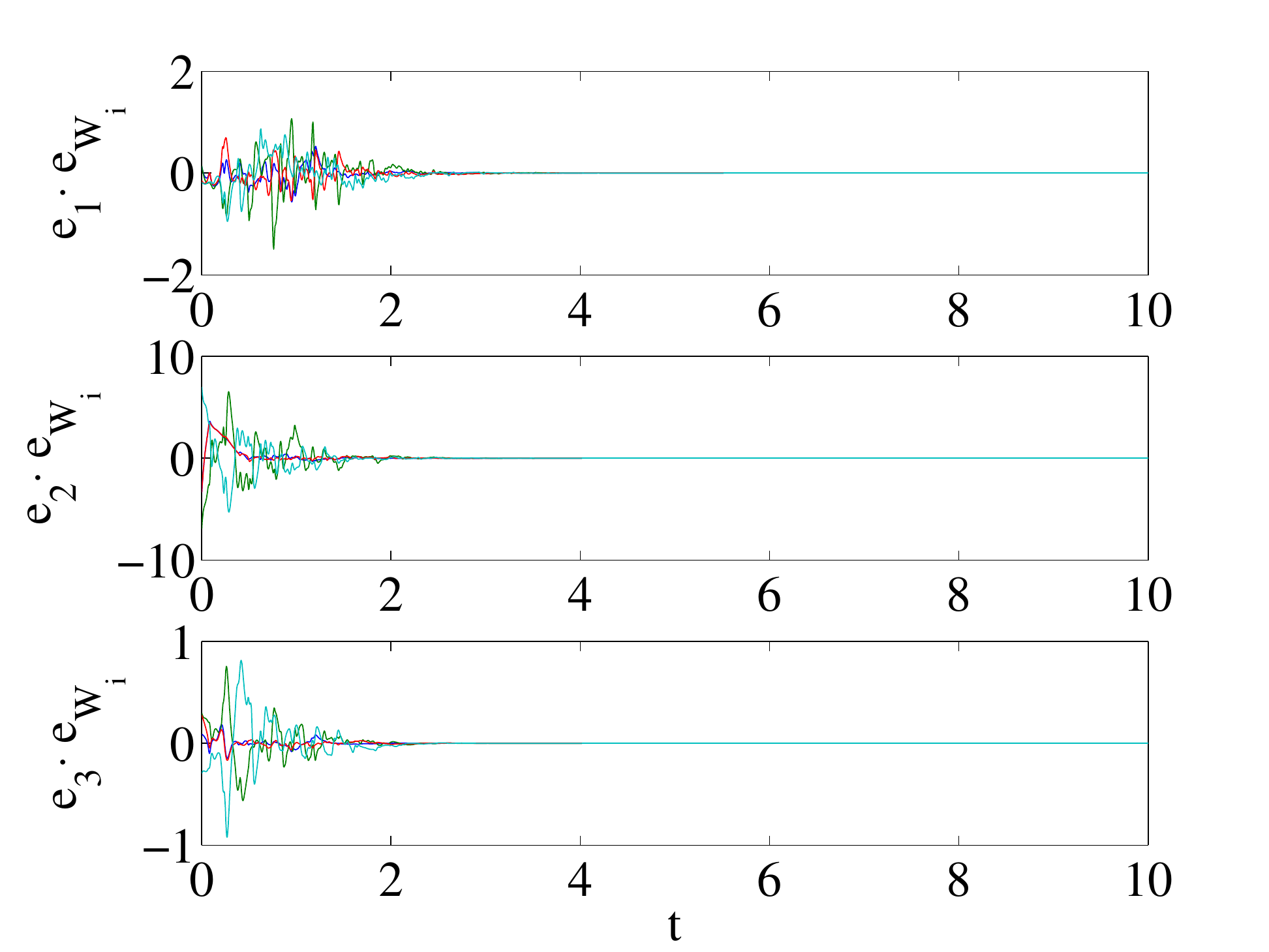}\label{fig:case2eW}}
}
\centerline{\hspace{-0.5cm}
	\subfigure[Payload attitude error $\psi_{0}$]{
		\includegraphics[width=0.5\columnwidth]{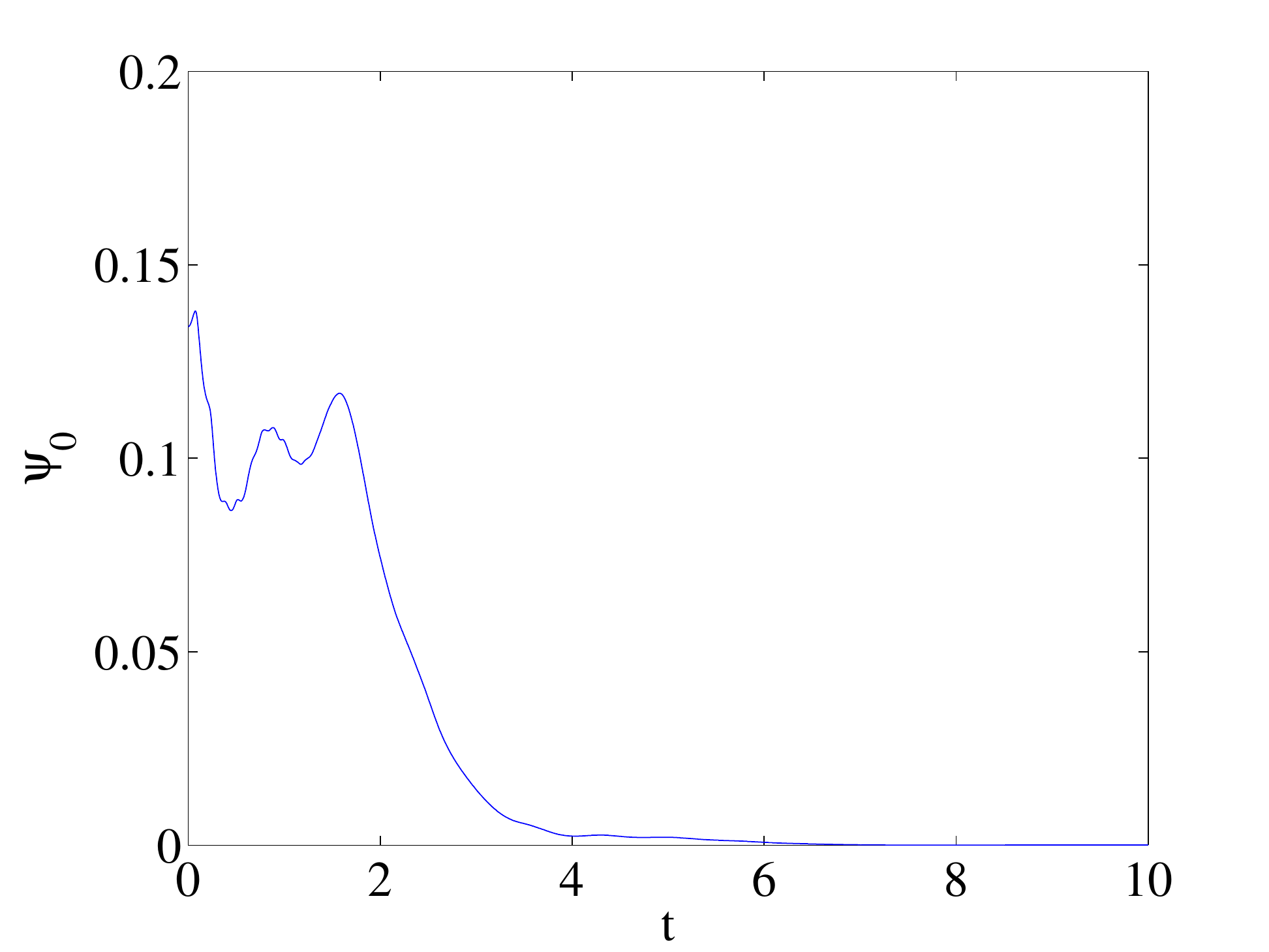}\label{fig:case2psi0}}
	\subfigure[Quadrotors attitude errors $\psi_{i}$]{
		\includegraphics[width=0.5\columnwidth]{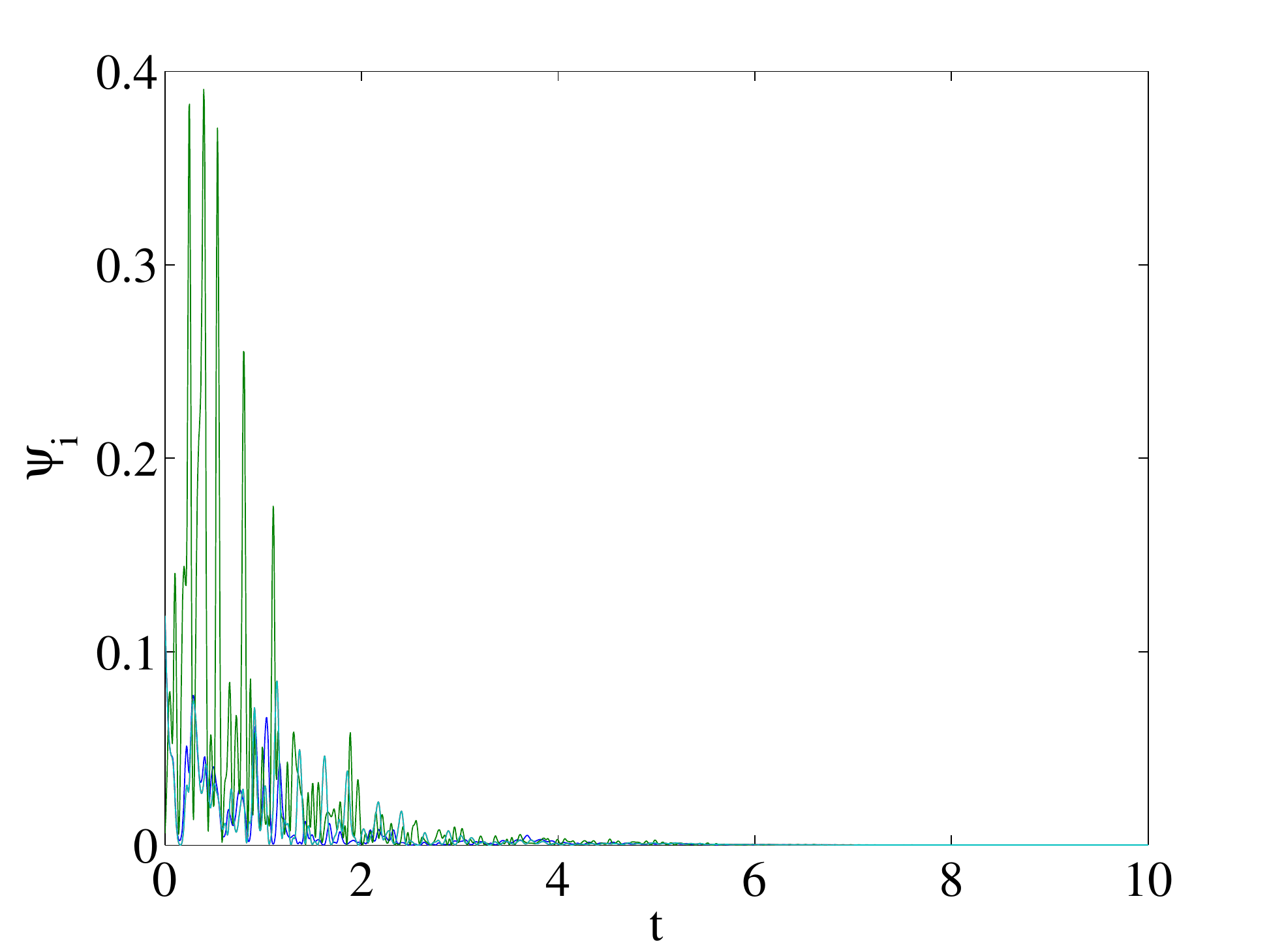}\label{fig:case2psii}}
}
\centerline{\hspace{-0.5cm}
	\subfigure[Quadrotors total thrust inputs $f_{i}$]{
		\includegraphics[width=0.5\columnwidth]{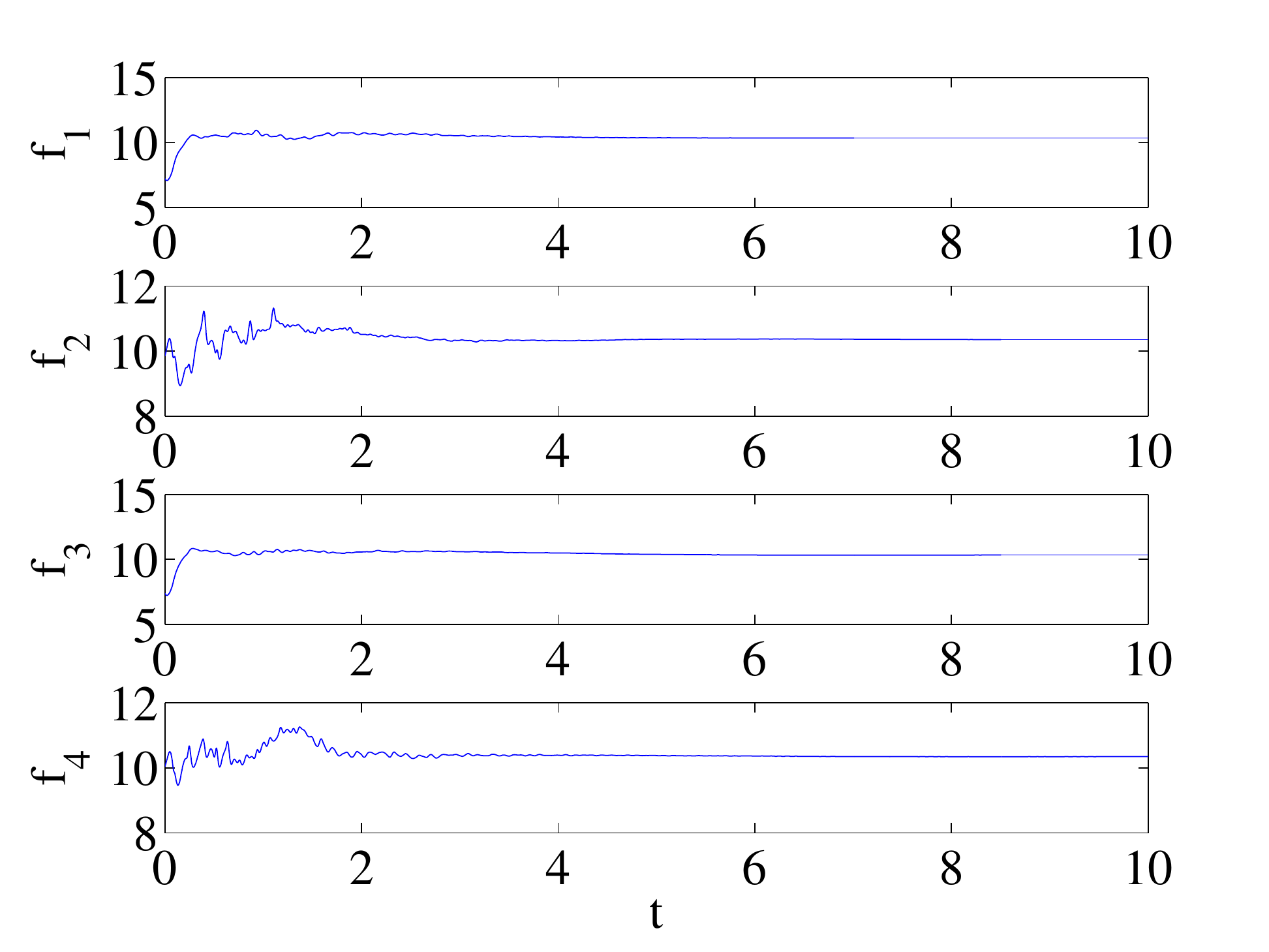}\label{fig:case2ui}}
	\subfigure[Direction error $e_{q}$, and angular velocity error $e_{\omega}$ for the links]{
		\includegraphics[width=0.5\columnwidth]{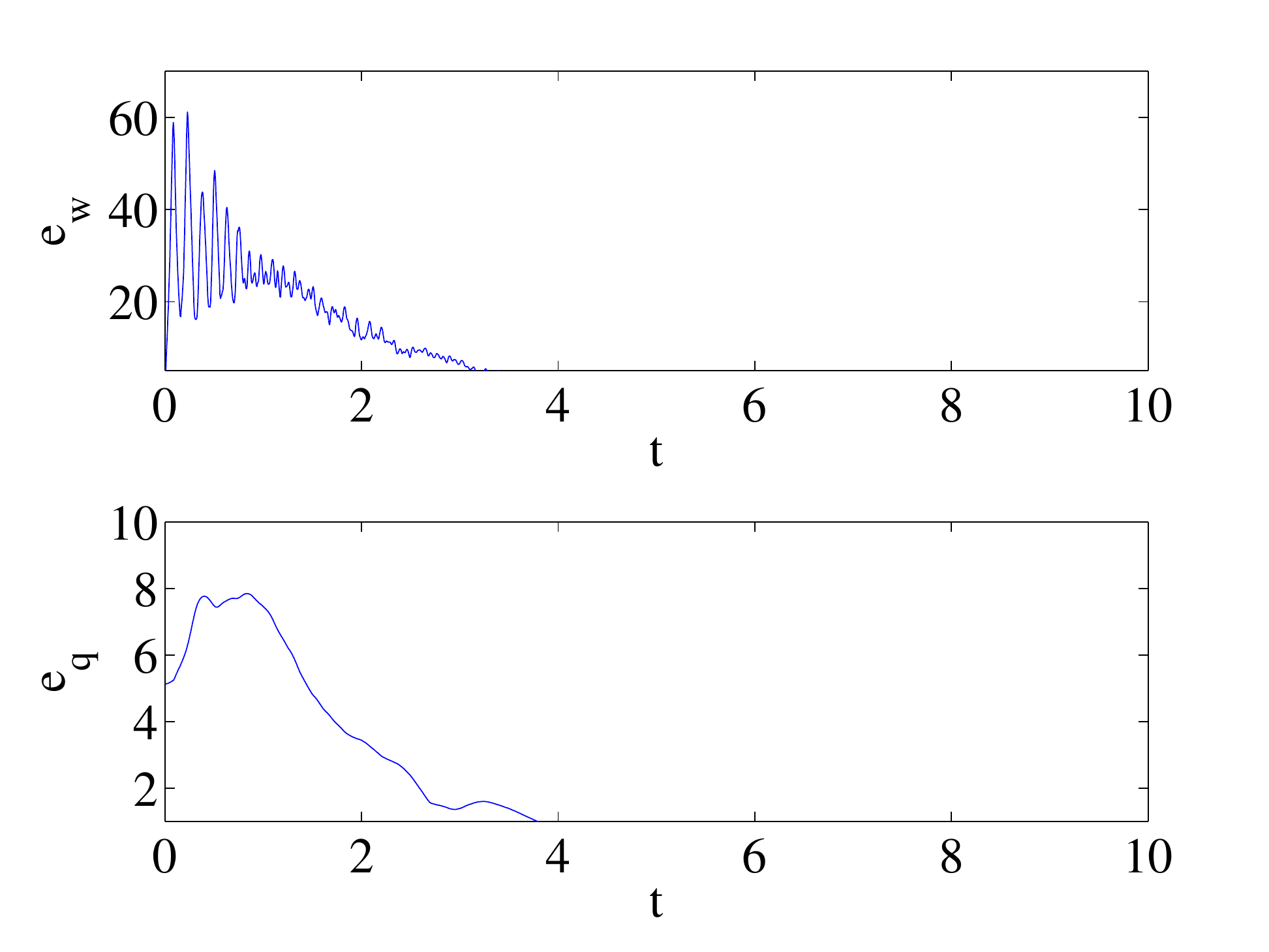}\label{fig:case2errors}}
}
\caption{Stabilization of a payload with multiple quadrotors connected with flexible cables.}\label{fig:simresults2}
\end{figure}
\subsection{Payload Stabilization with Large Initial Attitude Errors}
In the second case, we consider large initial errors for the attitude of the payload and quadrotors. Initially, the rigid body is tilted about its $b_{1}$ axis by $30$ degrees, and the initial direction of the links are chosen such that two cables are curved along the horizontal direction. The initial conditions are given by
\begin{gather*}
x_{0}(0)=[2.4,\; 0.8,\; -1.0]^{T},\; v_{0}(0)=0_{3\times 1},\\
\omega_{ij}(0)=0_{3\times 1},\; \Omega_{i}(0)=0_{3\times 1}\\
R_{0}(0)=R_{x}(30^\circ),\; \Omega_{0}=0_{3\times 1},
\end{gather*}
where $R_x(30^\circ)$ denotes the rotation about the first axis by $30^\circ$. The initial attitude of quadrotors are chosen as
\begin{gather*}
R_{1}(0)=R_{y}(-35^\circ),\; R_{2}(0)=I_{3\times 3},\\ 
R_{3}(0)=R_{y}(-35^\circ),\; R_{4}(0)=I_{3\times 3}.
\end{gather*}
The mass properties of quadrotors are chosen as pervious example. The payload is a box with mass $m_{0}=0.5\,\mathrm{kg}$, and its length, width, and height are $0.6$, $0.8$, and $0.2\,\mathrm{m}$, respectively. Each cable connecting the rigid body to the $i$-th quadrotor is considered to be $n_{i}=5$ rigid links. All the links have the same mass of $m_{ij}=0.01\,\mathrm{kg}$ and length of $l_{ij}=0.15\,\mathrm{m}$. Each cable is attached to the following points of the payload
\begin{gather*}
\rho_{1}=[0.3,\; -0.4,\; -0.1]^T\,\mathrm{m},\; \rho_{2}=[0.3,\; 0.4,\; -0.1]^T\,\mathrm{m},\\
\rho_{3}=[-0.3,\; -0.4,\; -0.1]^T\,\mathrm{m},\; \rho_{4}=[-0.3,\; 0.4,\; -0.1]^T\,\mathrm{m}.
\end{gather*}

The payload mass is $m=1.0\,\mathrm{kg}$ , and its length, width, and height are $1.0$, $1.2$, and $0.2\,\mathrm{m}$, respectively.

\begin{figure}[h]
\centerline{
	\subfigure[$t=0$ Sec.]{
		\includegraphics[width=0.3\columnwidth]{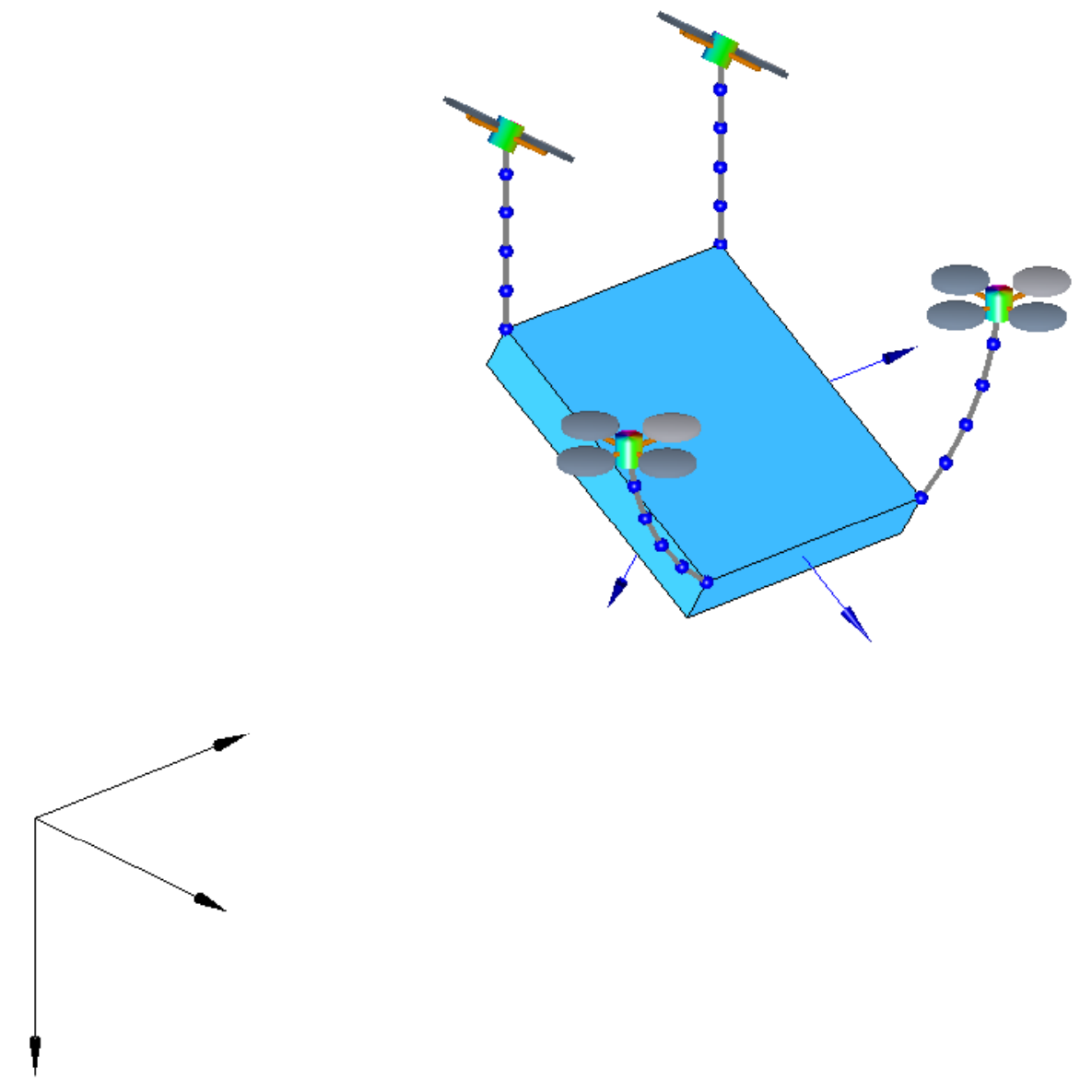}}
	\subfigure[$t=0.14$ Sec.]{
		\includegraphics[width=0.3\columnwidth]{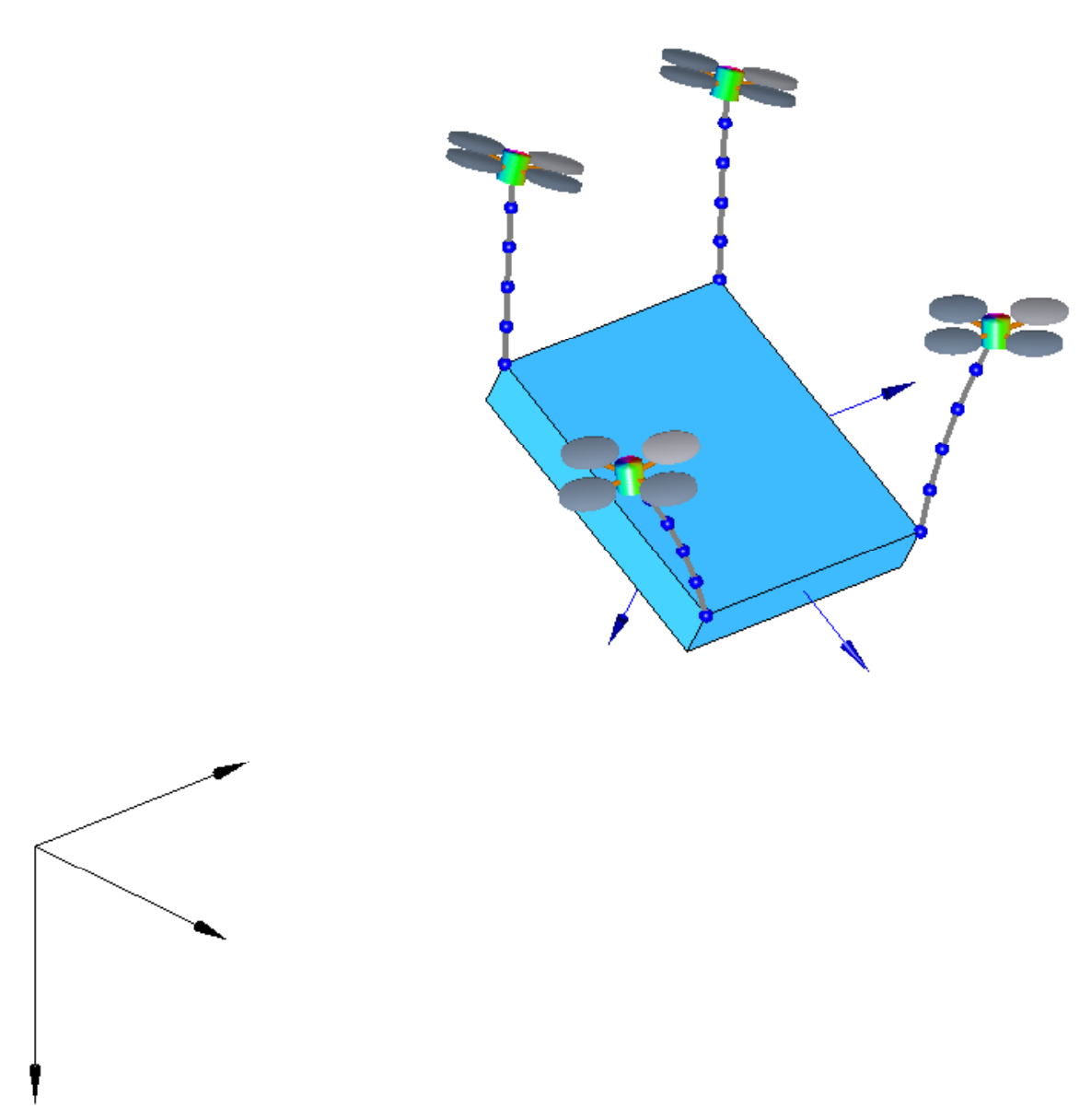}}
		\subfigure[$t=0.30$ Sec.]{
		\includegraphics[width=0.3\columnwidth]{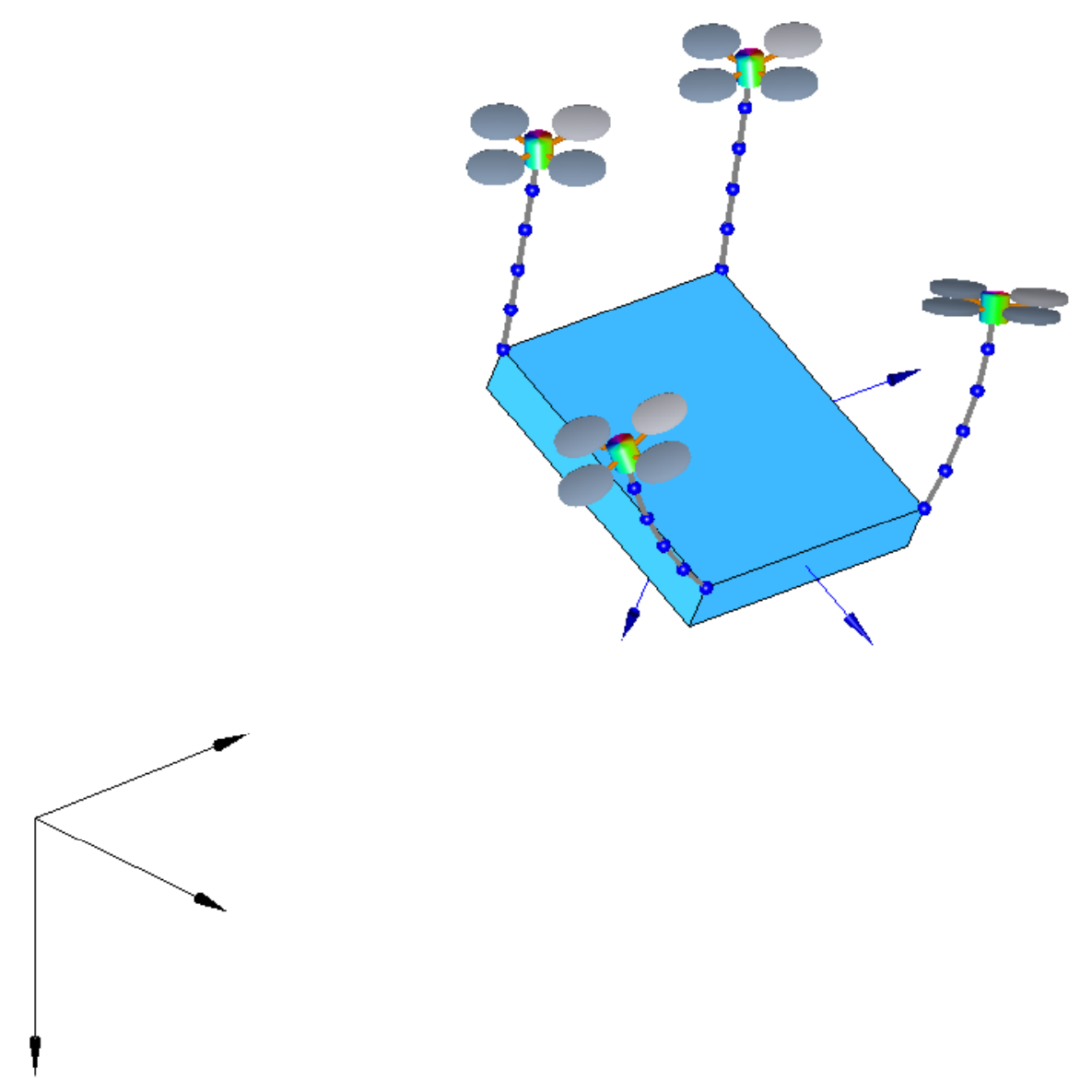}}
}
\centerline{
	\subfigure[$t=0.68$ Sec.]{
		\includegraphics[width=0.3\columnwidth]{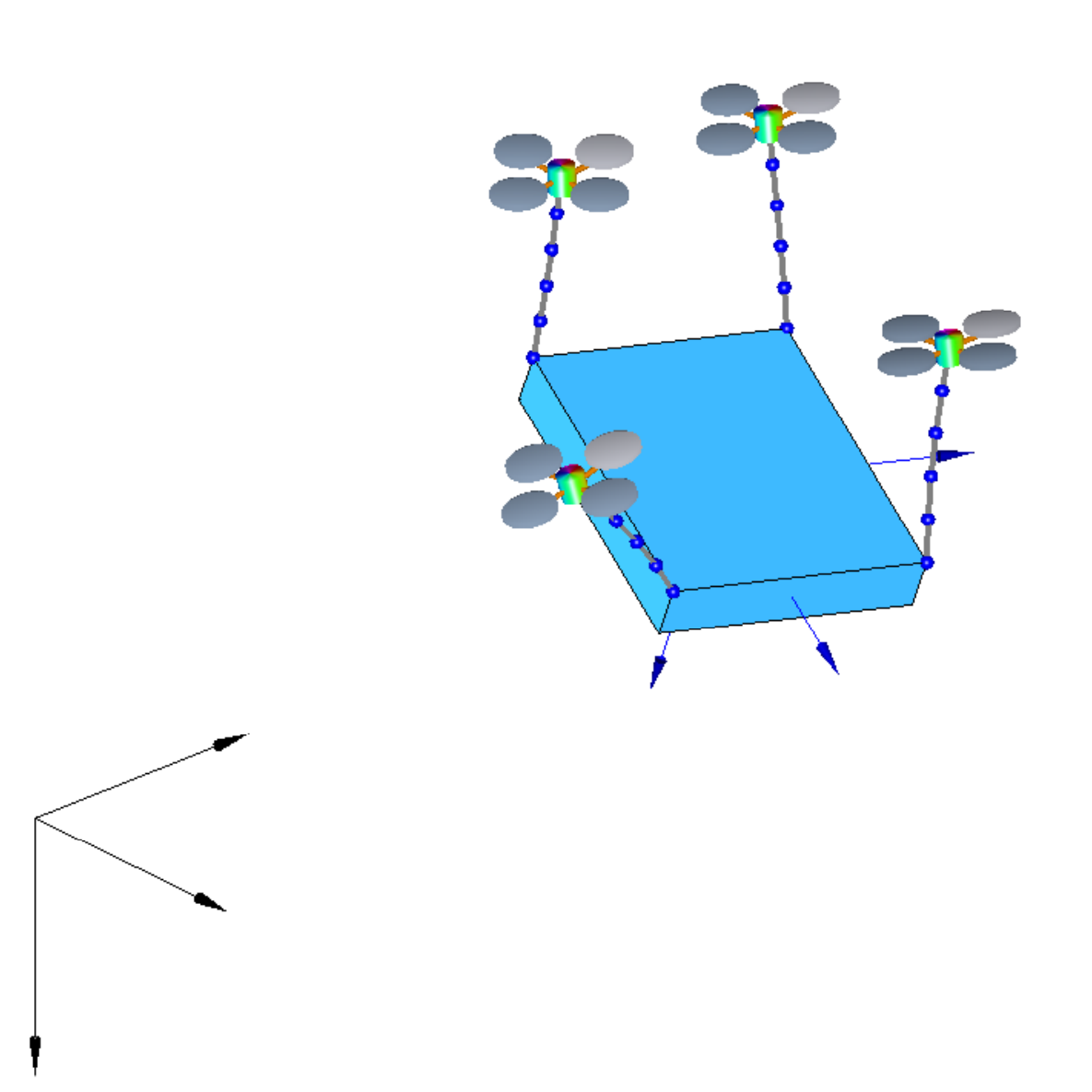}}
	\subfigure[$t=1.10$ Sec.]{
		\includegraphics[width=0.3\columnwidth]{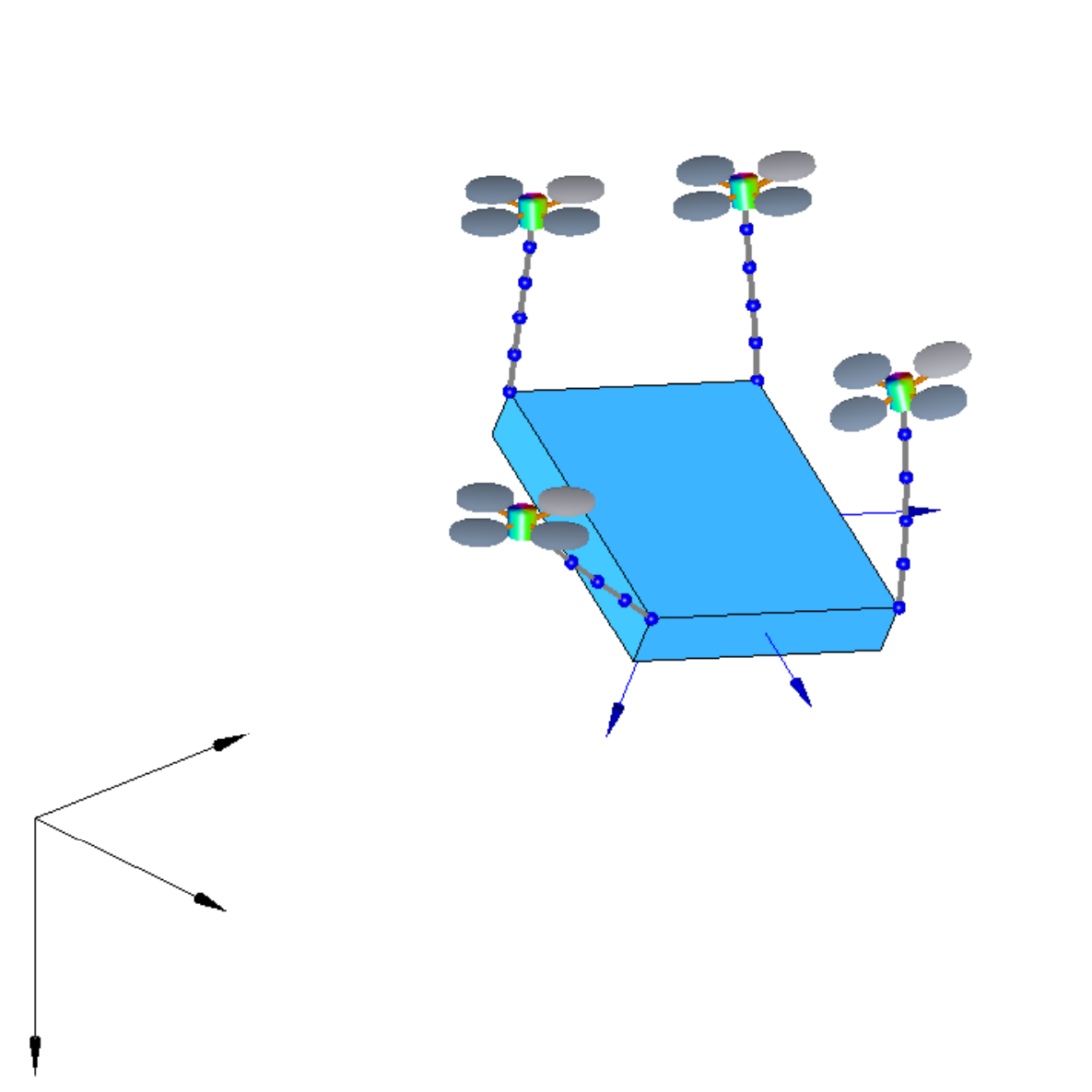}}
		\subfigure[$t=1.36$ Sec.]{
		\includegraphics[width=0.3\columnwidth]{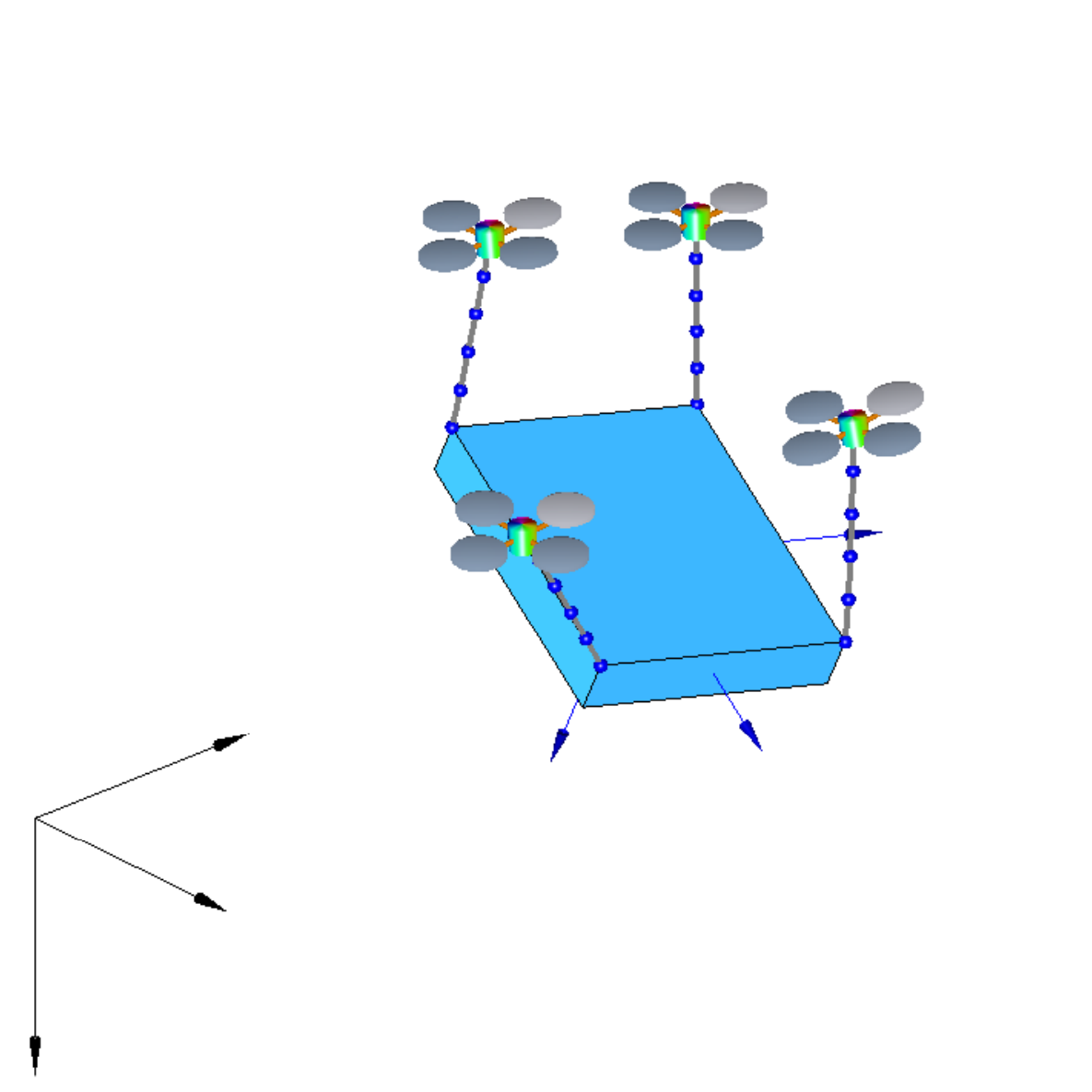}}
}
\centerline{
	\subfigure[$t=1.98$ Sec.]{
		\includegraphics[width=0.3\columnwidth]{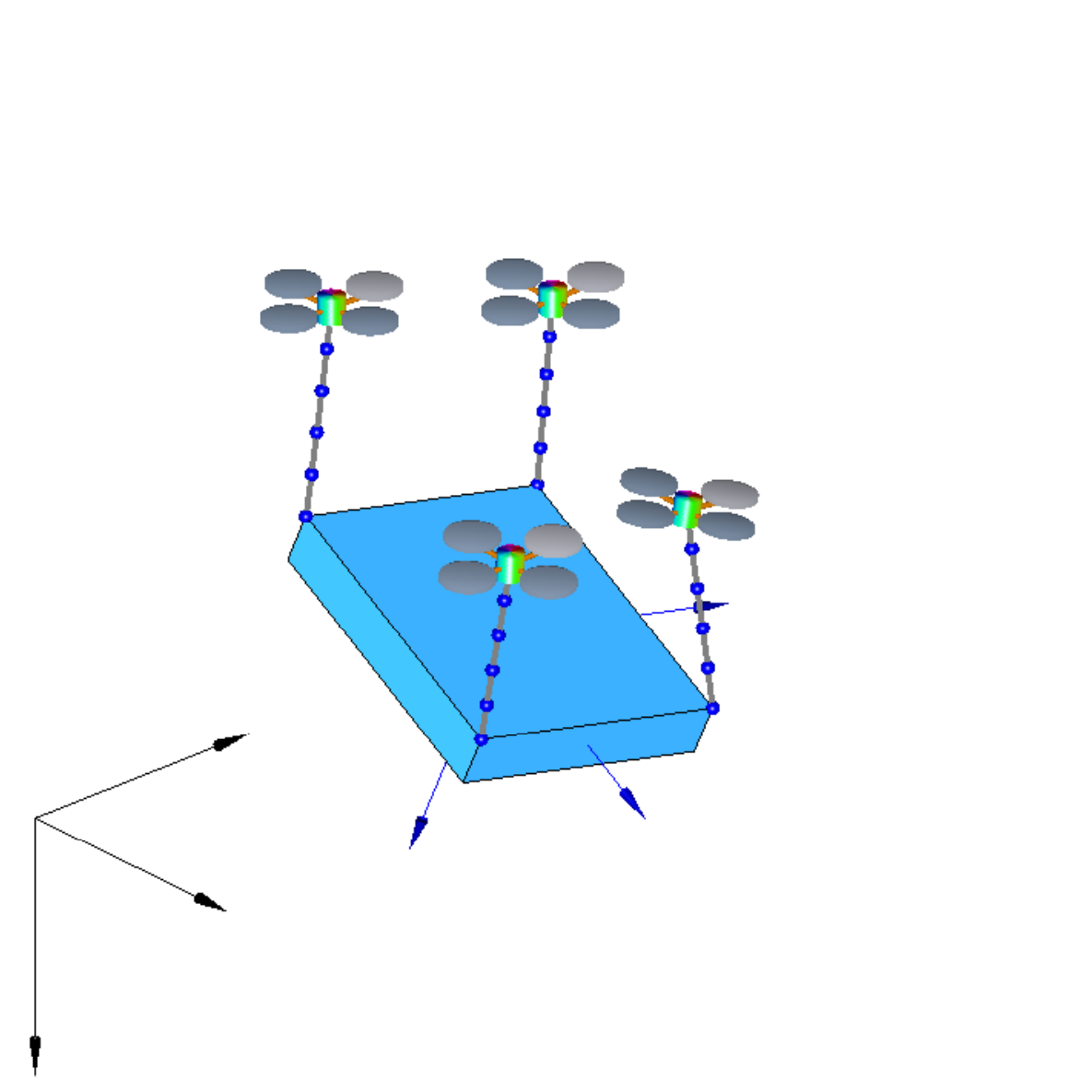}}
	\subfigure[$t=3.48$ Sec.]{
		\includegraphics[width=0.3\columnwidth]{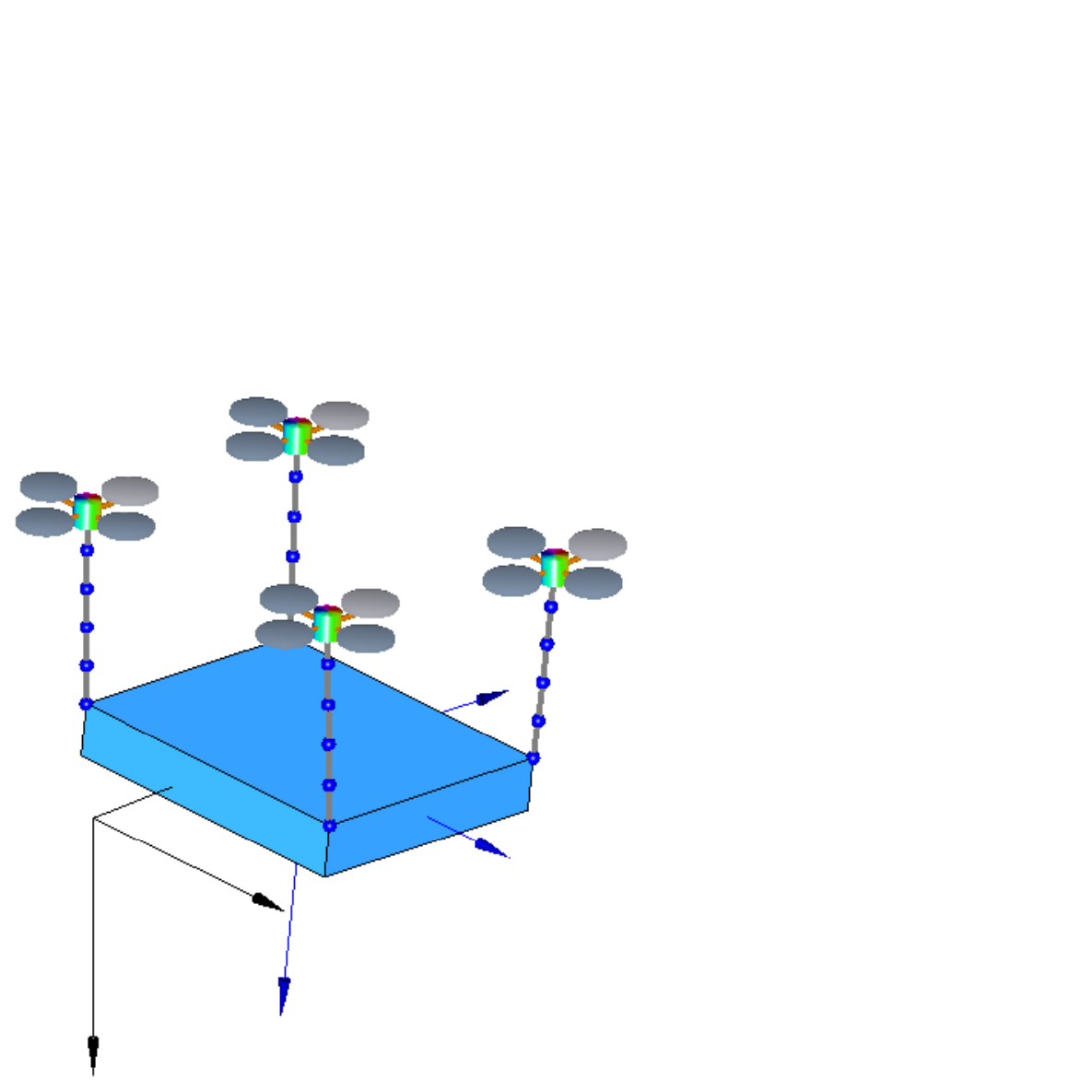}}
		\subfigure[$t=10$ Sec.]{
		\includegraphics[width=0.3\columnwidth]{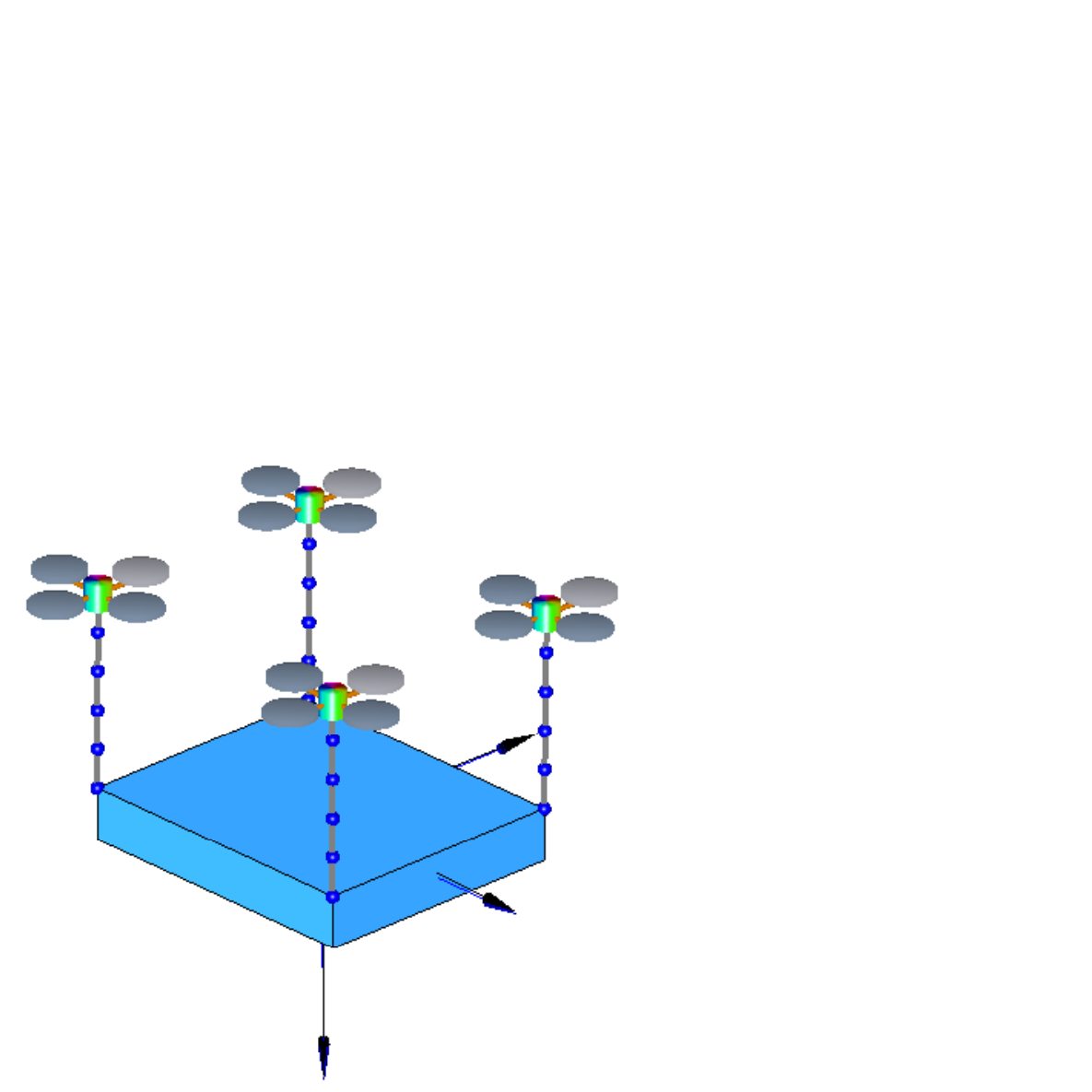}}
}
\caption{Snapshots of the controlled maneuver. A short animation is also available at  \href{https://www.youtube.com/watch?v=j14tSuHd8oA}{https://www.youtube.com/watch?v=j14tSuHd8oA}}
\label{fig:simresults3snap}
\end{figure}

Figure \ref{fig:simresults2} illustrates the tracking errors, and the total thrust of each quadrotor. Snapshots of the controlled maneuvers is also illustrated at Figure \ref{fig:simresults3snap}. It is shown that the proposed controller is able to stabilize the payload and cables at their desired configuration even from the large initial attitude errors.

\section{EXPERIMENT}\label{sec:ER}
In this section, an experimental setup is described and the proposed geometric nonlinear controller is validated with experiments.
\subsection{Hardware Description}\label{HD}
The quadrotor UAV developed at the flight dynamics and control laboratory at the George Washington University is shown at Figure~\ref{fig:Quad}, and its parameters are the same as described as the previous section. 
The angular velocity is measured from inertial measurement unit (IMU) and the attitude is obtained from IMU data. Position of the UAV is measured from motion capture system (Vicon) and the velocity is estimated from the measurement. Ground computing system receives the Vicon data and send it to the UAV via XBee. The Gumstix is adopted as micro computing unit on the UAV. The flight control software has three main threads, namely Vicon thread, IMU thread, and control thread. The Vicon thread receives the Vicon measurement and estimates linear velocity of the quadrotor. In IMU thread, it receives the IMU measurement and estimates the attitude. The last thread handles the control outputs at each time step. Also, control outputs are calculated at 120Hz which is fast enough to run any kind of aggressive maneuvers. Information flow of the system is illustrated in Figure \ref{fig:information_flow}. 

We developed an accurate CAD model as shown in Figure~\ref{fig:cadmodel} to identify several parameters of the quadrotor, such as moment of inertia and center of mass. Furthermore, a precise rotor calibration is performed for each rotor, with a custom-made thrust stand as shown in Figure~\ref{fig:stand} to determine the relation between the command in the motor speed controller and the actual thrust. For various values of motor speed commands, the corresponding thrust is measured, and those data are fitted with a second order polynomial. 
\begin{figure}[h]
\centerline{
	\subfigure[Hardware configuration]{
\setlength{\unitlength}{0.1\columnwidth}\scriptsize
\begin{picture}(7,4)(0,0)
\put(0,0){\includegraphics[width=0.7\columnwidth]{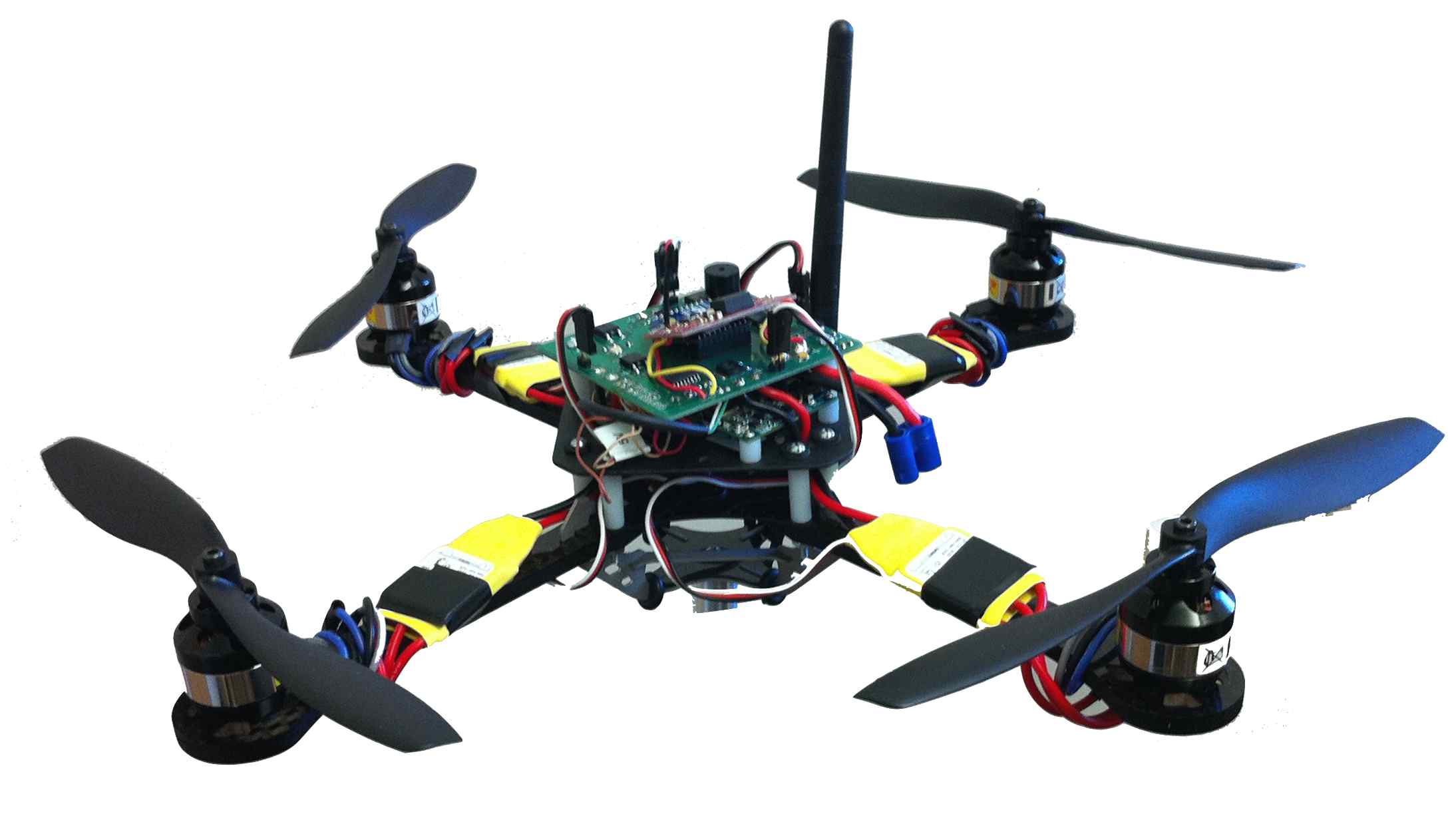}}
\put(1.95,3.2){\shortstack[c]{OMAP 600MHz\\Processor}}
\put(2.3,0){\shortstack[c]{Attitude sensor\\3DM-GX3\\ via UART}}
\put(0.85,1.4){\shortstack[c]{BLDC Motor\\ via I2C}}
\put(0.1,2.5){\shortstack[c]{Safety Switch\\XBee RF}}
\put(4.3,3.2){\shortstack[c]{WIFI to\\Ground Station}}
\put(5,2.0){\shortstack[c]{LiPo Battery\\11.1V, 2200mAh}}
\end{picture}\label{fig:Quad}}
	\subfigure[Motor calibration setup]{
	\includegraphics[width=0.27\columnwidth]{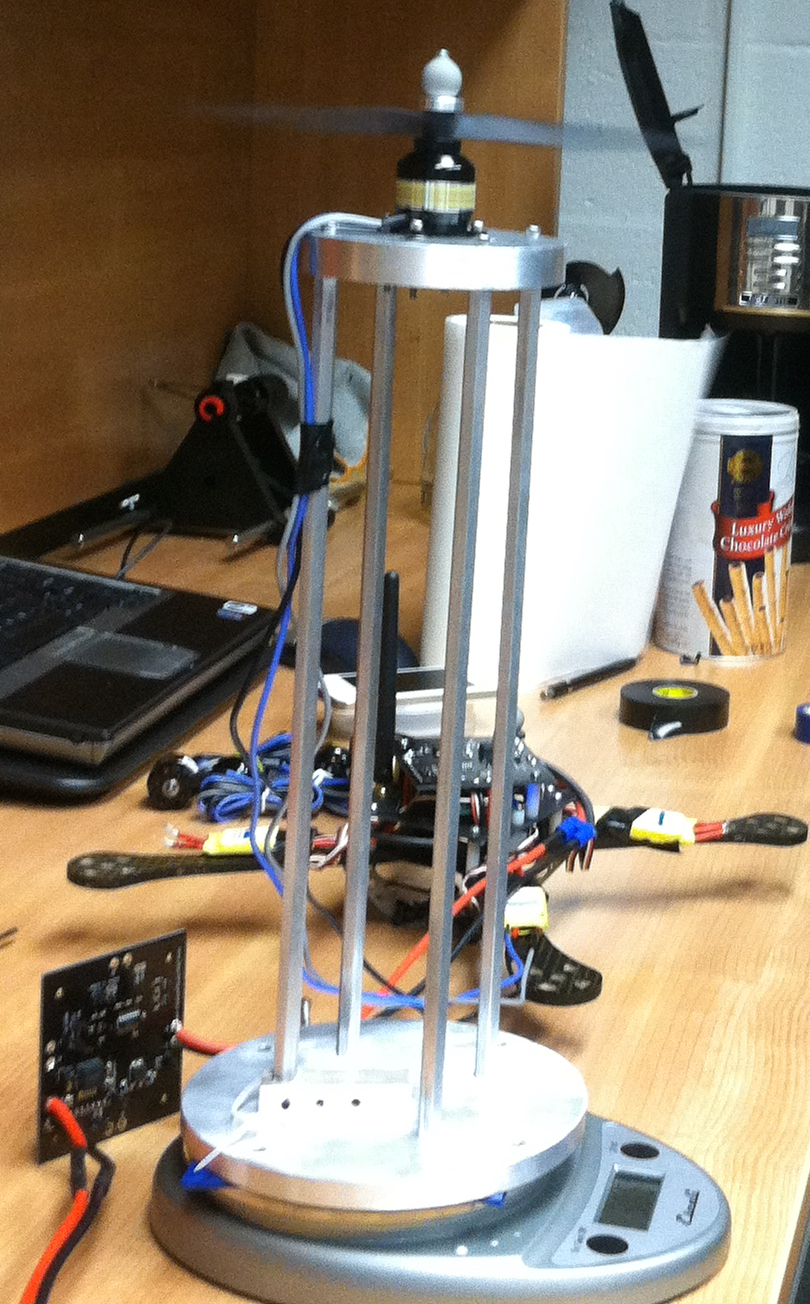}\label{fig:stand}}
}
\caption{Hardware development for each quadrotor UAV}
\end{figure}
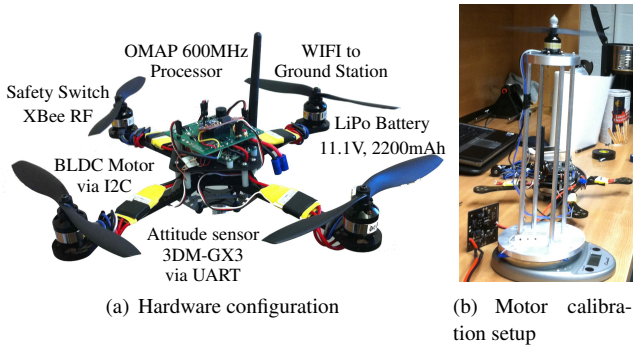
\begin{figure}[h]
\centerline{
	\includegraphics[width=0.8\columnwidth]{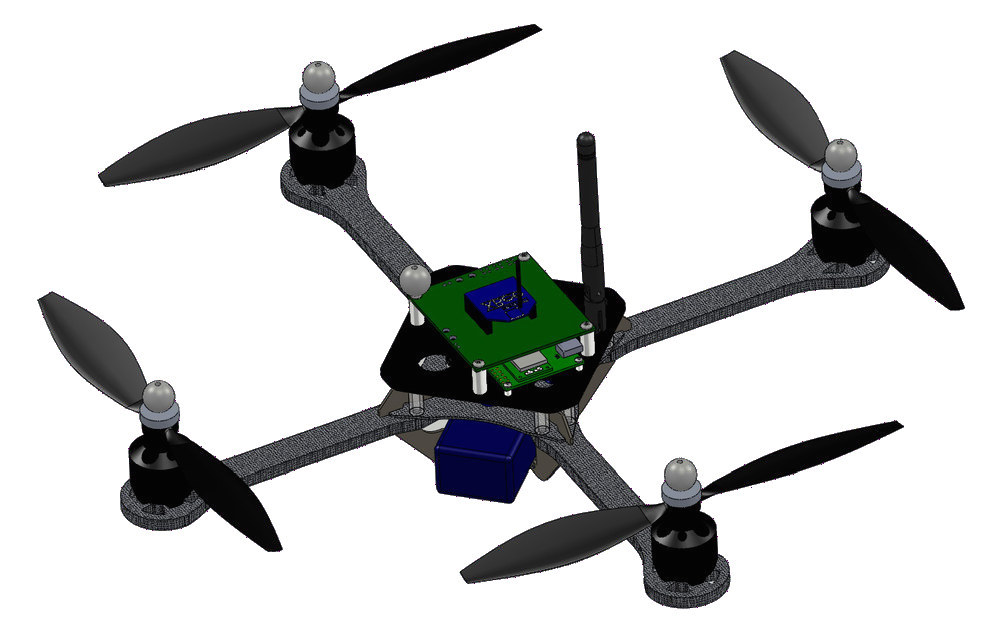}}
\caption{CAD Model}\label{fig:cadmodel}
\end{figure}
\begin{figure}[h]
\centerline{
	\includegraphics[width=1.0\columnwidth]{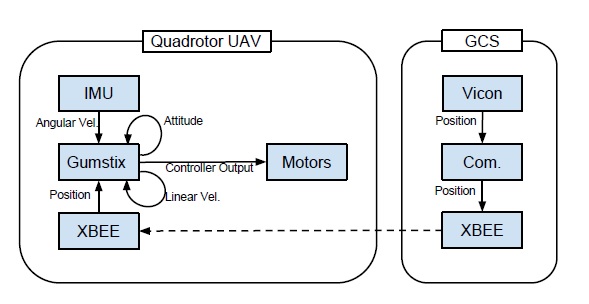}}
\caption{Information flow of overall system}\label{fig:information_flow}
\end{figure}

\subsection{Stabilizing a Rod with Two Quadrotors}
As a special case of rigid body payload, we considered a rod as a payload for experiment as shown in the Figures~\ref{fig:experimentstarts}. Two quadrotors are enough for this experiment to control the position of the payload. 

\begin{figure}[h]
\centerline{
\subfigure[Experiment]{
		\includegraphics[width=0.45\columnwidth]{Figures/fig00}}
	\subfigure[Simulation]{
		\includegraphics[width=0.5\columnwidth]{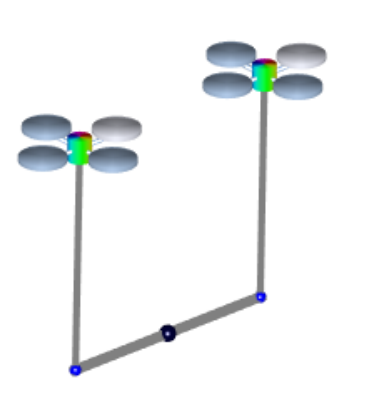}
		\put(-115,105){\shortstack[c]{$m_{1},J_{1},x_{1}$}}
		\put(-90,46){\shortstack[c]{$\downarrow q_{1},l_{1}$}}
		\put(-65,10){\shortstack[c]{$m_{0},x_{0},q_{0}\rightarrow$}}
		\put(-55,35){\shortstack[c]{$l_{0}$}}
		\put(-85,25){\shortstack[c]{$l_{0}$}}
		\put(-30,60){\shortstack[c]{$\downarrow q_{2},l_{2}$}}
		\put(-55,125){\shortstack[c]{$m_{2},J_{2},x_{2}$}}}
}
\caption{Two quadrotors transporting a rod}\label{fig:experimentstarts}
\end{figure}

We prepared a benchmark and a proposed controller case for this experiment. In both cases, two quadrotors are employed to hover at a fixed position initially while holding a rigid body rod which is at the equilibrium of the whole system. Then we utilized a wire to pull the payload and releasing it to simulate the disturbance. The performance of the proposed controller is then compared to the situation where there is no active controller specially works to stabilize the payload (benchmark). Both cables have length of $l_{1},l_{2}=1.3$ $m$ and rod has mass and length of $m_{0}=0.52$ $kg$ and $2l_{0}=2.05$ $m$ respectively. Each quadrotor has mass of $m_{1}=m_{2}=0.755$ $kg$ and the following moment of inertia which is obtained from the CAD model
\begin{align*}
J_{1},J_{2} =
\begin{bmatrix}
    5.5711 & 0.0618 & -0.0251\\
    0.06177 & 5.5757 & 0.0101\\
    -0.02502 & 0.01007 & 1.05053
\end{bmatrix}\times 10^{-2}\quad\mathrm{kgm}^2 .
 \end{align*}
The following relations and initial conditions is applied for both cases for this equilibrium condition
\begin{gather}
q_{0}(0)=e_{1},\quad q_{1}(0)=q_{2}(0)=e_{3},\nonumber\\
x_{1}(0)=x_{0}(0)-l_{0}q_{0}(0)-l_{1}q_{1}(0),\nonumber\\
x_{2}(0)=x_{0}(0)+l_{0}q_{0}(0)-l_{1}q_{2}(0),\nonumber\\
R_{1}=R_{2}=I_{3},\nonumber\\ 
\Omega_{1}=\Omega_{2}=\Omega_{0}=0,\nonumber
\end{gather}
\subsubsection{Benchmark}
We choose this case as a comparison benchmark and it's not the result of the proposed controller in this paper. In this case, two quadrotors are hovering above the payload while cables are aligned to the vertical direction and we apply a disturbance to the payload. Here, the cables and payload dynamics are not considered into the control system and geometric nonlinear controller is used for each quadrotor to maintain quadrotors hovering at the fixed position and considers the forces applied to the quadrotors form the payload as disturbance. The payload is pulled by $30^{\circ}$ in the direction of $y$-axis of inertial frame and releases. The payload is oscillating bellow the quadrotors and forces applies to each quadrotor from the cables are just considered as disturbances to the quadrotors, so we experience large oscillations of the payload and cables.

\begin{figure}[h]
\centerline{
	\subfigure[Rod's actual and desired positions, $x_{0}$, $x_{0_{d}}$]{
		\includegraphics[width=0.5\columnwidth]{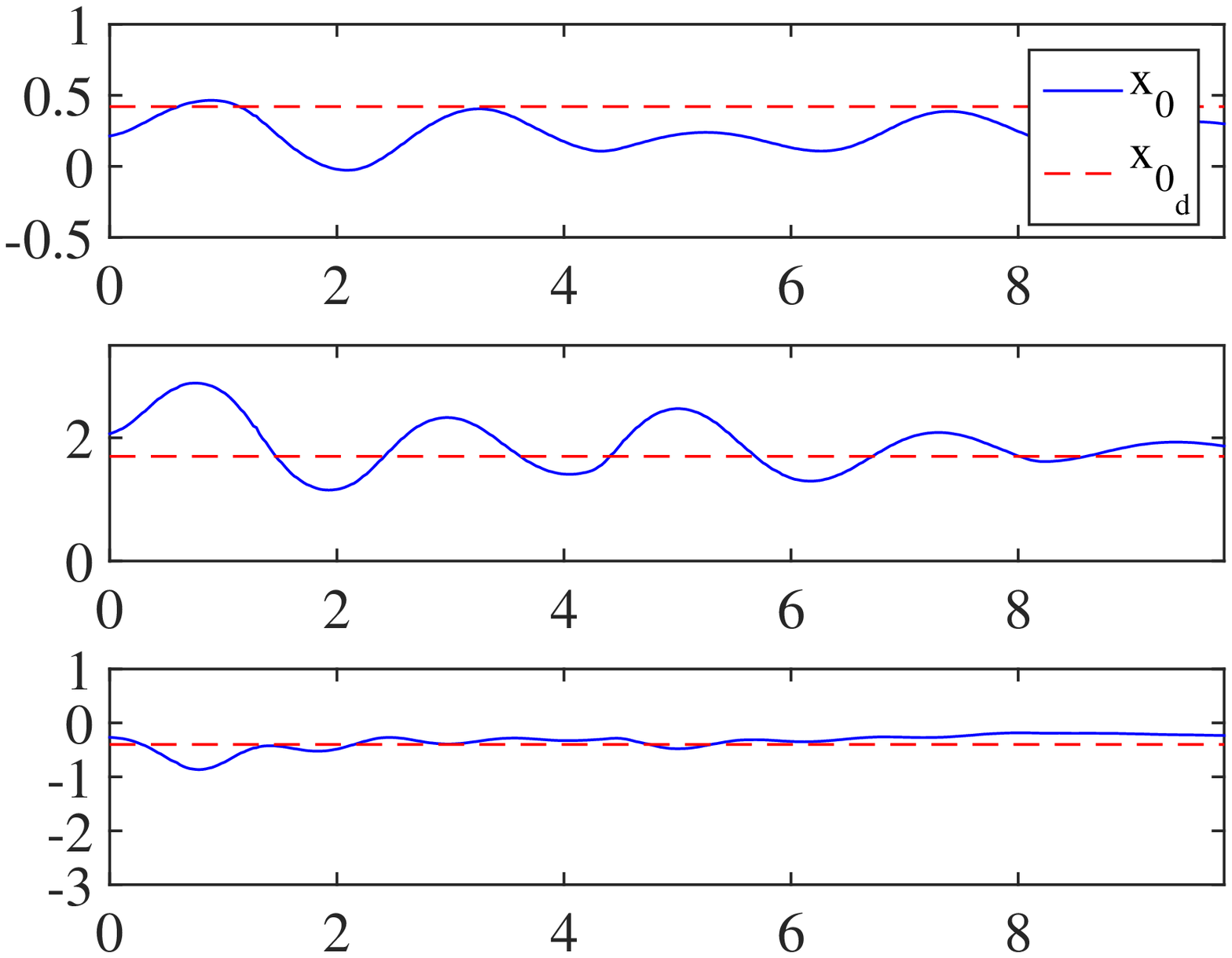}}
	\subfigure[Quadrotor's positions, $x_{1}$, $x_{2}$]{
		\includegraphics[width=0.5\columnwidth]{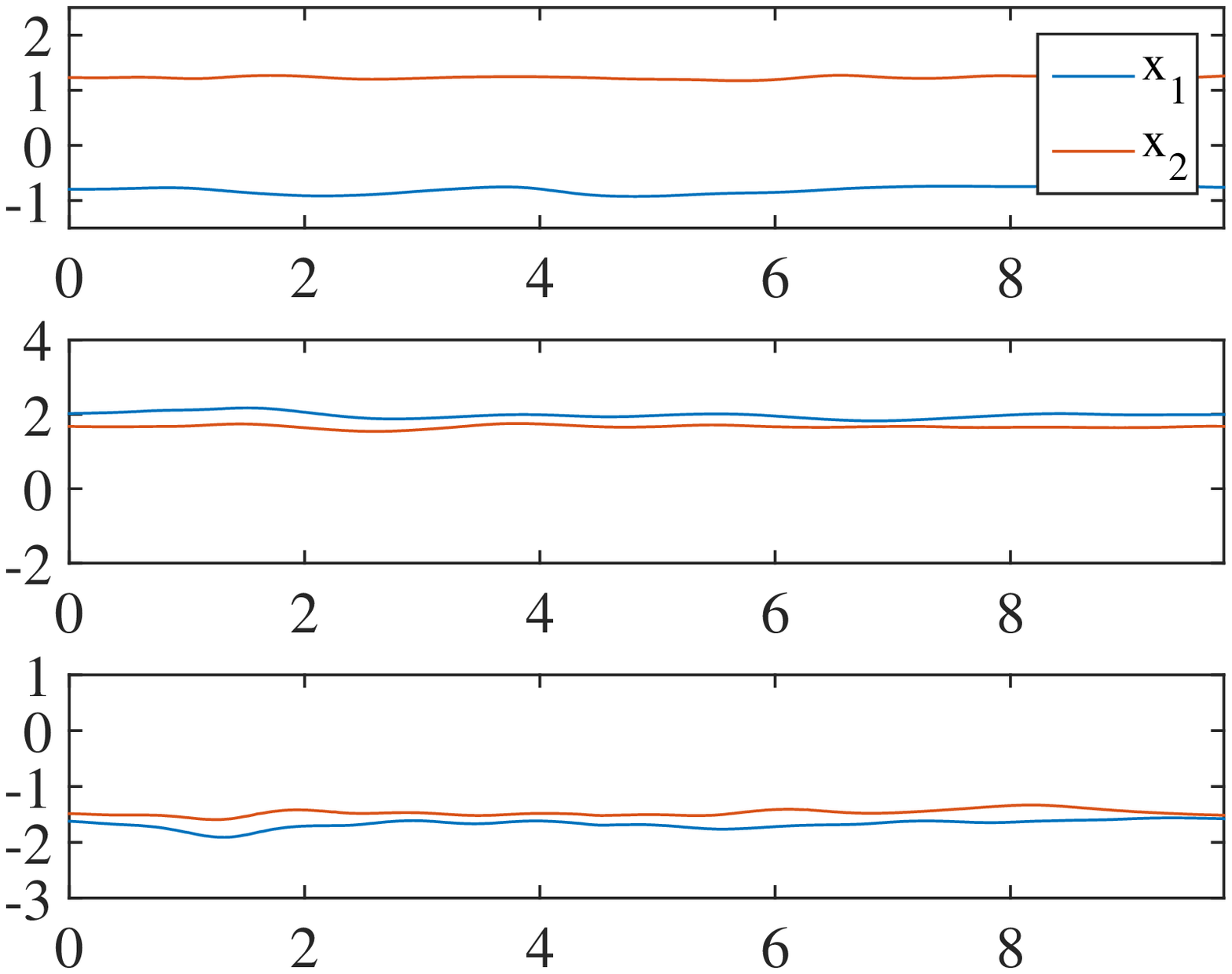}}
}
\centerline{
	\subfigure[First cable's direction, $q_{1}$]{
		\includegraphics[width=0.5\columnwidth]{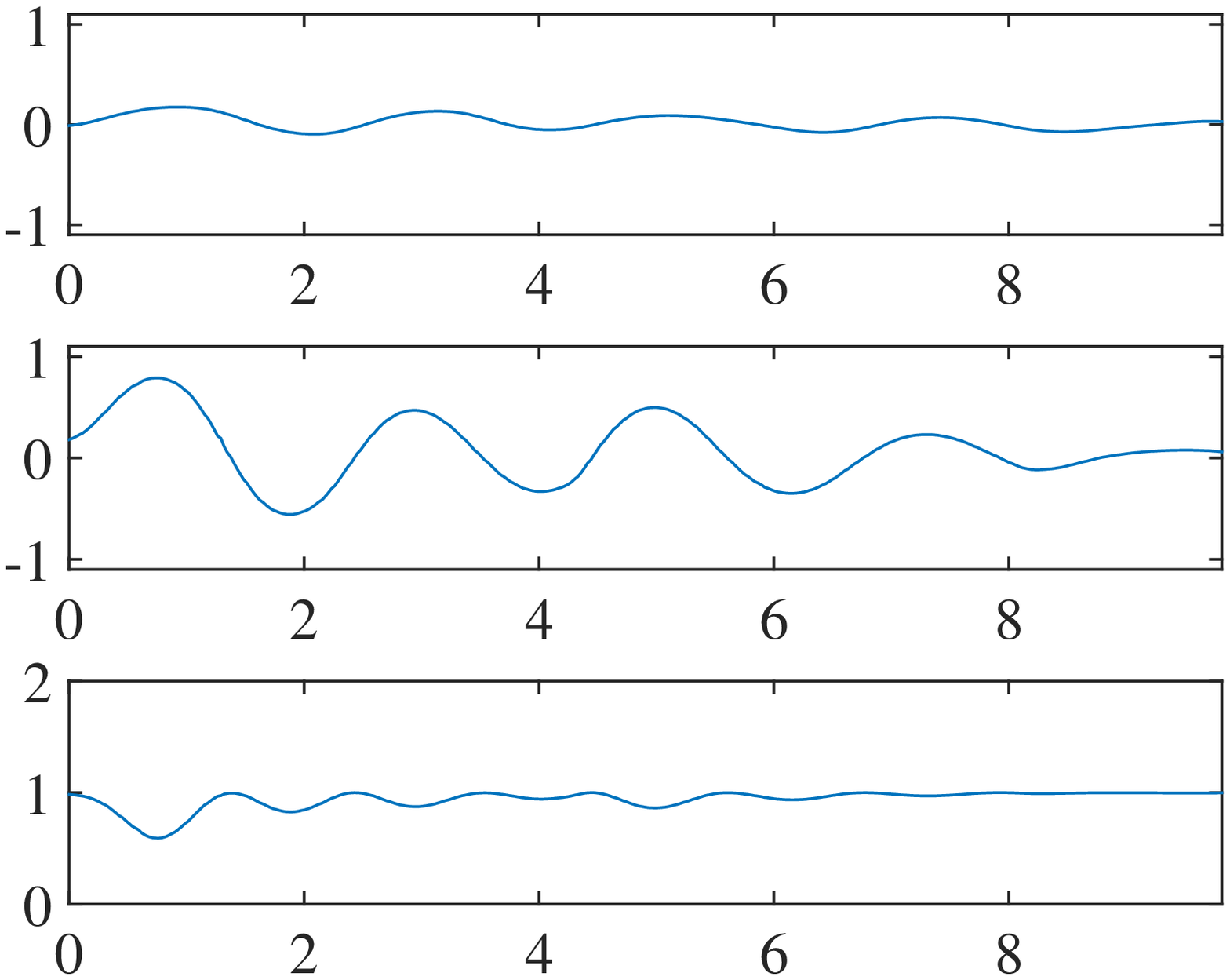}}
	\subfigure[Second cable's direction, $q_{2}$]{
		\includegraphics[width=0.5\columnwidth]{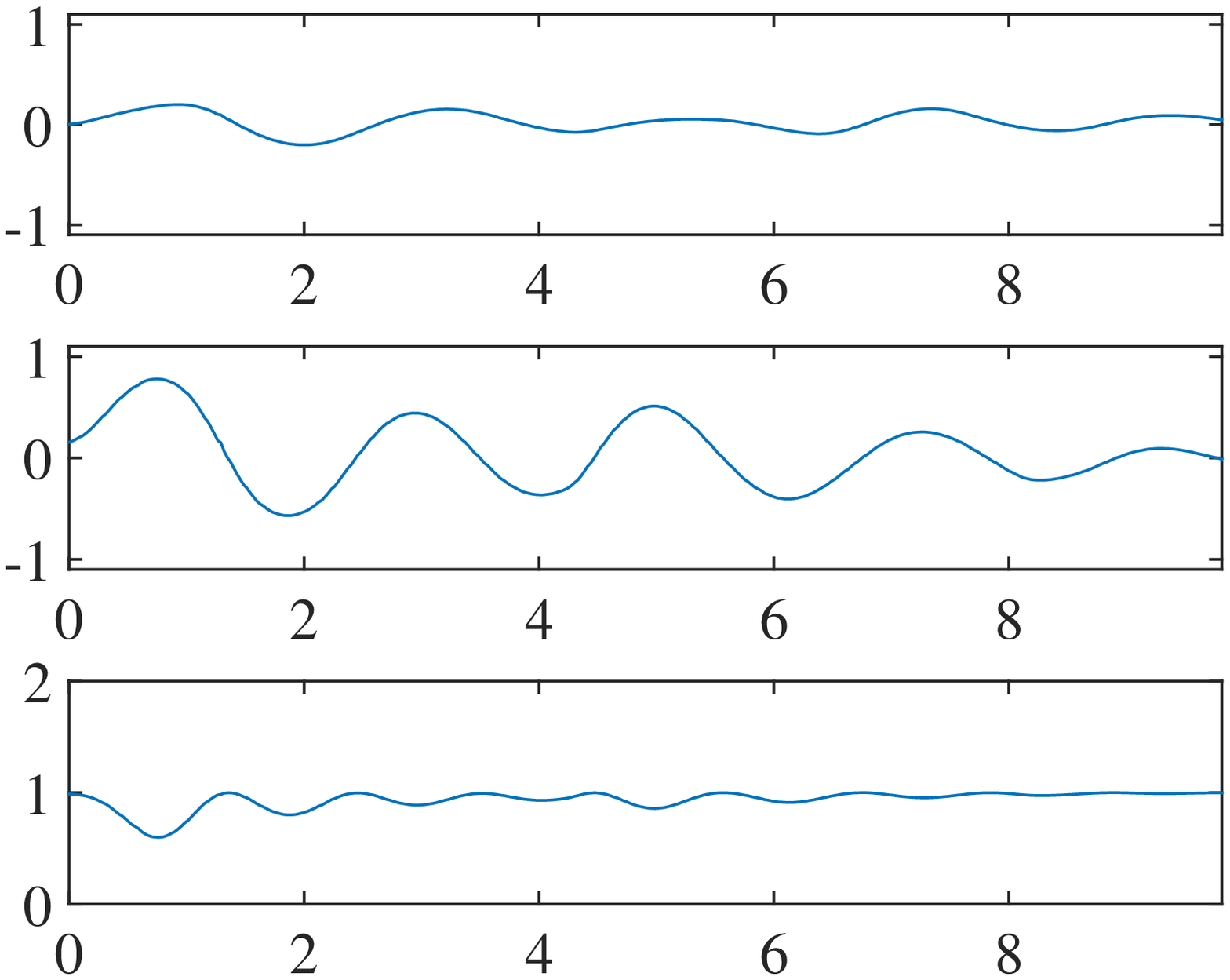}}
}
\caption{Benchmark: quadrotor position control system~\cite{Farhad2013}}\label{fig:expcase1}
\end{figure}

Numerical results for this experiment are presented in Figures~\ref{fig:expcase1} which presents the first and second's link directions, position of the payload and positions of the quadrotors during this test. 

\subsubsection{Proposed Dynamical System and Controller}
In this case, quadrotors are hovering at a fixed position using the geometric nonlinear controller while holding the payload. The payload is pulled with an external wire up to $30^\circ$ angle same as the first case and then releases. Then, the proposed controller is switched in to stabilize the system. In this scenario, dynamic of the cables and payload are considered into the control system and quadrotors cooperatively work to stabilize the payload to the desired fixed position while aligning the cables in the vertical directions and the payload in the desired direction of first inertial axis, $e_{1}$.

\begin{figure}[h]
\centerline{
	\subfigure[Rod's actual and desired positions, $x_{0}$, $x_{0_{d}}$]{
		\includegraphics[width=0.5\columnwidth]{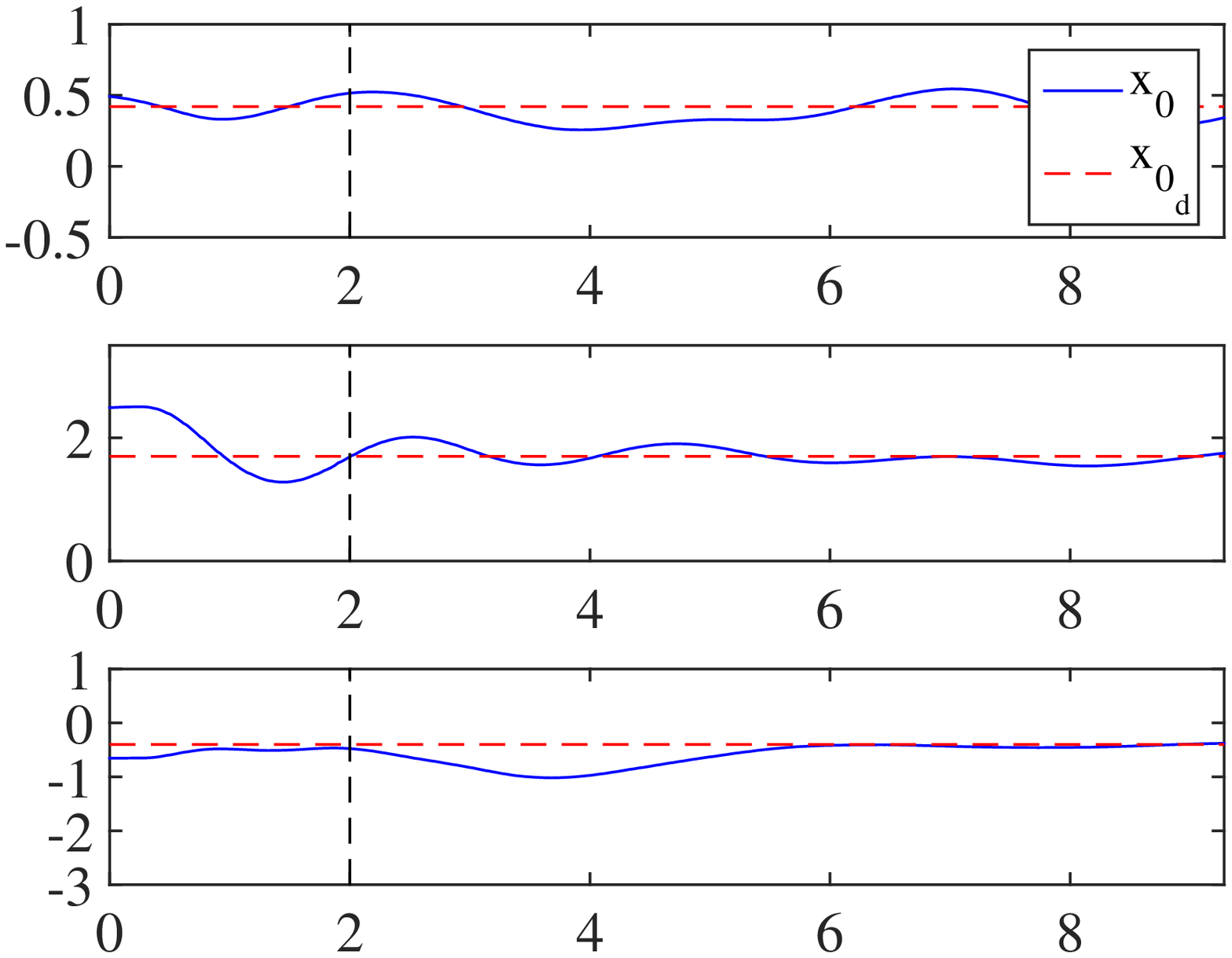}}
	\subfigure[Quadrotor's positions, $x_{1}$, $x_{2}$]{
		\includegraphics[width=0.5\columnwidth]{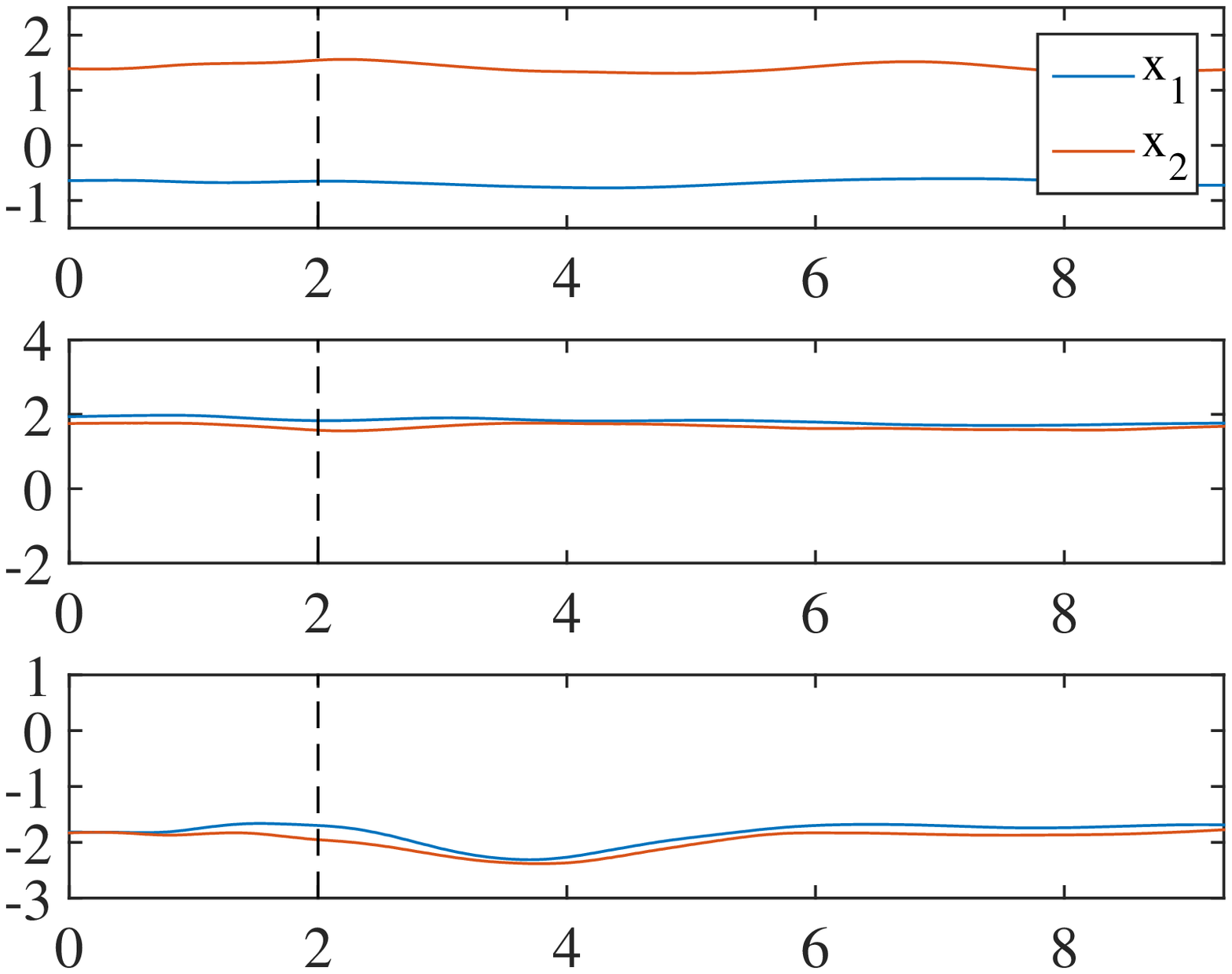}}
}
\centerline{
	\subfigure[First cable's direction, $q_{1}$]{
		\includegraphics[width=0.5\columnwidth]{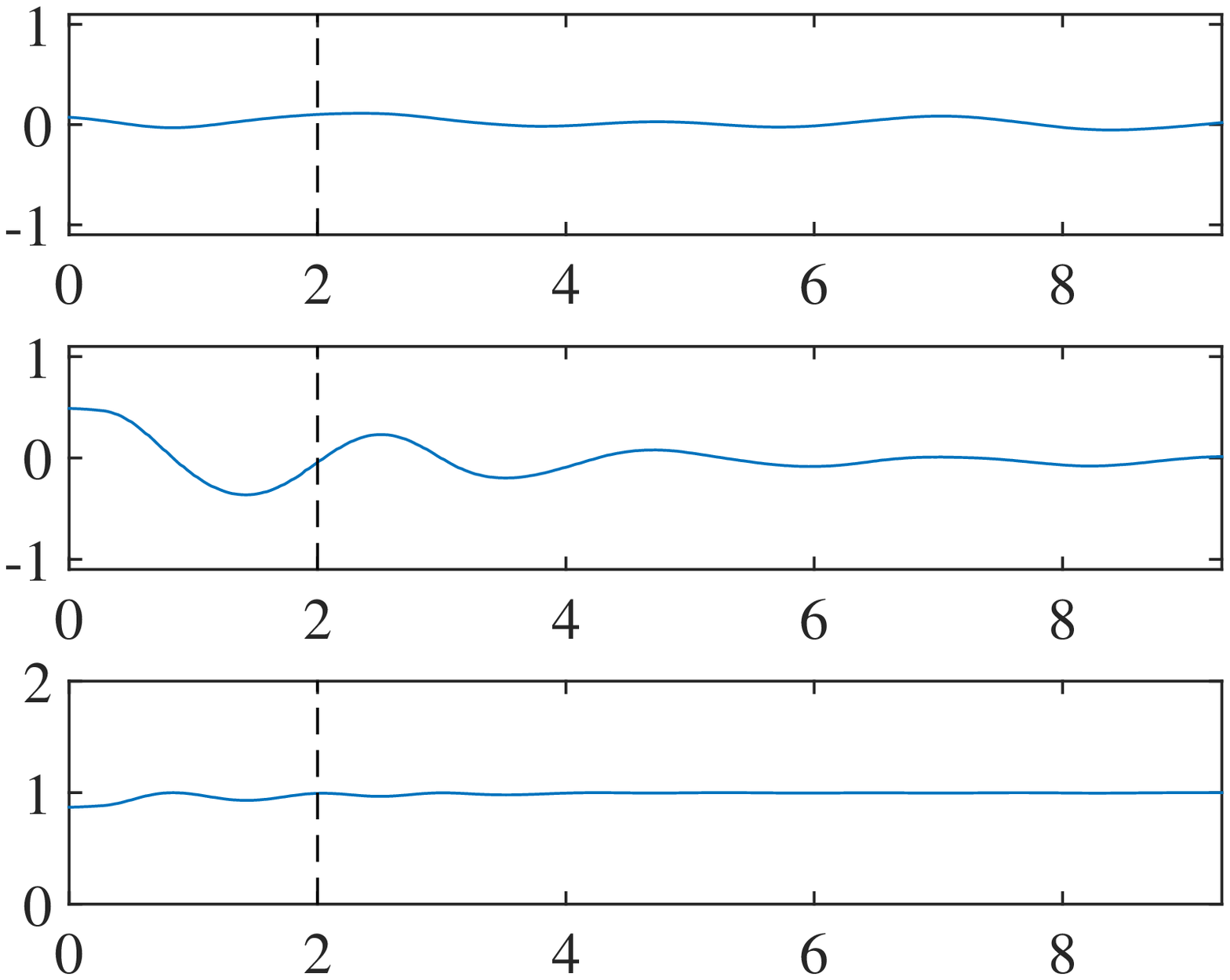}}
	\subfigure[Second cable's direction, $q_{2}$]{
		\includegraphics[width=0.5\columnwidth]{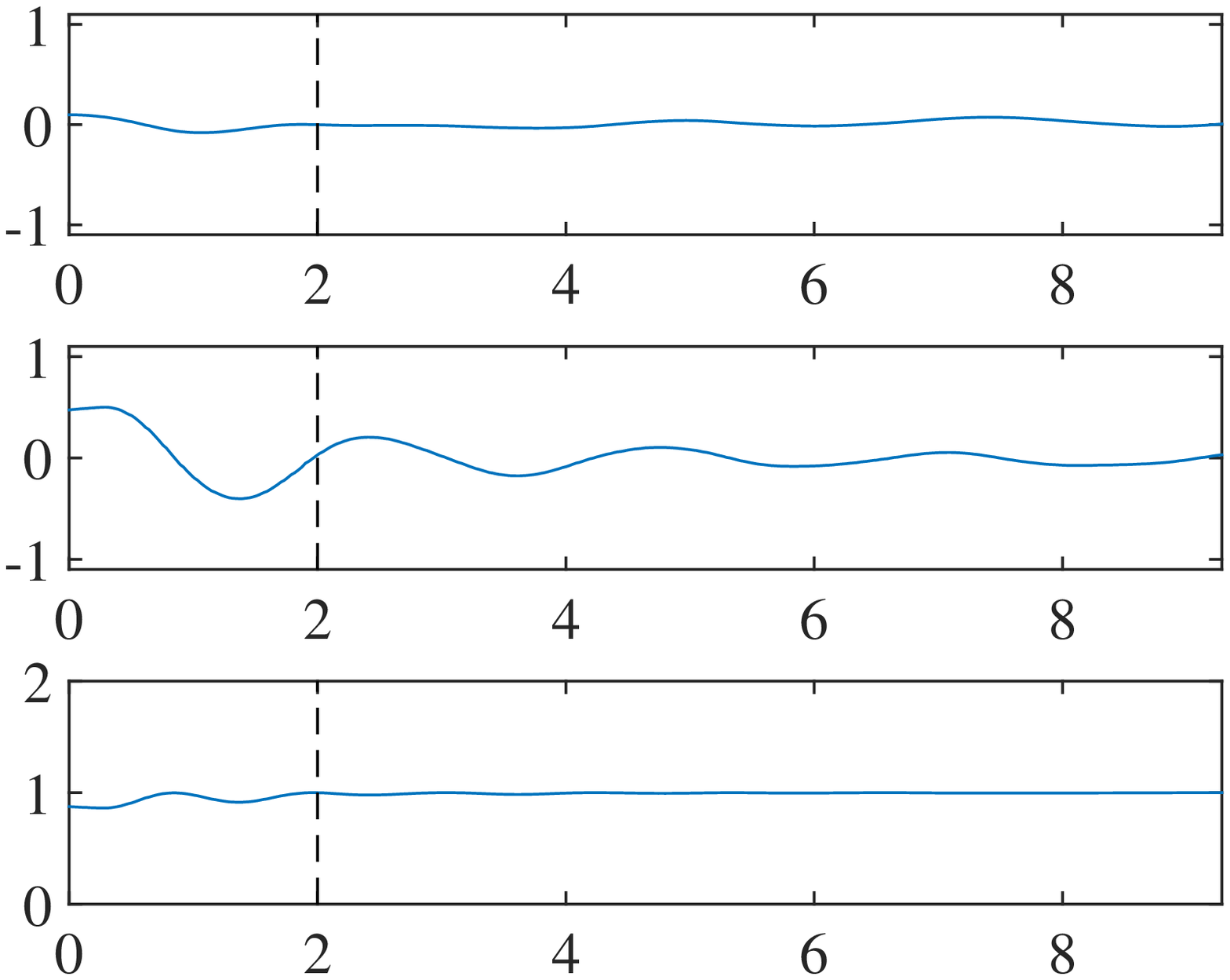}}
}
\caption{Proposed Controller: two quadrotors with rigid body payload. (The vertical dotted line indicates the time when controller switched in and stabilizes the system.)}\label{fig:expcase2}
\end{figure}

Figure~\ref{fig:expcase2} illustrates the position of the payload and quadrotors during this experiment where we applied the proposed controller. The vertical dotted line indicates the time when geometric nonlinear controller is switched with the proposed controller and stabilizes the system. 

As shown in this figure, the proposed controller reduces and eliminates the oscillations of the cables and payload much effectively while considering the payload and cables dynamics. The desired cables directions are along the vertical direction $e_{3}=[0,\; 0,\; 1]^{T}$ and the desired rod's directions is along the $e_{1}$ axis. 

\begin{figure}[h]
\centerline{
	\subfigure[]{
		\includegraphics[width=0.45\columnwidth]{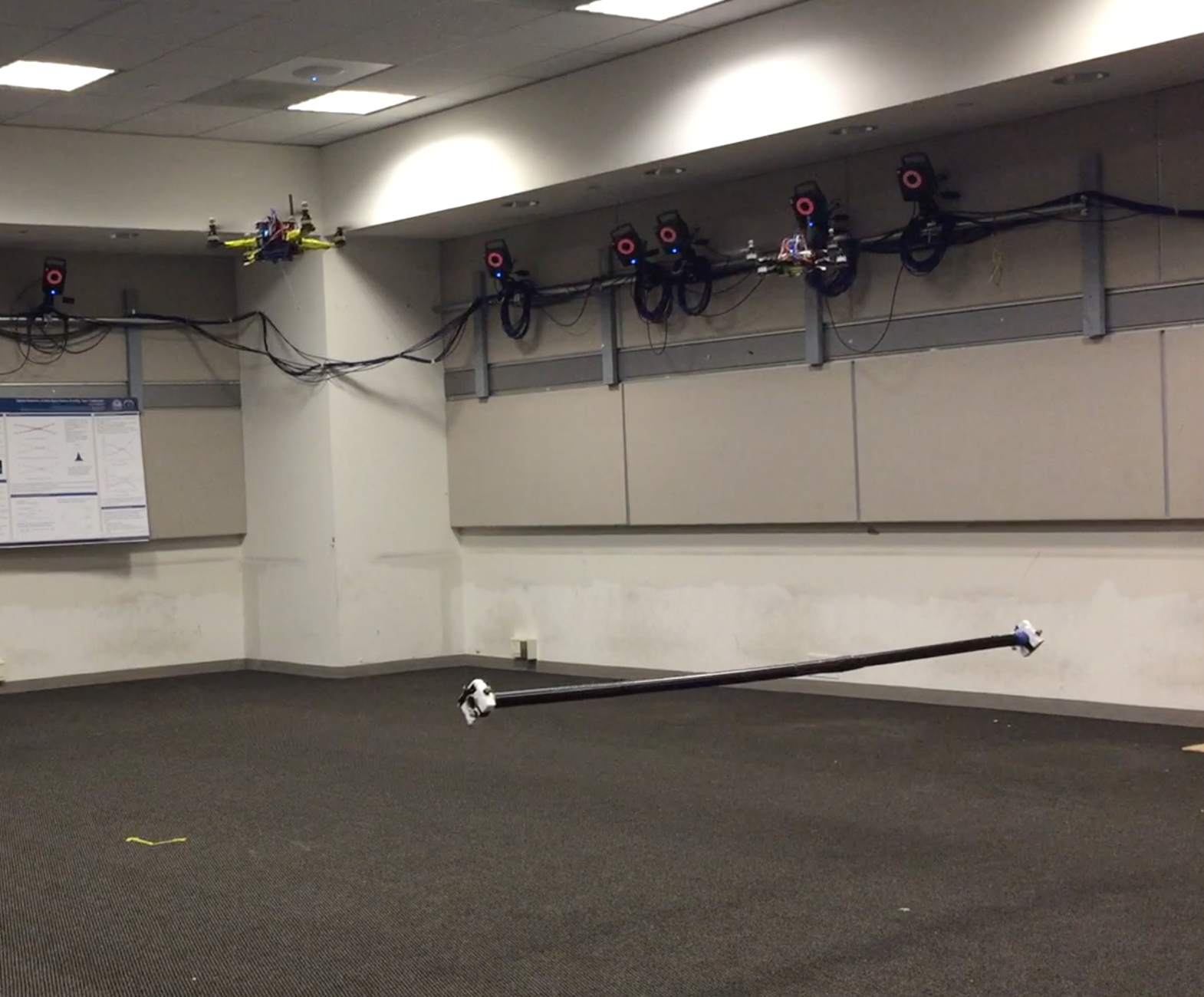}}
	\subfigure[]{
		\includegraphics[width=0.45\columnwidth]{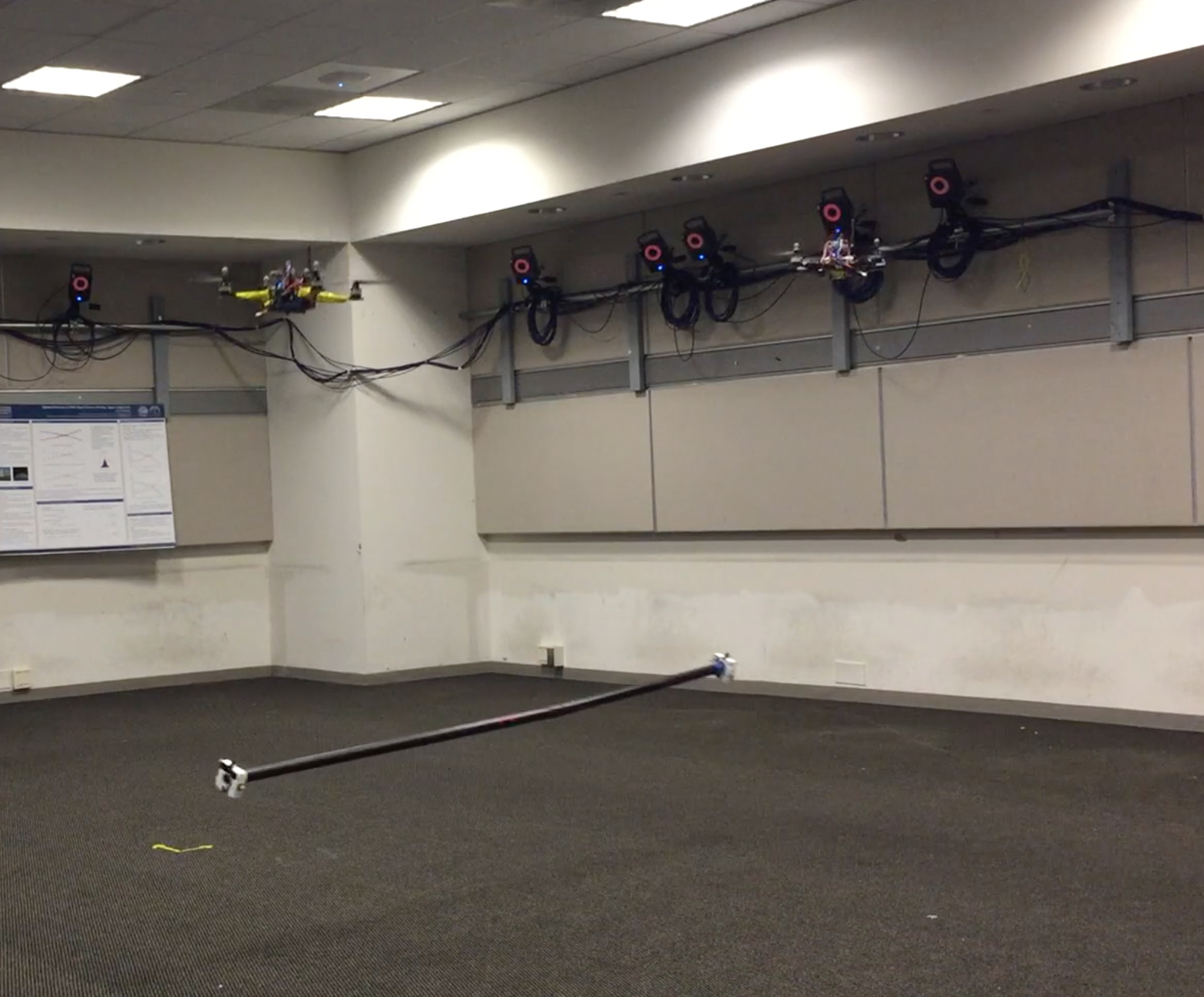}}
}
\centerline{
	\subfigure[]{
		\includegraphics[width=0.45\columnwidth]{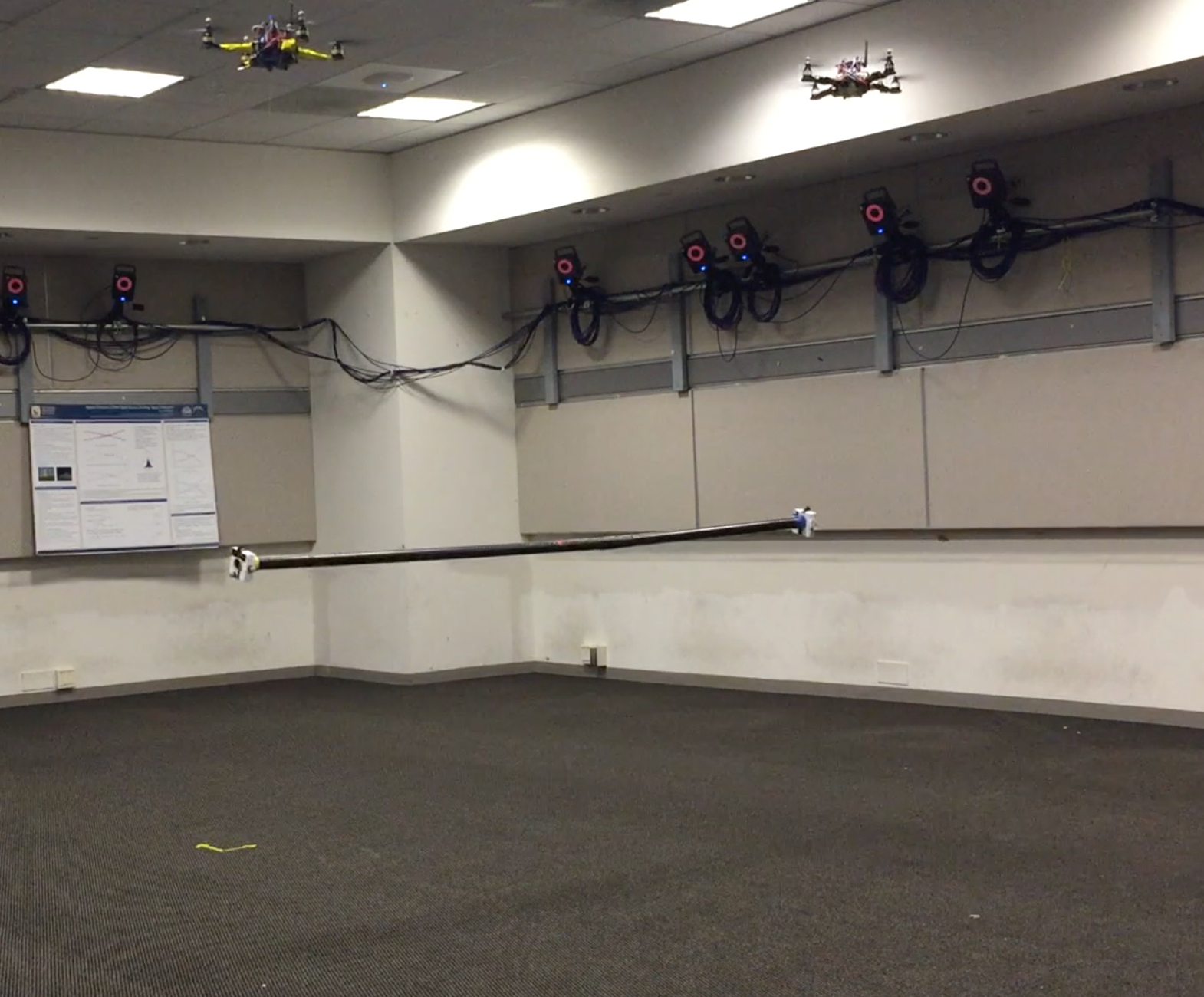}}
	\subfigure[]{
		\includegraphics[width=0.45\columnwidth]{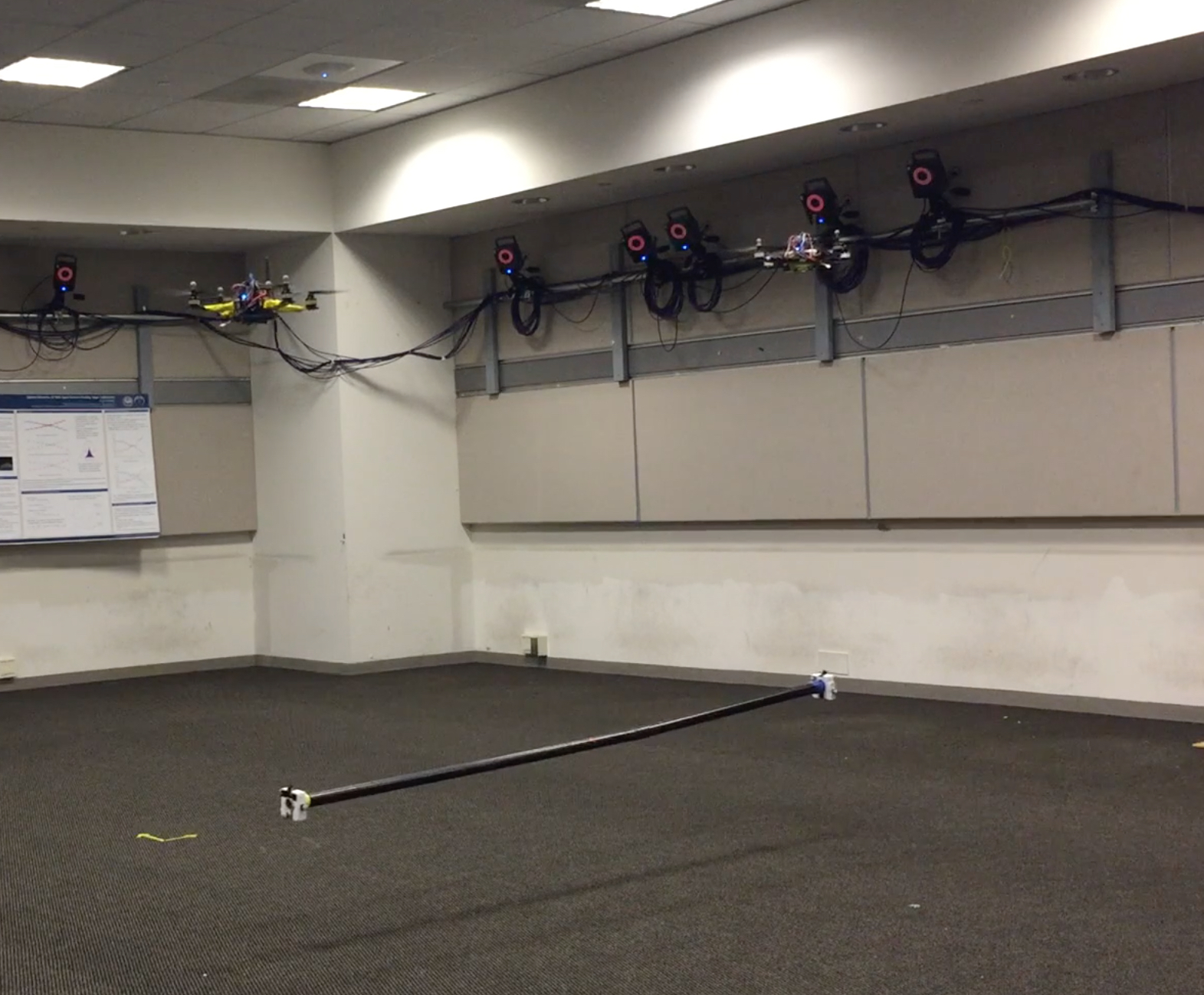}}
}
\caption{Snapshots of controlled stabilization of a rod with two quadrotors. A short video of this comparison is available at \href{https://www.youtube.com/watch?v=u65GqIl2skY}{https://www.youtube.com/watch?v=u65GqIl2skY}}\label{fig:expsnap}
\end{figure}
Snapshots of the controlled maneuvers is also illustrated at Figure \ref{fig:expsnap}. 
\section{Conclusions}
We utilized Euler-Lagrange equations to derive the complete model of multiple quadrotor UAVs transporting a rigid-body connected via flexible cables in 3D space. These derivations are developed in a remarkably compact form which allow us to choose an arbitrary number and any configuration of the links and an arbitrary number of quadrotors. We developed a geometric nonlinear controller to stabilize the links below the quadrotors and payload in the equilibrium position from an initial condition. We expanded these derivations in such a way that there is no need of using local angle coordinates which is an advantageous technique to signalize our derivations. A rigorous Lyapunov stability analysis is also presented to illustrate the stability properties without any time-scale separation. Numerical simulation and experimental results and for multiple cooperative quadrotors stabilizing a rigid-body performed and presented in this paper and shows the the accuracy and performance of the purposed model and control system.

\begin{acknowledgment}
This research has been supported in part by NSF under the grant CMMI-1243000 (transferred from 1029551), CMMI-1335008, and CNS-1337722.
\end{acknowledgment}
\bibliographystyle{asmems4}
\bibliography{asme2e}
\appendix
\section{Proof for Proposition \ref{prop:FDM}}\label{sec:PfFDM}
\subsection{Kinetic Energy}
The kinetic energy of the whole system is composed of the kinetic energy of quadrotors, cables and the rigid body, as
\begin{align}
T=&\frac{1}{2}m_{0}\|\dot{x}_{0}\|^{2}+\sum_{i=1}^{n}\sum_{j=1}^{n_{i}}{\frac{1}{2}m_{ij}\|\dot{x}_{ij}\|^{2}}+\frac{1}{2}\sum_{i=1}^{n}{m_{i}\|\dot{x}_{i}\|^{2}}\nonumber\\
&+\frac{1}{2}\sum_{i=1}^{n}{\Omega_{i}\cdot J_{i}\Omega_{i}}+\frac{1}{2}\Omega_{0}\cdot J_{0}\Omega_{0}.
\end{align}
Substituting the derivatives of \refeqn{xi} and \refeqn{xij} into the above expression we have
\begin{align}
T=&\frac{1}{2}m_{0}\|\dot{x}_{0}\|^{2}+\sum_{i=1}^{n}\sum_{j=1}^{n_{i}}{\frac{1}{2}m_{ij}\|\dot{x}_{0}+\dot{R}_{0}\rho_{i}-\sum_{a=j+1}^{n_{i}}{l_{ia}\dot{q}_{ia}}\|^{2}} \nonumber \\
&+\frac{1}{2}\sum_{i=1}^{n}{m_{i}\|\dot{x}_{0}+\dot{R}_{0}\rho_{i}-\sum_{a=1}^{n_{i}}{l_{ia}\dot{q}_{ia}}\|^{2}} \nonumber\\
&+\frac{1}{2}\sum_{i=1}^{n}{\Omega_{i}\cdot J_{i}\Omega_{i}}+\frac{1}{2}\Omega_{0}\cdot J_{0}\Omega_{0}.
\end{align}
We expand the above expression as follow
\begin{align}\label{eqn:kineticbs}
T=&\frac{1}{2}(m_{0}\|\dot{x}_{0}\|^{2}+\sum_{i=1}^{n}\sum_{j=1}^{n_{i}}{m_{ij}\|\dot{x}_{0}\|^{2}}+\sum_{i=1}^{n}{m_{i}\|\dot{x}_{0}\|^2}) \nonumber\\
&+\frac{1}{2}\sum_{i=1}^{n}(\sum_{j=1}^{n_{i}}{m_{ij}\|\dot{R}_{0}\rho_{i}\|^2}+m_{i}\|\dot{R}_{0}\rho_{i}\|^2) \nonumber\\
&+\sum_{i=1}^{n}(\sum_{j=1}^{n_{i}}{m_{ij}\dot{x}_{0}\cdot\dot{R}_{0}\rho_{i}}+m_{i}\dot{x}_{0}\cdot\dot{R}_{0}\rho_{i}) \nonumber\\
&+\frac{1}{2}\sum_{i=1}^{n}(\sum_{j=1}^{n_{i}}{m_{ij}\|\sum_{a=j+1}^{n_{i}}{l_{ia}\dot{q}_{ia}}\|^2}+{m_{i}\|\sum_{a=1}^{n_{i}}{l_{ia}\dot{q}_{ia}}\|^2}) \nonumber\\
&-\sum_{i=1}^{n}(\sum_{j=1}^{n_{i}}{m_{ij}\dot{x}_{0}}\cdot\sum_{a=j+1}^{n_{i}}{{l_{ia}\dot{q}_{ia}}}+\dot{x}_{0}\cdot\sum_{a=1}^{n_{i}}{l_{ia}\dot{q}_{ia}}) \nonumber\\
&-\sum_{i=1}^{n}(\sum_{j=1}^{n_{i}}{m_{ij}\dot{R}_{0}\rho_{i}}\cdot\sum_{a=j+1}^{n_{i}}{l_{ia}\dot{q}_{ia}}+m_{i}\dot{R}_{0}\rho_{i}\cdot\sum_{a=1}^{n_{i}}{l_{ia}\dot{q}_{ia}}) \nonumber\\
&+\frac{1}{2}\sum_{i=1}^{n}{\Omega_{i}\cdot J_{i}\Omega_{i}}+\frac{1}{2}\Omega_{0}\cdot J_{0}\Omega_{0},
\end{align}
and substituting \refeqn{def1}, \refeqn{def3}, it is rewritten as
\begin{align}\label{eqn:kinetic}
T=&\frac{1}{2}M_{T}\|\dot{x}_{0}\|^2+\frac{1}{2}\sum_{i=1}^{n}{M_{iT}\|\dot{R}_{0}\rho_{i}\|^{2}}+\sum_{i=1}^{n}({M_{iT}\dot{x}_{0}\cdot\dot{R}_{0}\rho_{i}}) \nonumber\\
&+\sum_{i=1}^{n}\sum_{j,k=1}^{n_{i}}{M_{0ij}l_{ik}\dot{q}_{ij}\cdot\dot{q}_{ik}}-\sum_{i=1}^{n}({\dot{x}_{0}\cdot\sum_{j=1}^{n_{i}}{M_{0ij}l_{ij}\dot{q}_{ij}}}) \nonumber\\
&-\sum_{i=1}^{n}({\dot{R}_{0}\rho_{i}}\cdot\sum_{j=1}^{n_{i}}{M_{0ij}l_{ij}\dot{q}_{ij}}) \nonumber\\
&+\frac{1}{2}\sum_{i=1}^{n}{\Omega_{i}\cdot J_{i}\Omega_{i}}+\frac{1}{2}\Omega_{0}\cdot J_{0}\Omega_{0}.
\end{align}
\subsection{Potential Energy}
We can derive the potential energy expression by considering the gravitational forces on each part of system as given
\begin{align}
V=-m_{0}ge_{3}\cdot x_{0}-\sum_{i=1}^{n}{m_{i}ge_{3}}\cdot x_{i}-\sum_{i=1}^{n}\sum_{j=1}^{n_{i}}{m_{ij}ge_{3}}\cdot x_{ij}.
\end{align}
Using \refeqn{xi} and \refeqn{xij}, we obtain
\begin{align}
V=&-m_{0}ge_{3}\cdot x_{0}-\sum_{i=1}^{n}{m_{i}ge_{3}}\cdot (x_{0}+R_{0}\rho_{i}-\sum_{a=1}^{n_{i}}{l_{ia}q_{ia}}) \nonumber\\
&-\sum_{i=1}^{n}\sum_{j=1}^{n_{i}}{m_{ij}ge_{3}}\cdot (x_{0}+R_{0}\rho_{i}-\sum_{a=j+1}^{n_{i}}{l_{ia}q_{ia}}),
\end{align}
and utilizing \refeqn{def3}, we can simplify the potential energy as
\begin{align}
V=-M_{T}ge_{3}\cdot x_{0}-\sum_{i=1}^{n}{M_{iT}ge_{3}\cdot R_{0}\rho_{i}}+\sum_{i=1}^{n}\sum_{j=1}^{n_{i}}{M_{0ij}l_{ij}q_{ij}\cdot e_{3}}.
\end{align}
\subsection{Derivatives of Lagrangian}
We develop the equation of motion for the Lagrangian $L=T-V$. The derivatives of the Lagrangian are given by
\begin{align}
&D_{\dot{x}_{0}}L=M_{T}\dot{x}_{0}+\sum_{i=1}^{n}{M_{iT}\dot{R}_{0}\rho_{i}}-\sum_{i=1}^{n}\sum_{j=1}^{n_{i}}{M_{0ij}l_{ij}\dot{q}_{ij}},\\
&D_{x_{0}}L=M_{T}ge_{3},\\
&D_{\dot{q}_{ij}}L=\sum_{i=1}^{n}\sum_{j=1}^{n_{i}}{M_{0ij}l_{ik}\dot{q}_{ik}}-\sum_{i=1}^{n}{M_{0ij}l_{ij}(\dot{x}_{0}}+{\dot{R}_{0}\rho_{i}}),\\
&D_{q_{ij}}L=-\sum_{i=1}^{n}{M_{0ij}l_{ij}e_{3}},
\end{align}
where $D_{\dot x_0}$ denote the derivative with respect to $\dot x_0$, and other derivatives are defined similarly. We also have
\begin{align}
D_{\Omega_{0}}L=&J_{0}\Omega_{0}+\sum_{i=1}^{n}{M_{iT}\hat{\rho}_{i}R_{0}^{T}\dot{x}_{0}},\nonumber\\
&-\sum_{i=1}^{n}\sum_{j=1}^{n_{i}}{M_{0ij}l_{ij}\hat{\rho}_{i}R_{0}^{T}\dot{q}_{ij}}-\sum_{i=1}^{n}{M_{iT}\hat{\rho}_{i}^{2}\Omega_{0}},
\end{align}
which can be rewritten as
\begin{align}
&D_{\Omega_{0}}L=\bar{J}_{0}\Omega_{0}+\sum_{i=1}^{n}{\hat{\rho}_{i}R_{0}^{T}(M_{iT}\dot{x}_{0}-\sum_{j=1}^{n_{i}}{M_{0ij}l_{ij}\dot{q}_{ij}})},
\end{align}
where $\bar{J}_{0}$ is defined as
\begin{align}
\bar{J}_{0}=J_{0}-\sum_{i=1}^{n}{M_{iT}\hat{\rho}_{i}^{2}}.
\end{align}
The derivative with respect ti $\Omega_{i}$ is simply given by
\begin{align}
D_{\Omega_{i}}L=\sum_{i=1}^{n}{J_{i}\Omega_{i}}.
\end{align}
The derivative of the Lagrangian with respect to $R_{0}$ along $\delta R_{0}=R_{0}\hat{\eta}_{0}$ is given by
\begin{align}
D_{R_{0}}L\cdot\delta R_{0}=&\sum_{i=1}^{n}{M_{iT}R_{0}\hat{\eta}_{0}\hat{\Omega}_{0}\rho_{i}\cdot\dot{x}_{0}}\nonumber\\
&-\sum_{i=1}^{n}{R_{0}\hat{\eta}_{0}\hat{\Omega}_{0}\rho_{i}\cdot\sum_{j=1}^{n_{i}}{M_{0ij}l_{ij}\dot{q}_{ij}}}\nonumber\\
&+\sum_{i=1}^{n}{M_{iT}ge_{3}\cdot R_{0}\hat{\eta}_{0}\rho_{i}},
\end{align}
which can be rewritten as
\begin{align}
&D_{R_{0}}L\cdot\delta R_{0}=d_{R_{0}}\cdot \eta_{0},
\end{align}
where
\begin{align}
d_{R_{0}}=\sum_{i=1}^{n}&(((\widehat{\hat{\Omega}_{0}\rho_{i}}R_{0}^{T}(M_{iT}\dot{x}_{0})-\sum_{j=1}^{n_{i}}{M_{0ij}l_{ij}\dot{q}_{ij}})\nonumber\\
&+M_{iT}g\hat{\rho}_{i}R_{0}^{T}e_{3})).
\end{align}
\subsection{Lagrange-d'Alembert Principle}
Consider $\mathfrak{G}=\int_{t_{0}}^{t_{f}}{L}$ be the action integral. Using the equations derived in previous section, the infinitesimal variation of the action integral can be~\cite{SchMurIICRA12} written as
\begin{align}
\delta \mathfrak{G}=&\int_{t_{0}}^{t_{f}}D_{\dot{x}_{0}}L\cdot\delta\dot{x}_{0}+D_{x_{0}}\cdot\delta x_{0} \nonumber\\
&+D_{\Omega_{0}}L(\dot{\eta}_{0}+\Omega_{0}\times \eta_{0})+d_{R_{0}}L\cdot\eta_{0} \nonumber\\
&+\sum_{i=1}^{n}\sum_{j=1}^{n_{i}}{D_{\dot{q}_{ij}}L(\dot{\xi}_{ij}\times q_{ij}+\xi_{ij}\times\dot{q}_{ij})}\nonumber\\
&+\sum_{i=1}^{n}\sum_{j=1}^{n_{i}}{D_{q_{ij}}L\cdot(\xi_{ij}\times q_{ij})} \nonumber\\
&+\sum_{i=1}^{n}{D_{\Omega_{i}}L\cdot(\dot{\eta}_{i}+\Omega_{i}\times\eta_{i})}.
\end{align}
The total thrust at the $i$-th quadrotor with respect to the inertial frame is denoted by $u_{i}=-f_{i}R_{i}e_{3}\in\Re^{3}$ and the total moment at the $i$-th quadrotor is defined as $M_{i}\in\Re^{3}$. The corresponding virtual work due to the controls and disturbances is given by
\begin{align}
\delta W=\int_{t_{0}}^{t_{f}}&{\sum_{i=1}^{n}{(u_{i}+\Delta_{x_i})\cdot\{\delta x_{0}+R_{0}\hat{\eta}_{0}\rho_{i}-\sum_{j=1}^{n_{i}}{l_{ij}\dot{\xi}_{ij}\times q_{ij}}}}\} \nonumber\\
&+(M_{i}+\Delta_{R_i})\cdot \eta_{i}\; dt.
\end{align}
According to Lagrange-d Alembert principle, we have $\delta \mathfrak{G}=-\delta W$ for any variation of trajectories with fixes end points. By using integration by parts and rearranging, we obtain the following Euler-Lagrange equations
\begin{gather}
\frac{d}{dt}D_{\dot{x}_{i}}L-D_{x_{0}}L=\sum_{i=1}^{n}{(u_{i}+\Delta_{x_i})},\\
\frac{d}{dt}D_{\Omega_{0}}+\Omega_{0}\times D_{\Omega_{0}}-d_{R_{0}}=\sum_{i=1}^{n}{\hat{\rho}_{i}R_{0}^{T}(u_{i}+\Delta_{x_i})},\\
\hat{q}_{ij}\frac{d}{dt}D_{\dot{q}_{ij}}L-\hat{q}_{ij}D_{q_{i}}L=-l_{ij}\hat{q}_{ij}(u_{i}+\Delta_{x_i}),\\
\frac{d}{dt}D_{\Omega_{i}}L+\Omega_{i}\times D_{\Omega_{i}}L=M_{i}+\Delta_{R_i}.
\end{gather}
Substituting the derivatives of Lagrangians into the above expression and rearranging, the equations of motion are given by \refeqn{EOMM1}, \refeqn{EOMM2}, \refeqn{EOMM3}, \refeqn{EOMM4}.

\section{Proof for Proposition \ref{prop:stability1}}\label{sec:P1stability}
The variations of $x$ and $q$ are given by \refeqn{xlin} and \refeqn{qlin}. From the kinematics equation $\dot q_{ij}=\omega_{ij}\times q_{ij}$ and
\begin{align*}
\delta \dot q_{ij} = \dot\xi_{ij} \times e_3 =\delta\omega_{ij} \times e_3 + 0\times (\xi_{ij}\times e_3)=\delta\omega_{ij} \times e_3.
\end{align*}
Since both sides of the above equation is perpendicular to $e_3$, this is equivalent to $e_3\times(\dot\xi_{ij}\times e_3) = e_3\times(\delta\omega_{ij}\times e_3)$, which yields
\begin{gather*}
\dot \xi_{ij} - (e_3\cdot\dot\xi_{ij}) e_3 = \delta\omega_{ij} -(e_3\cdot\delta\omega_{ij})e_3.
\end{gather*}
Since $\xi_{ij}\cdot e_3 =0$, we have $\dot\xi_{ij}\cdot e_3=0$. As $e_3\cdot\delta\omega_{ij}=0$ from the constraint, we obtain the linearized equation for the kinematics equation of the link as
\begin{align}
\dot\xi_{ij} = \delta\omega_{ij}.\label{eqn:dotxii}
\end{align}
The infinitesimal variation of $R_{0}\in\SO$ in terms of the exponential map
\begin{align}
\delta R_{0} = \frac{d}{d\epsilon}\bigg|_{\epsilon = 0} R_{0}\exp (\epsilon \hat\eta_{0}) = R_{0}\hat\eta_{0},\label{eqn:delR0}
\end{align}
for $\eta_{0}\in\Re^3$. Substituting these into \refeqn{EOMM1}, \refeqn{EOMM2}, and \refeqn{EOMM3}, and ignoring the higher order terms, we obtain the following sets of linearized equations of motion 
\begin{gather}
M_{T}\delta \ddot{x}_{0}-\sum_{i=1}^{n}{M_{iT}\hat{\rho}_{i}}\delta\dot{\Omega}_{0}\nonumber\\
+\sum_{i=1}^{n}\sum_{j=1}^{n_{i}}M_{0ij}l_{ij}\hat{e}_{3}C(C^{T}\ddot{\xi}_{ij})=\sum_{i=1}^{n}{\delta u_{i}}\\
\sum_{i=1}^{n}{M_{iT}\hat{\rho}_{i}\delta\ddot{x}_{0}}+\bar{J}_{0}\delta\dot{\Omega}_{0}+\sum_{i=1}^{n}\sum_{j=1}^{n_{i}}M_{0ij}l_{ij}\hat{\rho}_{i}\hat{e}_{3}C(C^{T}\ddot{\xi}_{ij})\nonumber\\
+\sum_{i=1}^{n}\frac{m_{0}}{n}g\hat{\rho}_{i}\hat{e}_{3}\eta_{0}=\sum_{i=1}^{n}{\hat{\rho}_{i}\delta u_{i}},\\
-M_{0ij}C^{T}\hat{e}_{3}\delta\ddot{x}_{0}+M_{0ij}C^{T}\hat{e}_{3}\hat{\rho}_{i}\delta\dot{\Omega}_{0}+\sum_{k=1}^{n_{i}}{M_{0ij}l_{ik}I_{2}(C^{T}\ddot{\xi}_{ij})}\nonumber\\
=-C^{T}\hat{e}_{3}\delta u_{i}+(-M_{iT}-\frac{m_{0}}{n}+M_{0ij})ge_{3} I_{2}(C^{T}\xi_{ij}),\\
\dot{\eta}_{i}=\delta\Omega_{i},\quad \dot{\eta}_{0}=\delta\Omega_{0},\quad J_{i}\delta\Omega_{i}=\delta M_{i},
\end{gather}
which can be written in a matrix form as presented in \refeqn{EOMLin}. We used $C^{T}\hat{e}_{3}^{2}C=-I_{2}$ to simplify these derivations. 

\section{Proof for Proposition \ref{prop:stability}}\label{sec:Pstabilityddd}
We first show stability of the rotational dynamics of each quadrotor, and later it is combined with the stability analysis for the remaining parts.
\subsection{Attitude Error Dynamics}
Here, the attitude error dynamics for $e_{R_{i}}$, $e_{\Omega_{i}}$ are derived and we find conditions on control parameters to guarantee the stability. The time-derivative of $J_{i}e_{\Omega_{i}}$ can be written as
\begin{align}
J_{i}\dot e_{\Omega_{i}} & = \{J_{i}e_{\Omega_{i}} + d_{i}\}^\wedge e_{\Omega_{i}} - k_R e_{R_{i}}-k_\Omega e_{\Omega_{i}}-k_{I}e_{I_{i}}+\Delta_{R_{i}},\label{eqn:JeWdot}
\end{align}
where $d_{i}=(2J_{i}-\trs{[J_{i}]I})R_{i}^TR_{i_{d}}\Omega_{i_d}\in\Re^3$~\cite{Farhad2013}. The important property is that the first term of the right hand side is normal to $e_{\Omega_{i}}$, and it simplifies the subsequent Lyapunov analysis.

\subsection{Stability for Attitude Dynamics}
Define a configuration error function on $\SO$ as follows
\begin{align}
\Psi_{i}= \frac{1}{2}\trs{I- R_{{i}_c}^T R_{i}}.
\end{align}
We introduce the following Lyapunov function
\begin{align}
\mathcal{V}_{2}=\sum_{i=1}^{n}{\mathcal{V}_{2_{i}}},
\end{align}
where 
\begin{align}
\mathcal{V}_{2_i}=&\frac{1}{2}e_{\Omega_{i}}\cdot J_{i}\dot{e}_{\Omega_{i}}+k_{R}\Psi_{i}(R_{i},R_{d_{i}})+c_{2_i}e_{R_{i}}\cdot e_{\Omega_{i}}\\
&+\frac{1}{2}k_{I}\|e_{I_{i}}-\frac{\Delta_{R_{i}}}{k_{I}}\|^{2}.
\end{align}
Consider a domain $D_{2}$ given by
\begin{align}
D_2 = \{ (R_{i},\Omega_{i})\in \SO\times\Re^3\,|\, \Psi_{i}(R_{i},R_{d_{i}})<\psi_{2_i}<2\}.\label{eqn:D2}
\end{align}
In this domain we can show that $\mathcal{V}_{2}$ is bounded as follows~\cite{Farhad2013}
\begin{align}\label{eqn:ffff1}
\begin{split}
z_{2_i}^{T}M_{i_{21}}z_{2_i}+&\frac{k_{I}}{2}\|e_{I_{i}}-\frac{\Delta_{R_i}}{k_{I}}\|^{2}\leq\mathcal{V}_{2_i}\\
&\leq z_{2_i}^{T}M_{i_{22}}z_{2_i}+\frac{k_{I}}{2}\|e_{I_{i}}-\frac{\Delta_{R_i}}{k_{I}}\|^{2},
\end{split}
\end{align}
where $z_{2_i}=[\|e_{R_{i}}\|,\|e_{\Omega_{i}}\|]^{T}\in \Re^{2}$ and matrices $M_{i_{21}}$, $M_{i_{22}}\in\Re^{2\times 2}$ are given by
\begin{align}
M_{i_{21}}=&\frac{1}{2}\begin{bmatrix}
k_{R}&-c_{2_i}\lambda_{M_{i}}\\
-c_{2_i}\lambda_{M_{i}}&\lambda_{m_{i}}
\end{bmatrix},\nonumber\\
M_{i_{22}}=&\frac{1}{2}\begin{bmatrix}
\frac{2k_{R}}{2-\psi_{2_i}}&c_{2_i}\lambda_{M_{i}}\\
c_{2_i}\lambda_{M_{i}}&\lambda_{M_{i}}\nonumber
\end{bmatrix}.
\end{align}
The time derivative of $\mathcal{V}_2$ along the solution of the controlled system is given by
\begin{align*}
\dot{\mathcal{V}}_2 =&\sum_{i=1}^{n}-k_\Omega\|e_{\Omega_{i}}\|^2 -e_{\Omega_{i}}\cdot(k_{I}e_{I_i}-\Delta_{R_{i}})\\
&+ c_{2_i} \dot e_{R_{i}} \cdot J_{i}e_{\Omega_{i}}+ c_{2_i} e_{R_{i}} \cdot J_{i}\dot e_{\Omega_{i}}+(k_{I}e_{I_i}-\Delta_{R_i})\dot{e}_{I_i}.
\end{align*}
We have $\dot{e}_{I_i}=c_{2}e_{R_i}+e_{\Omega_{i}}$. Substituting \refeqn{JeWdot}, the above equation becomes
\begin{align*}
\dot{\mathcal{V}}_2 =&
\sum_{i=1}^{n}-k_\Omega\|e_{\Omega_{i}}\|^2  + c_{2_i} \dot e_{R_{i}} \cdot J_{i}e_{\Omega_{i}}-c_{2_i} k_R \|e_{R_{i}}\|^2 \\
&+ c_{2_i} e_{R_{i}} \cdot ((J_{i}e_{\Omega_{i}}+d_{i})^\wedge e_{\Omega_{i}} -k_\Omega e_{\Omega_{i}}).
\end{align*}
We have $\|e_{R_{i}}\|\leq 1$, $\|\dot e_{R_{i}}\|\leq \|e_{\Omega_{i}}\|$~\cite{TFJCHTLeeHG}, and choose a constant $B_{2_i}$ such that $\|d_{i}\|\leq B_{i_2}$. Then we obtain
\begin{align}
\dot{\mathcal{V}}_2 \leq -\sum_{i=1}^{n} z_{2_i}^T W_{2_i} z_{2_i},\label{eqn:dotV2}
\end{align}
where the matrix $W_{2_i}\in\Re^{2\times 2}$ is given by
\begin{align*}
W_{2_i} = \begin{bmatrix} c_{2_i}k_R & -\frac{c_{2_i}}{2}(k_\Omega+B_{2_i}) \\ 
-\frac{c_{2_i}}{2}(k_\Omega+B_{2_i}) & k_\Omega-2c_{2_i}\lambda_{M_{i}} \end{bmatrix}.
\end{align*}
The matrix $W_{2_i}$ is a positive definite matrix if 
\begin{align}\label{eqn:c2}
c_{2_i}<\min\{\frac{\sqrt{k_{R}\lambda_{m_i}}}{\lambda_{M_i}},\frac{4k_{\Omega}}{8k_{R}\lambda_{M_i}+(k_{\Omega}+B_{i_2})^{2}} \}.
\end{align}
This implies that
\begin{align}\label{eqn:eq2}
\dot{\mathcal{V}}_{2}\leq -\sum_{i=1}^{n} {\lambda_{m}(W_{2_i})\|z_{2_i}\|^{2}},
\end{align} 
which shows stability of the attitude dynamics of quadrotors.
\subsection{Error Dynamics of the Payload and Links}
We derive the tracking error dynamics and a Lyapunov function for the translational dynamics of a payload and the dynamics of links. Later it is combined with the stability analyses of the rotational dynamics. From \refeqn{EOMM1}, \refeqn{EOMLin}, \refeqn{Ai}, and \refeqn{fi}, the equation of motion for the controlled dynamic model is given by
\begin{align}\label{eqn:salam}
\Mb\ddot \xb  + \Gb\xb=\Bb(u-u_{d})+\g(\xb,\dot{\xb})+\Bb\Delta_{x},
\end{align}
where $\Delta_x\in\Re^{3n\times 1}$ is
\begin{align}
\Delta_x=\begin{bmatrix}
\Delta_{x_1} & \Delta_{x_2} &\cdots & \Delta_{x_n}
\end{bmatrix}^{T},
\end{align}
and
\begin{align}
u=\begin{bmatrix}
u_{1}\\
u_{2}\\
\vdots\\
u_{n}
\end{bmatrix},\; u_{d}=\begin{bmatrix}
-(M_{1T}+\frac{m_{0}}{n})ge_{3}\\
-(M_{2T}+\frac{m_{0}}{n})ge_{3}\\
\vdots\\
-(M_{nT}+\frac{m_{0}}{n})ge_{3}
\end{bmatrix},
\end{align}
and $\g(\xb,\dot{\xb})$ corresponds to the higher order terms. As $u_i=-f_i R_i e_3$ for the full dynamic model, $\delta u=u-u_{d}$ is given by
\begin{align}\label{eqn:deltauu}
\delta u=
\begin{bmatrix}
-f_{1}R_{1}e_{3}+(M_{1T}+\frac{m_{0}}{n})ge_{3}\\
-f_{2}R_{2}e_{3}+(M_{2T}+\frac{m_{0}}{n})ge_{3}\\
\vdots\\
-f_{n}R_{n}e_{3}+(M_{nT}+\frac{m_{0}}{n})ge_{3}
\end{bmatrix}.
\end{align}
The subsequent analyses are developed in the domain $D_{1}$
\begin{align}
D_1=\{&(\xb,\dot{\xb},R_i,e_{\Omega_i})\in\Re^{D_{\xb}}\times\Re^{D_{\xb}}\times \SO\times\Re^3\,|\,\nonumber\\
& \Psi_{i}< \psi_{1_i} < 1\}.\label{eqn:D}
\end{align}
In the domain $D_{1}$, we can show that 
\begin{align}
\frac{1}{2} \norm{e_{R_{i}}}^2 \leq  \Psi_i(R_i,R_{c_i}) \leq \frac{1}{2-\psi_{1_i}} \norm{e_{R_i}}^2\label{eqn:eRPsi1}.
\end{align}
Consider the quantity $e_{3}^{T}R_{c_i}^{T}R_{i}e_{3}$, which represents the cosine of the angle between $b_{3_i}=R_{i}e_{3}$ and $b_{3_{c_i}}=R_{c_i}e_{3}$. Since $1-\Psi_i(R_i,R_{c_i})$ represents the cosine of the eigen-axis rotation angle between $R_{c_i}$ and $R_i$, we have $e_{3}^{T}R_{c_i}^{T}Re_{3}\geq 1-\Psi_i(R_i,R_{c_i})>0$ in $D_{1}$. Therefore, the quantity $\frac{1}{e_{3}^{T}R_{c_i}^{T}R_i e_{3}}$ is well-defined. We add and subtract $\frac{f_i}{e_{3}^{T}R_{c_i}^{T}R_i e_{3}}R_{c_i}e_{3}$ to the right hand side of \refeqn{deltauu} to obtain
\begin{align}\label{eqn:hallaa}
\delta u=
\begin{bmatrix}
\frac{-f_1}{e_{3}^{T}R_{c_1}^{T}R_1 e_{3}}R_{c_1}e_{3}-X_1+(M_{1T}+\frac{m_{0}}{n})ge_{3}\\
\frac{-f_2}{e_{3}^{T}R_{c_2}^{T}R_2 e_{3}}R_{c_2}e_{3}-X_2+(M_{2T}+\frac{m_{0}}{n})ge_{3}\\
\vdots\\
\frac{-f_n}{e_{3}^{T}R_{c_n}^{T}R_n e_{3}}R_{c_n}e_{3}-X_n+(M_{nT}+\frac{m_{0}}{n})ge_{3}
\end{bmatrix}.
\end{align}
where $X_i \in \Re^{3}$ is defined by
\begin{align}\label{eqn:Xdef}
X_i=\frac{f_i}{e_{3}^{T}R_{c_i}^{T}R_i e_{3}}((e_{3}^{T}R_{c_i}^{T}R_i e_{3})R_i e_{3}-R_{c_i}e_{3}).
\end{align}
Using 
\begin{align}
-\frac{f_i}{e_{3}^{T}R_{c_i}^{T}R_i e_{3}}R_{c_i}e_{3}=-\frac{(\|A_i\|R_{c_i}e_{3})\cdot R_i e_{3}}{e_{3}^{T}R_{c_i}^{T}R_i e_{3}}\cdot -\frac{A_i}{\|A_i\|}=A_i,
\end{align}
the equation \refeqn{hallaa} becomes
\begin{align}
\delta u=
\begin{bmatrix}
A_{1}-X_1+(M_{1T}+\frac{m_{0}}{n})ge_{3}\\
A_{2}-X_2+(M_{2T}+\frac{m_{0}}{n})ge_{3}\\
\vdots\\
A_{n}-X_n+(M_{nT}+\frac{m_{0}}{n})ge_{3}
\end{bmatrix}.
\end{align}
Substituting \refeqn{Ai} into the above equation, \refeqn{salam} becomes
\begin{align}
\Mb\ddot \xb  + \Gb\xb=\Bb(-K_{\xb}\xb-K_{\dot{\xb}}\dot{\xb}-X-K_{z}\sat_{\sigma}(e_{\xb})+\Delta_x)+\g(\xb,\dot{\xb}),
\end{align}
where $X=[X_{1}^T,\; X_{2}^T,\; \cdots,\; X_{n}^T]^{T}\in\Re^{3n}$. This can be rearranged as
\begin{align}
\begin{split}
\ddot{\xb}=&-(\Mb^{-1}\Gb+\Mb^{-1}\Bb K_{x})\xb-(\Mb^{-1}\Bb K_{\dot{x}})\dot{\xb}\\
&-\Mb^{-1}\Bb X-\Mb^{-1}\Bb K_{z}\sat_{\sigma}(e_{\xb})+\Mb^{-1}\g(\xb,\dot{\xb})+\Mb^{-1}\Bb\Delta_{x}.
\end{split}
\end{align}
Using the definitions for $\mathds{A}$, $\mathds{B}$, and $z_{1}$ presented before, the above expression can be rearranged as
\begin{align}\label{eqn:zdot1}
\dot{z}_{1}=&\mathds{A} z_{1}+\mathds{B}(-\Bb X+\g(\xb,\dot{\xb})-\Bb K_{z}\sat_{\sigma}(e_{\xb})+\Bb\Delta_{x}).
\end{align}

\subsection{Lyapunov Candidate for Simplified Dynamics}
From the linearized control system developed at section 3, we use matrix $P$ to introduce the following Lyapunov candidate for translational dynamics
\begin{align}
\mathcal{V}_{1}=z_{1}^{T}Pz_{1}+2\int_{p_{eq}}^{e_{\xb}}{(\Bb K_{z}\satr_{\sigma}(\mu)-\Bb\Delta_{x})}\cdot d \mu.
\end{align}
The last integral term of the above equation is positive definite about the equilibrium point $e_{\xb}=p_{eq}$ where
\begin{align}
p_{eq}=[\frac{\Delta_{x}}{k_{z}},0,0,\cdots],
\end{align}
if $\delta< k_z\sigma$, considering the fact that $\sat_{\sigma}{y}=y$ if $y<\sigma$. The time derivative of the Lyapunov function using the Leibniz integral rule is given by
\begin{align}\label{eqn:devrr}
\dot{\mathcal{V}_{1}}=\dot{z}_{1}^{T}Pz_{1}+z_{1}^{T}P\dot{z}_{1}+2\dot{e}_{\xb}\cdot(\Bb K_{z}\satr_{\sigma}(e_{\xb})-\Bb\Delta_{x}).
\end{align}
Since $\dot{e}_{\xb}^{T}=((P\mathds{B})^{T}z_{1})^{T}=z_{1}^{T}P\mathds{B}$ from \refeqn{exterm}, the above expression can be written as
\begin{align}\label{eqn:devvcv}
\dot{\mathcal{V}_{1}}=\dot{z}_{1}^{T}Pz_{1}+z_{1}^{T}P\dot{z}_{1}+2z_{1}^{T}P\mathds{B}(\Bb K_{z}\satr_{\sigma}(e_{\xb})-\Bb\Delta_{x}).
\end{align}
Substituting \refeqn{zdot1} into \refeqn{devvcv}, it reduces to
\begin{align}\label{eqn:beforsimp}
\dot{\mathcal{V}}_{1}=z_{1}^{T}(\mathds{A}^{T}P+P\mathds{A})z_{1}+2z_{1}^{T}P\mathds{B}(-\Bb X+\g(\xb,\dot{\xb})).
\end{align}
Let $c_{3}=2\|P\mathds{B}\Bb\|_{2}\in\Re$ and using $\mathds{A}^{T}P+P\mathds{A}=-Q$, we have
\begin{align}\label{eqn:test}
\dot{\mathcal{V}}_{1}\leq-z_{1}^{T}Qz_{1}+c_{3}\|z_{1}\|\|X\|+2z_{1}^{T}P\mathds{B}\g(\xb,\dot{\xb}).
\end{align}
The second term on the right hand side of the above equation corresponds to the effects of the attitude tracking error on the translational dynamics. We find a bound of $X_{i}$, defined at \refeqn{Xdef}, to show stability of the coupled translational dynamics and rotational dynamics in the subsequent Lyapunov analysis. Since 
\begin{align}
f_{i}=\|A_{i}\|(e_{3}^{T}R_{c_i}^{T}R_i e_{3}), 
\end{align}
we have
\begin{align}\label{eqn:ssstr}
\|X_i\|\leq\|A_i\|\|(e_{3}^{T}R_{c_i}^{T}R_i e_{3})R_i e_{3}-R_{c_i}e_{3}\|.
\end{align}
The last term $\|(e_{3}^{T}R_{c_i}^{T}R_i e_{3})R_i e_{3}-R_{c_i}e_{3}\|$ represents the sine of the angle between $b_{3_i}=R_i e_{3}$ and $b_{3_{c_i}}=R_{c_i}e_{3}$, since $(b_{3_{c_i}}\cdot b_{3_i})b_{3_i}-b_{3_{c_i}}=b_{3_i}\times(b_{3_i}\times b_{3_{c_i}})$. The magnitude of the attitude error vector, $\|e_{R_i}\|$ represents the sine of the eigen-axis rotation angle between $R_{c_i}$ and $R_i$. Therefore, $\|(e_{3}^{T}R_{c_i}^{T}R_i e_{3})R_i e_{3}-R_{c_i}e_{3}\|\leq\|e_{R_i}\|$ in $D_{1}$. It follows that
\begin{align}
\|(e_{3}^{T}R_{c_i}^{T}R_i e_{3})R_i e_{3}-R_{c_i}e_{3}\|&\leq\|e_{R_i}\|=\sqrt{\Psi_i(2-\Psi_i)}\nonumber\\
&\leq\{\sqrt{\psi_{1_i}(2-\psi_{1_i})}\triangleq\alpha_i\}<1,
\end{align}
therefore
\begin{align}
\|X_i\|&\leq \|A_i\|\|e_{R_i}\|\nonumber\\
&\leq\|A_i\|\alpha_i.
\end{align}
We find an upper boundary for
\begin{align}\label{eqn:AA}
A_i=-K_{\xb} \xb-K_{\dot{\xb}}\dot\xb -K_{z}\satr_{\sigma}(e_{\xb})+u_{i_{d}}.
\end{align}
We define $K_{max}, K_{z_{m}}\in\Re$
\begin{align*}
&K_{\max}=\max\{\|K_{\xb}\|,\|K_{\dot{\xb}}\|\}, \\
&K_{z_{m}}=\|K_{z}\|,
\end{align*}
by defining $\|u_{i_{d}}\|\leq B_{1_i}$, the upper bound of $A_i$ is given by
\begin{align}
\|A_i\| & \leq K_{\max}(\|\xb\|+\|\dot{\xb}\|)+\sigma K_{z_{m}}+B_{1_i}\\
&\leq 2K_{\max}\|z_{1}\|+(B_{1_i}+\sigma K_{z_{m}}),\label{eqn:normA}
\end{align}
Using the above steps we can show that
\begin{align}
\|X\|&\leq \sum_{i=1}^{n}((2K_{\max}\|z_{1}\|+(B_{1_i}+\sigma K_{z_{m}}))\|e_{R_{i}}\|)\nonumber\\
&\leq (2K_{\max}\|z_{1}\|+(B_{1_i}+\sigma K_{z_{m}}))\alpha,
\end{align}
where $\alpha=\sum_{i=1}^{n}\alpha_{i}$. Then, we can simplify \refeqn{test} as 
\begin{align}\label{eqn:eq1}
\dot{\mathcal{V}}_{1} \leq& -(\lambda_{\min}(Q)-2c_{3}K_{\max}\alpha) \|z_{1}\| ^{2}\nonumber\\
&+\sum_{i=1}^{n}+c_{3}(B_{1_i}+\sigma K_{z_{m}})\|z_{1}\|\|e_{R_i}\|+2z_{1}^{T}P\mathds{B}\g(\xb,\dot{\xb}).
\end{align}
\subsection{Lyapunov Candidate for the Complete System}
Let $\mathcal{V}=\mathcal{V}_{1}+\mathcal{V}_{2}$ be the Lyapunov function for the complete system. The time derivative of $\mathcal{V}$ is given by
\begin{align}
\dot{\mathcal{V}}=\dot{\mathcal{V}}_{1}+\dot{\mathcal{V}}_{2}.
\end{align}
Substituting \refeqn{eq1} and \refeqn{eq2} into the above equation
\begin{align}
\dot{\mathcal{V}}\leq& -(\lambda_{\min}(Q)-2c_{3}K_{\max}\alpha) \|z_{1}\| ^{2}+2z_{1}^{T}P\mathds{B}\g(\xb,\dot{\xb})\nonumber\\
&+\sum_{i=1}^{n}+c_{3}(B_{1_i}+\sigma K_{z_{m}})\|z_{1}\|\|e_{R_i}\|-\sum_{i=1}^{n}\lambda_{m}(W_{2_i})\|z_{2_i}\|^{2},
\end{align}
and using $\|e_{R_{i}}\|\leq \|z_{2_i}\|$, it can be written as
\begin{align}\label{eqn:finalsimp}
\dot{\mathcal{V}}\leq& -(\lambda_{\min}(Q)-2c_{3}K_{\max}\alpha) \|z_{1}\| ^{2}+2z_{1}^{T}P\mathds{B}\g(\xb,\dot{\xb})\nonumber\\
&+\sum_{i=1}^{n}c_{3}(B_{1_i}+\sigma K_{z_{m}}) \|z_{1}\|\|z_{2_i}\|-\sum_{i=1}^{n}\lambda_{m}(W_{2_i})\|z_{2_i}\|^{2}.
\end{align}
The $2z_{1}^{T}P\mathds{B}\g(\xb,\dot{\xb})$ term in the above equation is indefinite. But, the function $\g(\xb,\dot{\xb})$ satisfies
\begin{align*}
\frac{\|\g(\xb,\dot{\xb})\|}{\|z_{1}\|}\rightarrow 0\quad \mbox{as}\quad \|z_{1}\|\rightarrow 0.
\end{align*}
Then, for any $\gamma>0$ there exists $r>0$ such that
\begin{align*}
\|\g(\xb,\dot{\xb})\|<\gamma\|z_{1}\|\quad \forall\|z_{1}\|<r.
\end{align*}
Therefore
\begin{align}
2z_{1}^{T}P\mathds{B}\g(\xb,\dot{\xb})\leq 2\gamma\|P\|_{2}\|z_{1}\|^{2}.
\end{align}
Substituting the above inequality into \refeqn{finalsimp}
\begin{align}
\dot{\mathcal{V}}\leq& -(\lambda_{\min}(Q)-2c_{3}K_{\max}\alpha) \|z_{1}\| ^{2}+2\gamma\|P\|_{2}\|z_{1}\|^{2}\nonumber\\
&+\sum_{i=1}^{n}c_{3}(B_{1_i}+\sigma K_{z_{m}}) \|z_{1}\|\|z_{2_i}\|-\sum_{i=1}^{n}\lambda_{m}(W_{2_i})\|z_{2_i}\|^{2},
\end{align}
and rearranging
\begin{align}
\dot{\mathcal{V}}\leq -\sum_{i=1}^{n}(&\frac{\lambda_{\min}(Q)-2c_{3}K_{\max}\alpha}{n} \|z_{1}\| ^{2}\nonumber\\
&-c_{3}(B_{1_i}+\sigma K_{z_{m}}) \|z_{1}\|\|z_{2_i}\|+\lambda_{m}(W_{2_i})\|z_{2_i}\|^{2})\nonumber\\
&+2\gamma\|P\|_{2}\|z_{1}\|^{2},
\end{align}
we obtain
\begin{align}
\dot{\mathcal{V}}\leq-\sum_{i=1}^{n}(\zb_{i}^{T}W_{i}\zb_{i})+2\gamma\|P\|_{2}\|z_{1}\|^{2},
\end{align}
where $\zb_{i}=[\|z_{1}\|,\|z_{2_{i}}\|]^{T}\in\Re^{2}$ and
\begin{align}
W_i=\begin{bmatrix}
\frac{\lambda_{\min}(Q)-2c_{3}(B_{1_i}+\sigma K_{z_{m}})\alpha}{n}&-\frac{c_{3}B_{1_i}}{2}\\
-\frac{c_{3}(B_{1_i}+\sigma K_{z_{m}})}{2}&\lambda_{m}(W_{2_i})
\end{bmatrix}.
\end{align}
By using $\|z_{1}\|\leq\|\zb_i\|$, we obtain
\begin{align}
\dot{\mathcal{V}}\leq -\sum_{i=1}^{n}(\lambda_{\min}(W_i)-\frac{2\gamma\|P\|_{2}}{n})\|\zb_{i}\|^{2}.
\end{align}
Choosing $\gamma<n(\lambda_{\min}(W_i))/2\|P\|_{2}$, and
\begin{align}
\lambda_{m}(W_{2_i})>\frac{n\|\frac{c_{3}(B_{1_i}+\sigma K_{z_{m}})}{2}\|^2}{\lambda_{\min}(Q)-2c_{3}K_{\max}\alpha},
\end{align}
ensures that $\dot{\mathcal{V}}$ is negative semi-definite. This implies that the zero equilibrium of tracking errors is stable in the sense of Lyapunov and $\mathcal{V}$ is non-increasing. Therefore all of error variables $z_{1}$, $z_{2_i}$ and integral control terms $e_{I_i}$, $e_{\xb}$ are uniformly bounded. Also, from Lasalle-Yoshizawa theorem~\cite[Thm 3.4]{Kha96}, we have $\zb_i\rightarrow 0$ as $t\rightarrow \infty$.
\end{document}